\documentclass[12pt]{article}
\usepackage{latexsym}
\usepackage{indentfirst}
\usepackage{amssymb,amsfonts,amsmath,amsthm}
\usepackage{graphicx}
\usepackage{mathrsfs}
\usepackage{array}
\usepackage{color}
\usepackage{epstopdf}
\usepackage[all]{xy}
\usepackage{hyperref}
\usepackage{url}

\usepackage[a4paper,bindingoffset=0.2in,           left=0.5in,right=0.51in,top=0.6in,bottom=0.5in,            footskip=.25in]{geometry}

\usepackage{relsize}
\usepackage[bbgreekl]{mathbbol}
\DeclareSymbolFontAlphabet{\mathbb}{AMSb} %to ensure that the meaning of \mathbb does not change
\DeclareSymbolFontAlphabet{\mathbbl}{bbold}

\usepackage{tikz-cd}

\newtheorem{thm}{Theorem}[section]
\newtheorem{prop}[thm]{Proposition}
\newtheorem{lem}[thm]{Lemma}

\newtheorem{conv}[thm]{Convention}
\newtheorem{const}[thm]{Construction}

\newtheorem{assumption}[thm]{Assumption}
\newtheorem{cor}[thm]{Corollary}

\theoremstyle{remark}

\newtheorem{rmk}[thm]{Remark}
\newtheorem{dfn}[thm]{Definition}

\numberwithin{equation}{section}

\newcommand{\frakS}{{\mathfrak S}}

\newcommand{\frakU}{{\mathfrak U}}
\newcommand{\frakV}{{\mathfrak V}}

\newcommand{\frakX}{{\mathfrak X}}
\newcommand{\frakY}{{\mathfrak Y}}

\newcommand{\bA}{{\mathbb A}}

\newcommand{\bD}{{\mathbb D}}

\newcommand{\bF}{{\mathbb F}}
\newcommand{\bG}{{\mathbb G}}

\newcommand{\bM}{{\mathbb M}}
\newcommand{\bN}{{\mathbb N}}

\newcommand{\bP}{{\mathbb P}}

\newcommand{\bZ}{{\mathbb Z}}

\newcommand{\calC}{{\mathcal C}}

\newcommand{\calI}{{\mathcal I}}
\newcommand{\calJ}{{\mathcal J}}

\newcommand{\calM}{{\mathcal M}}

\newcommand{\calO}{{\mathcal O}}
\newcommand{\calP}{{\mathcal P}}

\newcommand{\rD}{{\mathrm D}}
\newcommand{\rd}{{\mathrm{d}}}

\newcommand{\rF}{{\mathrm F}}

\newcommand{\rH}{{\mathrm H}}

\newcommand{\rP}{{\mathrm P}}

\newcommand{\rR}{{\mathrm R}}

\newcommand{\rT}{{\mathrm T}}

\newcommand{\rW}{{\mathrm W}}

\newcommand{\Fp}{{\bF_p}}

\newcommand{\colim}{{\mathrm{colim}}}       %colimit
\newcommand{\End}{{\mathrm{End}}}           %End-functor
\newcommand{\Hom}{{\mathrm{Hom}}}           %Hom-functor
\newcommand{\HIG}{{\mathrm{HIG}}}           %Higgs module
\newcommand{\id}{{\mathrm{id}}}             %Identity
\newcommand{\Ima}{{\mathrm{Im}}}            %image
\newcommand{\Ker}{{\mathrm{Ker}}}           %kernal
\newcommand{\MIC}{{\mathrm{MIC}}}
\newcommand{\nil}{\mathrm{nil}}
\newcommand{\pr}{{\mathrm{pr}}}             %Projection
\newcommand{\Sh}{{\mathrm{Sh}}}             %Shimura varieties
\newcommand{\Spf}{{\mathrm{Spf}}}           %formal affine schemes
\newcommand{\Spec}{{\mathrm{Spec}}}         %affine schemes
\newcommand{\Strat}{{\mathrm{Strat}}}
\newcommand{\Tot}{{\mathrm{Tot}}}           %Total complex
\newcommand{\Vect}{{\mathrm{Vect}}}         %Vector bundle

\newcommand{\cris}{{\mathrm{cris}}}         % cristalline
\newcommand{\et}{{\mathrm{\acute{e}t}}}    % etale
\newcommand{\pd}{{\mathrm{pd}}}
\usepackage{relsize}
\usepackage[bbgreekl]{mathbbol}
\DeclareSymbolFontAlphabet{\mathbb}{AMSb} %to ensure that the meaning of \mathbb does not change
\DeclareSymbolFontAlphabet{\mathbbl}{bbold}
\newcommand{\Prism}{{\mathlarger{\mathbbl{\Delta}}}} % Prism

\newcommand{\ya}{{\rangle}}
\newcommand{\za}{{\langle}}

\begin{document}

\title{Prismatic crystals for smooth schemes in characteristic $p$ with Frobenius lifting mod $p^2$}

\author{Yupeng Wang\footnote{\textbf{Address:} Shanghai Center for Mathematical Sciences, Fudan University, Shanghai, 200438, China. \textbf{Email:} wangyupeng@fudan.edu.cn}}

\date{}
\maketitle

\begin{abstract}
  Let $(A,(p))$ be a crystalline prism with $A_n = A/p^{n+1}A$ for all $n\geq 0$.
  Let $\frakX_0$ be a smooth scheme over $A_0$. Suppose that $\frakX_0$ admits a smooth lifting $\frakX_n$ over $A_n$ and the absolute Frobenius $\rF_{\frakX_0}:\frakX_0\to \frakX_0$ admits a lifting over $A_1$. Then we show that there is an equivalence between the category of the prismatic crystals of truncation $n$ on $(\frakX_0/A)_{\Prism}$ and the category of $p$-connections over $\frakX_n$, which is compatible with cohomologies. This generalises a previous work of Ogus. %\cite{Ogu22}.
  We also give some remarks on trivializing the Hodge--Tate gerbe $\pi_{\frakX_0}^{\rm HT}:\frakX_0^{\rm HT}\to\frakX_0$ introduced by Bhatt--Lurie. %\cite{BL22b}.


\end{abstract}

%\keywords{\textbf{Keywords:} prismatic crystals, $p$-connections, Frobenius lifting}

%\keywords{\textbf{MSC2020:} 14F30}

\setcounter{tocdepth}{1}
\tableofcontents
\section{Introduction}
  Let $\bN_*:=\bN\cup\{\infty\}$. For a $p$-adically complete ring (or sheaf of rings) $S$, we always let $S/p^n:= S$ for $n = \infty$. 
  Fix a crystalline prism $(A,(p))$, let $\phi_A$ be the induced Frobenius map of $A$ and put $A_n = A/p^{n+1}A$ for any $n\in\bN_*$. In this paper, $\frakX_0$ will always denote a smooth scheme over $\Spec(A_0)$ of relative dimension $d$. We will study prismatic crystals on the prismatic site $(\frakX_0/A)_{\Prism}$ introduced by Bhatt--Scholze \cite{BS22}, and the readers are referred to loc.cit. for the foundations of prismatic theory.

  Our main result is the following theorem:
  \begin{thm}[Main theorem]\label{thm:main}
      Let $n\geq 1$ belong to $\bN_*$. Assume that $\frakX_0$ lifts to a smooth formal scheme $\frakX_n$ over $A_n$ and that the absolute Frobenius $\rF_{\frakX_0}:\frakX_0\to\frakX_0$ lifts to a morphism $\rF_{\frakX_1}:\frakX_1\to\frakX_1$ over $\phi_A:A_1\to A_1$. Let $\nu:(\frakX_0/A)_{\Prism}\to \frakX_{0,\et}\simeq \frakX_{n,\et}$ be the natural morphism of sites ( cf. \cite[Const. 4.4]{BS22}). Then we have that
      \begin{enumerate}
          \item[(1)] there exists an equivalence
      \begin{equation}\label{equ:prismatic vs p-connection}
          \Vect((\frakX_0/A)_{\Prism},\calO_{\Prism,n})\simeq \MIC_p^{\rm tqn}(\frakX_n)
      \end{equation}
       between the category of $\calO_{\Prism,n}$-crystals on $(\frakX_0/A)_{\Prism}$ (cf. Definition \ref{dfn:crystal}) and the category of topologically quasi-nilpotent flat $p$-connections on $\frakX_{n,\et}$ (cf. Definition \ref{dfn:p-connection}), which preserves ranks, tensor products and dualities;

       \item[(2)] and that for any $\bM\in \Vect((\frakX_0/A)_{\Prism},\calO_{\Prism,n})$ with the corresponding $(\calM,\nabla_{\calM})\in \MIC_p^{\rm tqn}(\frakX_n)$ via the equivalence in item (1), there exists a quasi-isomorphism of complexes of sheaves of $A_n$-modules on $\frakX_{n,\et}$
       \[\rR\nu_*\bM\simeq \rD\rR(\calM,\nabla_{\calM}),\]
       functorial in $\bM$. In particular, we have a quasi-isomorphism of complexes of $A_n$-modules
       \[\rR\Gamma((\frakX_0/A)_{\Prism},\bM)\simeq \rR\Gamma(\frakX_{n,\et},\rD\rR(\calM,\nabla_{\calM})).\]
      \end{enumerate}
  \end{thm}
  
  Note that when $ n = 0$, the theorem establishes an equivalence
  \begin{equation}\label{equ:HT vs Higgs}
          \Vect((\frakX_0/A)_{\Prism},\overline \calO_{\Prism})\simeq\HIG^{\nil}(\frakX_0).
  \end{equation}
  between the category $\Vect((\frakX_0/A)_{\Prism},\overline 
  \calO_{\Prism})$ of Hodge--Tate crystals and the category $\HIG^{\nil}(\frakX_0)$ of nilpotent Higgs bundles on $\frakX_{0,\et}$ which is compatible with cohomologies.

  The item (1) of Theorem \ref{thm:main} will be proved in \S\ref{ssec:proof-I} while the item (2) will be proved in \S\ref{ssec:proof-II}. Here, we exhibit our strategy: First, we prove the result when $\frakX_0$ is affine. In this case, $\frakX_0$ admits a ``prismatic lifting''. In other words, the absolute Frobenius $\rF_{\frakX_0}$ of $\frakX_0$ admits a lifting over $A$. So we can construct the desired equivalence and quasi-isomorphism as in \cite{Tia23}. More precisely, in this situation, we are able to give an explicit description of the cosimplicial prism associated to the \v Cech nerve of the given ``prismatic lifting'', which consists of several self-products of this ``prismatic lifting'', and then obtain the desired equivalence (cf. Proposition \ref{prop:local equivalence}) and quasi-isomorphism (cf. Proposition \ref{prop:cohomology comparison}) by studying the stratification with respect to this cosimplicial prism. Next, we show that this local construction depends only on the lifting of $\rF_{\frakX_0}$ over $A_1$. See Proposition \ref{prop:compare local equivalence} and Proposition \ref{prop:local compatibility of cohomology} for precise statements for the local equivalence and quasi-isomorphism, respectively. Finally, for general $\frakX_0$, by choosing an affine covering, we can reduce Theorem \ref{thm:main} to the affine case by using the given Frobenius lifting $\rF_{\frakX_1}$ over $A_1$ (cf. \S\ref{ssec:proof-I} and \S\ref{ssec:proof-II} for details). The key ingredient is an observation of Ogus \cite[Thm. 2.3.13(1)]{Ogu22} (cf. Proposition \ref{prop:key} for the precise statement), which allows us to give an explicit description of the prismatic envelope of ``local prismatic liftings'' of $\frakX_0$ with different $\delta$-structures. This helps us to prove that the local construction only depends on the Frobenius lifting over $A_1$ (cf. \S\ref{ssec:compare local equivalence} and \S\ref{ssec:compatibility of cohomology}).

  \begin{rmk}\label{rmk:ogus}
      \begin{enumerate}
          \item[(1)] When $\frakX_0$ is affine (and thus always admits a lifting over $A$, and so does $\rF_{\frakX_0}$), Theorem \ref{thm:main} for $n = 0$ was obtained by Tian \cite[Thm. 5.12]{Tia23}.

          \item[(2)] When $\rF_{\frakX_0}:\frakX_0\to\frakX_0$ admits a lifting $\rF_{\frakX}:\frakX\to\frakX$ over $\phi_A:A\to A$, the result was obtained by Ogus as \cite[Cor. 6.2.4 and Lem. 6.3.7]{Ogu22}.

          \item[(3)] In loc.cit., Ogus can deal with a more general case where $\frakX$ is a closed subscheme of a smooth scheme $\frakY_0$, and the absolute Frobenius $\rF_{\frakY_0}:\frakY_0\to\frakY_0$ admits a lifting $\rF_{\frakY}:\frakY\to\frakY$ over $\phi_A:A\to A$. Roughly, in this case, one can construct a prismatic envelope $\Prism_{\frakX}(\frakY)$ which is affine over $\frakY$ but not of finite type. Then the category $\Vect((\frakX_0/A)_{\Prism},\calO_{\Prism})$ is equivalent to the category $\MIC\rP(\frakX/\frakY/\Spf(A))$ of certain $p$-connections on $\Prism_{\frakX}(\frakY)$ \cite[Thm. 6.9]{Ogu22}. It is worth mentioning that the results of Ogus were generalized by Jiahong Yu \cite{Yu} to the case where $\rF_{\frakY_0}$ admits a lifting $\rF_{\frakY_1}:\frakY_1\to\frakY_1$ over $A_1$ merely.

          \item[(4)] When $n=1$, Theorem \ref{thm:main} may be well-known to experts by using non-abelian Hodge theory in characteristic $p$. But our proof is more direct by using prismatic theory merely and is compatible with the constructions in the literature. See Remark \ref{rmk:non-abelian Hodge} for more discussion.
      \end{enumerate}
  \end{rmk}
  \begin{rmk}\label{rmk:Bhatt-Lurie}
      Let $\pi^{\rm HT}_{\frakX_0}:\frakX_0^{\rm HT}\to\frakX_0$ be the structure morphism of the Hodge--Tate stack $\frakX^{\rm HT}_{0}$ of $\frakX$ \cite[Const. 3.7]{BL22b}, which is actually a gerbe banded by the flat group scheme $\rT_{\frakX_0/A_0}^{\sharp}$ over $\frakX_0$ \cite[Prop. 5.12]{BL22b}. Any trivialization of the gerbe $\pi_{\frakX}^{\rm HT}$ will induce an equivalence of categories (\ref{equ:HT vs Higgs}).
      This is the case when $\rF_{\frakX_0}$ admits a lifting $\rF_{\frakX}:\frakX\to\frakX$ over $\phi_A:A\to A$  \cite[Rem. 5.13(1)]{BL22b} (and as a consequence, we have the equivalence (\ref{equ:HT vs Higgs}) even in the derived sense). Let $\alpha_{\rm HT}(\frakX_0)\in \rH^2(\frakX_0,\rT_{\frakX_0/A_0}^{\sharp})$ be the obstruction class for trivializing $\pi^{\rm HT}_{\frakX_0}$. It induces a class $\alpha_{\rF}(\frakX_0)\in\rH^1(\frakX_0,\rF_{\frakX_0*}\calO_{\frakX_0}/\calO_{\frakX_0})$, which is conjectured to be the obstruction class for lifting $\rF_{\frakX_0}$ over $\phi_A:A_1\to A_1$ \cite[Conj. 5.14(a)]{BL22b}. This conjecture was recently proved by Jiahong Yu \cite{Yu2}. In particular, that the gerbe $\pi_{\frakX}^{\rm HT}$ is trivial implies that $\rF_{\frakX_0}$ \emph{must} admit a lifting $\rF_{\frakX_1}:\frakX_1\to\frakX_1$ over $\phi_A:A_1\to A_1$.
      So Theorem \ref{thm:main} is \emph{optimal} in this sense.
  \end{rmk}
    
      As we always have the equivalence (\ref{equ:HT vs Higgs}) when $\alpha_{\rF}(\frakX_0) = 0$ (or equivalently $\rF_{\frakX_0}$ admits a lifting $\rF_{\frakX_1}$ over $A_1$), one may ask whether we always have $\alpha_{\rm HT}(\frakX_0) = 0$ as long as $\alpha_{\rF}(\frakX_0) = 0$?
      So one may ask whether one can use the same approach as in the proof of \cite[Prop. 5.12]{BL22b} to show that the gerbe $\pi_{\frakX_0}^{\rm HT}$ admits a section once $\rF_{\frakX_0}$ admits a lifting $\rF_{\frakX_n}$ over for $A_n$ some $1\leq n<+\infty$? We will provide some calculations towards this question in \S\ref{sec:remarks on gerbes}.

    We also want to give a remark on the relationship between Theorem \ref{thm:main} and the non-abelian Hodge theory in characteristic $p$ (cf. \cite{OV,GLSQ10,Shiho,Xu19,LSZ19}, etc.).

\begin{rmk}\label{rmk:non-abelian Hodge}
    For any $*\in\{0,1\}$, let $\frakX_*^{\prime}:=\frakX_*\times_{\Spf A,\phi_A}\Spf(A)$. 
    It is well-known that there is an equivalence
    \[\Vect((\frakX_0/A)_{\cris},\overline \calO_{\frakX_0/A})\simeq \MIC^{\rm tqn}(\frakX_0)\]
    between the category $\Vect((\frakX_0/A)_{\cris},\overline \calO_{\frakX_0/A})$ of $\overline \calO_{\frakX_0/A}$-crystals (cf. Definition \ref{dfn:crystal}) on the crystalline site $(\frakX_0/A)_{\cris}$ and the category $\MIC^{\rm tqn}(\frakX_0)$ of (topologically) quasi-nilpotent connections on $\frakX_0$,
    which is compatible with cohomologies \cite[Thm. 6.6 and Thm. 7.1]{Ber-Ogus} (at least when $\frakX_0$ admits a lifting $\frakX$ over $A$). Here, $\overline \calO_{\frakX_0/A}:=\calO_{\frakX_0/A}/p$ is the quotient of the structure sheaf $\calO_{\frakX_0/A}$ on $(\frakX_0/A)_{\cris}$ by $p$.
    As we always have the equivalence of categories 
    \[\Vect((\frakX_0/A)_{\cris},\overline \calO_{\frakX_0/A})\simeq \Vect((\frakX_0^{\prime}/A)_{\Prism},\overline \calO_{\Prism})\]
    by \cite[Exam. 4.7]{BS23}, Theorem \ref{thm:main} provides an equivalence of categories
    \begin{equation}\label{equ:connection = p-connection}
       \psi_{\rm new}: \HIG^{\nil}(\frakX_0^{\prime})\simeq\MIC^{\rm tqn}(\frakX_0)
    \end{equation}
    which is compatible with cohomologies and is defined as the composite of the following equivalences
    \begin{equation}\label{equ:connection = p-connection-II}
        \HIG^{\nil}(\frakX_0^{\prime})\xrightarrow{\simeq}\Vect((\frakX_0^{\prime}/A)_{\Prism},\overline \calO_{\Prism})\xrightarrow{\simeq}\Vect((\frakX_0/A)_{\cris},\overline \calO_{\frakX_0/A})\xrightarrow{\simeq}\MIC^{\rm tqn}(\frakX_0).
    \end{equation}
    We emphasize that the equivalence (\ref{equ:connection = p-connection}) coincides with the equivalence 
    \[\psi_{\rm old}:\HIG^{\nil}(\frakX_0^{\prime})\xrightarrow{\simeq}\MIC^{\rm tqn}(\frakX_0)\]
    that already appeared in the literature (cf. \cite[Thm. 2.11(2) and Cor. 2.28]{OV}, \cite[Cor. 3.2]{Shiho}, \cite[Thm. 5.8]{GLSQ10}, etc.).
    To see this, noting that the following presheaves
    \[\frakY_0\in\frakX_{0,\et}\mapsto \HIG^{\nil}(\frakY_0^{\prime}) \text{ and }\frakY_0\in\frakX_{0,\et}\mapsto \MIC^{\rm tqn}(\frakY_0)\]
    are both \'etale sheaves, we may assume that $\frakX_0 = \Spec(R_0)$ is small affine in the sense that there exists an \'etale morphism 
    \[\frakX_0\to\Spec(A_0[T_1^{\pm 1},\dots,T_d^{\pm 1}]).\]
    In this situation, one can give a very explicit description of each equivalence appearing in \eqref{equ:connection = p-connection-II} by its construction, and thus get an explicit description of $\psi_{\rm new}$. Recall that $\psi_{\rm old}$ also admits an explicit description in this setting (cf. \cite[Rem. 1.12]{Shiho}). Then we can directly check $\psi_{\rm new} = \psi_{\rm old}$ by using their explicit descriptions.
    We also point out that the equivalence (\ref{equ:connection = p-connection}) admits a generalization (with more restrictions) where we do \emph{not} need a lifting of $\rF_{\frakX_0}$ any more (cf. \cite{OV,Xu19,LSZ19}, etc.). Thus, it is still worth considering the analogue of (\ref{equ:connection = p-connection}) for certain prismatic crystals on $(\frakX^{\prime}_0/A)_{\Prism}$ without any additional assumption on lifting $\rF_{\frakX_0}$.
\end{rmk}

    Finally, it is worth pointing out that the prismatic theory in the crystalline case also applies to the algebraic geometry in characteristic $p$. For example, using prismatic theory, Petrov \cite{Pet23} studied the (non-)decomposibility of de Rham complexes for schemes in characteristic $p$, which generalized the previous work of Deligne--Illusie \cite{DI}. In particular, in loc.cit., Petrov showed that when $A_0 = k$ is a perfect field, the de Rham complex of a smooth variety $\frakX_0$ over $k$ is formal if and only if $(\frakX_0,\rF_{\frakX_0})$ admits a lifting over $A_1 = \rW_1(k)$ (See also \cite[Conj. 5.14(c)]{BL22b} for a relevant conjecture). 
\subsection*{Organizations}
  The paper is organized in the following way: In \S\ref{sec:notations}, we introduce some basic definitions. In \S\ref{sec:key proposition},we calculate the prismatic envelope of different ``local prismatic liftings'' corresponding to different $\delta$-structures (cf. Proposition \ref{prop:key}), which is the key ingredient to prove Theorem \ref{thm:main} when $(\frakX_0,\rF_{\frakX_0})$ lifts to $A_1$ merely.
  We first show the main theorem when $\frakX_0$ is affine in \S\ref{sec:local construction}, and then prove Theorem \ref{thm:main} in \S\ref{sec:proof of main theorem} by using Proposition \ref{prop:key} to compare and glue the local constructions we obtained in the affine case. At the end of this paper, in \S\ref{sec:remarks on gerbes}, we give some remarks on trivializing the Hodge--Tate gerbe $\pi_{\frakX_0}^{\rm HT}$.
\subsection*{Acknowledgement}
  The author would like to thank Jingbang Guo, Fei Liu, Mao Sheng, Yichao Tian, Jinbang Yang and Jiahong Yu for valuable conversations. The author also wants to thank Heng Du and Yu Min for valuable comments and suggestions and thank Sasha Petrov for pointing out a mistake in the early version of this paper. Special thank goes to Jiahong Yu for careful reading and pointing out an error in an early version of the draft. Parts of the work were done during the author’s visit to Southern university of Science and Technology, and he would like to thank Hui Gao for invitation and the institute for comfortable research environment. The author completed this work when he was a postdoc at Beijing International Center for Mathematical Reaserch in Peking University. He would like to thank the institute and his mentor there, Professor Ruochuan Liu, for providing the position to him. Finally, the authorI would like to thank the anonymous referee for his or her extremely careful reading and very helpful suggestions. Without his or her help, the paper cannot be exhibited in this way. The author is partially supported by CAS Project for Young Scientists in Basic Research, Grant No. YSBR-032, and by New Cornerstone Science Foundation through Ruochuan Liu.

\section{Convention and Notation}\label{sec:notations}
  We introduce some notations and definitions in this section. Throughout this paper, we always fix a crystalline prism $(A,(p))$ and let $A_n:=A/p^{n+1}$ for any $n\in\bN_*$. For an $\Fp$-algebra $R_0$, let $\phi_{R_0}:R_0\to R_0$ be the absolute Frobenius endomorphism on $R_0$. 

  In this paper, $\frakX_0$ always denotes a smooth and separated scheme over $A_0$, and thus the fiber product of finitely many affine schemes over $\frakX_0$ is always affine.
  
  Let $\rW(-)$ denote the functor for the ring of the $p$-typical Witt vectors and $\rW_n(-)$ be the truncation of $\rW(-)$ of length $(n+1)$; for example, $\rW_0 = \bG_a$.

  For any $\delta$-ring $A$, let $\delta^n$ be the $n$-fold self-composite of $\delta$; that is, we have $\delta^0 = \id$ and $\delta^{n+1} = \delta\circ\delta^n$.

  Fix an integer $d\geq 0$ (which usually represents the dimension of a scheme in this paper). For any capital letter, for example $X$, we denote by $\underline X$ the sequence $X_1,\dots,X_d$. For two capital letters $X$ and $Y$, denote by $\underline X-\underline Y$ the sequence $X_1-Y_1,\dots,X_d-Y_d$. Similarly, $\frac{\underline X-\underline Y}{p}$ stands for $\frac{X_1-Y_1}{p},\dots,\frac{X_d-Y_d}{p}$, etc.. For any $i\geq 1$, $\underline X_{i}$ always stands for $X_{i,1},\dots,X_{i,d}$.

  Let $Y$ be an element in a ring $R$ which admits arbitrary pd-powers. Let $Y^{[n]}$ be the $n$-th pd-power of $Y$; for example, when $R$ is $p$-torsion free, $Y^{[n]} = \frac{Y^n}{n!}$ in $R[1/p]$. If $Y_1,\dots,Y_d\in R$ admit arbitrary pd-powers, for any $\underline n = (n_1,\dots,n_d)\in\bN^d$, we write $\underline Y^{[\underline n]}:= \prod_{i=1}^dY_i^{[n_i]}$. 

  Let $\underline 1_i$ be the generator of $i$-th component of the (additive) monoid $\bN^d$ and let $\underline 0= (0,0,\dots,0)$. For any $\underline n = (n_1,\dots,n_d)$, set $|\underline n| = n_1+\cdots+n_d$.

  For any set $S$ and elements $s_1,s_2\in S$, let $\delta_{s_1,s_2}$ be the Kronecker's $\delta$-function; that is, $\delta_{s_1,s_2} = 0$ when $s_1\neq s_2$ and $\delta_{s_1,s_2} = 1$ when $s_1 = s_2$.

  The main objects we want to study are prismatic crystals. Let us give the definition.
  \begin{dfn}\label{dfn:crystal}
      Let $\calC$ be a site and $\bA$ be a sheaf of rings on $\calC$. For any integer $r\geq 0$, by an \emph{$\bA$-crystal in vector bundles of rank $r$} on $\calC$, we mean a sheaf $\bM$ of $\bA$-modules satisfying the following conditions: 
      \begin{enumerate}
          \item[(1)] For any object $C\in\calC$, the evaluation of $\bM$ at $C$, denoted by $\bM(C)$, is a finite projective $\bA(C)$-module of rank $r$.

          \item[(2)] For any arrow $C_1\to C_2$ in $\calC$, the base change homomorphism 
          \[\bM(C_2)\otimes_{\bA(C_2)}\bA(C_1)\xrightarrow{\cong}\bM(C_1)\]
          is an isomorphism of $\bA(C_1)$-modules.
      \end{enumerate}
      The morphisms of $\bA$-crystals are morphisms of sheaves of $\bA$-modules.
      We denote by $\Vect(\calC,\bA)$ the category of $\bA$-crystals in vector bundles on $\calC$. When context is clear, we just call an $\bA$-crystal in vector bundles an \emph{$\bA$-crystal} for short. 

      In this paper, we often assume the site $\calC$ is the prismatic site $(\frakX_0/A)_{\Prism}$ (cf. \cite[Def. 4.1]{BS22}) while the sheaf $\bA$ is $\calO_{\Prism,n} = \calO_{\Prism}/p^{n+1}$ for some $n\in\bN_*$. Here, $\calO_{\Prism,n} = \calO_{\Prism}$ for $n = \infty$. In this case, we call $\calO_{\Prism,n}$-crystals the \emph{prismatic crystals of truncation $n$}. In particular, when $n= 0$, we have $\overline \calO_{\Prism} = \calO_{\Prism,0}$ and call $\overline \calO_{\Prism}$-crystals the \emph{Hodge--Tate crystals}.
  \end{dfn}
  To study prismatic crystals, we need to use the language of stratifications.
  \begin{dfn}[Stratification]\label{dfn:stratification}
      Let $S^{\bullet}$ be a cosimplicial ring with face maps $p_i^n$'s and degeneracy maps $\sigma_i^n$'s; that is, $p_i^n:S^n\to S^{n+1}$ is induced by the injection $[n]\to [n+1]\setminus\{i\}$ and that $\sigma_i^n:S^{n+1}\to S^n$ is induced by the surjection $[n+1]\to[n]$ satisfying $(\sigma^{n}_i)^{-1}(i) = \{i,i+1\}$. Let $q^n_i:S^0\to S^n$ be the morphism induced by the inclusion $[0]\xrightarrow{0\mapsto i}[n]$ of simplices.
     We often write $p_i$ (resp. $\sigma_i$, resp. $q_i$) instead of $p_i^n$ (resp. $\sigma_i^n$, resp. $q_i^n$) for short when its source and target are clear. 
     
     By a \emph{stratification} with respect to $S^{\bullet}$, we mean a finite projective $S^0$-module together with an $S^1$-linear isomorphism \[\varepsilon:M\otimes_{S^0,p_0}S^1\xrightarrow{\cong}M\otimes_{S^0,p_1}S^1\]
     satisfying the following conditions:
     \begin{enumerate}
         \item[(1)] $p_2^*(\varepsilon)\circ p_0^*(\varepsilon) = p_1^*(\varepsilon):M\otimes_{S^0,q_2}S^2\to M\otimes_{S^0,q_0}S^2$.
         
         \item[(2)] $\sigma_0^*(\varepsilon) = \id_M:M\to M$.
     \end{enumerate}
     For two stratifications $(M,\epsilon_M)$ and $(N,\epsilon_N)$, we call an $S^0$-linear morphism $f:M\to N$ a \emph{morphism of stratifications} if it fits into the following commutative diagram of morphisms of $S^1$-modules: 
     \[\xymatrix@C=0.5cm{
       M\otimes_{S^0,p_0}S^1\ar[d]^{f\otimes\id}\ar[rr]^{\epsilon_M}&&M\otimes_{S^0,p_1}S^1\ar[d]^{f\otimes\id}\\
       N\otimes_{S^0,p_0}S^1\ar[rr]^{\epsilon_N}&&N\otimes_{S^0,p_1}S^1.
     }\]
     We denote by $\Strat(S^{\bullet})$ the category of stratifications with respect to $S^{\bullet}$.
  \end{dfn}

  We remark that in Definition \ref{dfn:stratification}, condition (1) is usually called the \emph{cocycle condition}, and condition (2) can be deduced from condition (1) for some special cosimplicial ring $S^{\bullet}$. For example, all cosimplicial rings we shall meet in this paper are in this special case; see \cite[Rem. 4.6]{Tia23}.

\section{The key proposition}\label{sec:key proposition}
  Let $(A,(p))$ be a crystalline prism with $A_0 = A/p$.
  Throughout this section, let $\frakX_0 = \Spec(R_0)$ be an affine geometrically connected smooth scheme of dimension $d$ over $A_0$. We always assume that there exists a \emph{chart} on $\frakX_0$; that is, there is an \'etale morphism
  \[\Box: A_0[\underline T^{\pm 1}]:=A_0[T_1^{\pm 1},\dots,T_d^{\pm 1}]\to R_0,\]
  and fix such a chart $\Box$.
  By the smoothness of $\frakX_0$, it always admits a smooth lifting $\frakX = \Spf(R)$ over $A$, which is \emph{unique} up to (not necessarily unique) isomorphisms. In particular, $R$ is $p$-adically complete, and $R_0 = R/pR$. Let $A\za\underline T^{\pm 1}\ya$ be the $p$-adic completion of 
  \[A[\underline T^{\pm 1}] = A[T_1^{\pm 1},\dots,T_d^{\pm 1}],\]
  and thus it is a ($p$-completely) smooth lifting of $A_0[\underline T^{\pm 1}]$ over $A$. By the \'etaleness of $\Box$, there is a morphism of $A$-algebras $A\za \underline T^{\pm 1}\ya\to R$ lifting $\Box$, which is \emph{unique} up to a \emph{unique} isomorphism and makes $R$ a $p$-completely \'etale $A\za T^{\pm 1}\ya$-algebra. By abuse of notation, we still denote this lifting by $\Box:A\za\underline T^{\pm 1}\ya\to R$, and call $T_i$'s the \emph{coordinates} on $R$ induced by $\Box$.
  Let $R\widehat \otimes_{A}R$ dneote the $p$-adic completion of $R\otimes_{A}R$, and let $\underline T$ and $\underline S$ denote the respective coordinates on the first and second components, induced by the given chart $\Box$ on $R$. Let $J = \Ker(\Delta)$ be the kernel of the diagonal map $\Delta:R\widehat \otimes_AR\to R$.

  \begin{rmk}\label{rmk: On Ker(Delta)}
      We claim $J$ is generated by $T_1-S_1,\dots,T_d-S_d$ together with an element $e\in R\widehat \otimes_AR$ satisfying that $e^2-e$ belongs to the ideal of $R\widehat \otimes_AR$ generated by $T_1-S_1,\dots,T_d-S_d$. Indeed, note that the diagonal map $\Delta:R\widehat \otimes_AR\to R$ factors as 
      \[R\widehat \otimes_AR\to R\widehat \otimes_{A\za\underline T^{\pm 1}\ya}R\to R.\]
      As $R$ is $p$-completely \'etale over $A\za\underline T^{\pm 1}\ya$, the second map $R\widehat \otimes_{A\za\underline T^{\pm 1}\ya}R\to R$ above admits a section $s: R\to R\widehat \otimes_{A\za\underline T^{\pm 1}\ya}R$. Let $e$ be a pre-image of $1-s(1)\in R\widehat \otimes_{A\za\underline T^{\pm 1}\ya}R$ via the surjection $R\widehat \otimes_AR\to R\widehat \otimes_{A\za\underline T^{\pm 1}\ya}R$. As $1-s(1)$ is an idemponent of $R\widehat \otimes_{A\za\underline T^{\pm 1}\ya}R$, we see that 
      \[e^2-e\in\Ker(R\widehat \otimes_{A}R\to R\widehat \otimes_{A\za\underline T^{\pm 1}\ya}R),\]
      which is exactly the ideal generated by $T_1-S_1,\dots,T_d-S_d$. By abuse of notation, we will not distinguish $e$ with its image $1-s(1)$ in $R\widehat \otimes_{A\za\underline T^{\pm 1}\ya}R$.
  \end{rmk}

  One can equip $R$ with a $\delta$-structure $\delta:R\to R$ compatible with that on $A$. For example, we may equip $A\za\underline T^{\pm 1}\ya$ with the $\delta$-structure which is compatible with that on $A$ such that $\delta(T_1) = \dots = \delta(T_d) = 0$, and then extend this $\delta$-structure to $R$ by applying \cite[Lem. 2.18]{BS22}. Now, assume $R$ is equipped with a $\delta$-structure compatible with that on $A$ and let $\varphi$ be the induced Frobenius map on $R$. Then the diagonal morphism $\Delta:R\widehat \otimes_{A}R\to R$ is a morphism of $\delta$-rings. In particular, $\Delta$ is $\varphi$-equivariant. As $R$ is $p$-completely smooth over $A$, it is $p$-torsion free and thus the given $\delta$-structure on $R$ makes $(R,(p))$ a crystalline prism (cf. \cite[Exam. 3.3]{BS22}), which is indeed an object in $(R_0/A)_{\Prism}$ as $R/pR = R_0$. One of our purposes in this section is to study the self-coproducts of $(R,(p))$'s in $(R_0/A)_{\Prism}$. The next lemma tells us that these coproducts exist.
  \begin{lem}\label{lem:cover}
      Equip $R$ with a $\delta$-structure making $(R,(p))$ a prism in $(R_0/A)_{\Prism}$. Then $(R,(p))$ is a covering of the final object of the topos $\Sh((R_0/A)_{\Prism})$. Moreover, for any $(B,(p))\in (R_0/A)_{\Prism}$, let $I:=\Ker(B\widehat \otimes_AR\to B/p)$ and then the coproduct of $(B,(p))$ and $ (R,(p))$ exists in $(R_0/A)_{\Prism}$ and is given by the prismatic envelope $\big(B\widehat \otimes_AR\big)\{\frac{I}{p}\}^{\wedge}_p$ in the sense of \cite[Prop. 3.13]{BS22}.
  \end{lem}
  \begin{proof}
      The result follows from \cite[Lem. 4.2]{Tia23} and its proof immediately.
  \end{proof}

  Now, we are going to describe the self-coproduct of $(R,(p))$ in $(R_0/A)_{\Prism}$, following the strategy in the proof of \cite[Thm. 2.3.13]{Ogu22}. 
  \begin{const}\label{const:p-dilatation}
      Let $C$ be the $p$-adic completion of the following ring 
      \[R\widehat \otimes_AR[\frac{T_1-S_1}{p},\dots,\frac{T_d-S_d}{p}] = (R\widehat \otimes_AR)[Y_1,\dots,Y_d]/(pY_1-(T_1-S_1),\dots,pY_d-(T_d-S_d)).\]
      As $R\widehat \otimes_AR$ is $p$-torsion free and $p, Y_1,\dots,Y_d,T_1-S_1,\dots,T_d-S_d$ form a regular sequence in
      \[(R\widehat \otimes_AR)[Y_1,\dots,Y_d],\]
      we know that $C$ is also $p$-torsion free, and $p,\frac{T_1-S_1}{p},\dots,\frac{T_d-S_d}{p}$ form a regular sequence in $C$. It is easy to check that
      \[C_0:=C/p = R_0\otimes_{A_0[\underline T]}R_0[Y_1,\dots,Y_d]\]
      is the polynomial algebra over $R_0\otimes_{A_0[\underline T]}R_0$ with free variables $Y_i$'s (which are the images of $\frac{T_i-S_i}{p}$'s).
      The diagonal map $\Delta:R\widehat \otimes_AR\to R$ induces a natural map $C\to R$ sending each $\frac{T_i-S_i}{p}$ to $0$, whose reduction modulo $p$ factors as
      \[C_0 = R_0 \otimes_{A_0[\underline T]}R_0[Y_1,\dots,Y_d]\to D_0:=R_0[Y_1,\dots,Y_d]\to R_0.\]
      Via the section $s$ in Remark \ref{rmk: On Ker(Delta)}, we can regard $D_0$ as an \'etale $C_0$-algebra. Thus the map $C_0\to D_0$ lifts to a \emph{unique} $p$-completely \'etale morphism $C\to D$ with $D/p = D_0$, which admits a section compatible with $s$ in Remark \ref{rmk: On Ker(Delta)}. In particular, the kernel $\Ker(C\to D)$ is generated by an idemponent $e^{\prime}\in C$. Let $e$ be as in Remark \ref{rmk: On Ker(Delta)} and we identify it with its image in $C$ via the canonical map $R\widehat \otimes_AR\to C$. By constructions of $e$ and $e^{\prime}$, we have
      \[e\equiv e^{\prime}\mod pC.\]
      Note that $D$ is a $C$-algebra and thus an $(R\widehat \otimes_AR)$-algebra.
  \end{const}
  \begin{lem}\label{lem:Structure of D}
      Keep notations of Construction \ref{const:p-dilatation}.
      \begin{enumerate}
          \item[(1)] Regard $D$ as an $R$-algebra via the first (or the second) component of $R\widehat \otimes_AR$. Then the morphism of $R$-algebras
          \[R[Y_1,\dots,Y_d]^{\wedge}_p\to D\]
          mapping each $Y_i$ to the image of $\frac{T_i-S_i}{p}\in C$ via the map $C\to D$ is an isomorphism. Here, $R[Y_1,\dots,Y_d]^{\wedge}_p$ is the $p$-completion of the polynomial ring $R[Y_1,\dots,Y_d]$ with free variables $Y_1,\dots,Y_d$ over $R$.

          \item[(2)] The natural map $C[Y]\to C[\frac{e^{\prime}}{p}]$ sending $Y$ to $\frac{e^{\prime}}{p}$ induces an isomorphism 
          \[C[Y]/(pY^2-Y,pY-e^{\prime})\cong C[\frac{e^{\prime}}{p}].\]

          \item[(3)] Let $C[\frac{e^{\prime}}{p}]^{\wedge}_p$ be the $p$-adic completion of $C[\frac{e^{\prime}}{p}]$. Then the morphism $C\to D$ extends canonically to an isomorphism $C[\frac{e^{\prime}}{p}]^{\wedge}_p\xrightarrow{\cong} D$.

          \item[(4)] Let $R\widehat \otimes_AR[\frac{J}{p}]^{\wedge}_p$ be the $p$-adic completion of $R\widehat \otimes_AR[\frac{J}{p}]$, where $J = \Ker(\Delta:R\widehat \otimes_A R\to R)$ as before.
          Then the natural map $R\widehat \otimes_AR\to D$ induces an isomorphism of $(R\widehat \otimes_AR)$-algebras $R\widehat \otimes_AR[\frac{J}{p}]^{\wedge}_p\cong D$. As a consequence, $\frac{e}{p}$ is contained in the $p$-complete ideal of $D$ generated by $(Y_1,\dots,Y_d)$ (cf. Item (1)).
      \end{enumerate}
  \end{lem}
  \begin{proof}
      For Item (1): As both $R[Y_1,\dots,Y_d]^{\wedge}_p$ and $D$ are $p$-complete and $p$-torsion free, we are reduced to checking the isomorphism modulo $p$. But this is clear as $D_0 = D/p = R_0[Y_1,\dots,Y_d]$ by Construction \ref{const:p-dilatation}.

      For Item (2): As $(e^{\prime})^2 = e^{\prime}$ in $C$, we see the map $C[Y]\to C[\frac{e^{\prime}}{p}]$ factors as
      \[C[Y]\to C[Y]/(pY^2-Y,pY-e^{\prime})\to C[\frac{e^{\prime}}{p}].\]
      It remains to prove the above map is an isomorphism. Let $C^{\prime}$ and $C^{\prime\prime}$ be the component of $C$ on which $e^{\prime} = 0$ and $e^{\prime} = 1$, respectively. It suffices to check that the morphism
      \[C^{\prime}[Y]/(pY^2-Y,pY)\to C^{\prime}\] 
      sending $Y$ to $0$ and the morphism
      \[C^{\prime\prime}[Y]/(pY^2-Y,pY-1)\to C^{\prime\prime}[\frac{1}{p}]\]
      sending $Y$ to $\frac{1}{p}$ are isomorphisms. For the first, we can conclude as
      \[C^{\prime}[Y]/(pY^2-Y,pY) = C^{\prime}[Y]/(Y) = C^{\prime}.\]
      For the second, we can conclude as
      \[C^{\prime\prime}[Y]/(pY^2-Y,pY-1) = C^{\prime\prime}[Y]/(pY-1) = C^{\prime\prime}[\frac{1}{p}].\]

      For Item (3): By Construction \ref{const:p-dilatation}, we know that $D = C^{\prime}$. Then the result follows as 
      \[C[\frac{e^{\prime}}{p}] = C^{\prime}\oplus C^{\prime\prime}[\frac{1}{p}].\]

      For Item (4): Recall $J$ is generated by $(T_1-S_1,\dots,T_d-S_d,e)$ (cf. Remark \ref{rmk: On Ker(Delta)}). We have 
      \[R\widehat \otimes_AR[\frac{J}{p}]^{\wedge}_p = R\widehat \otimes_AR[\frac{T_1-S_1}{p},\dots,\frac{T_d-S_d}{p},\frac{e}{p}]^{\wedge}_p = C[\frac{e}{p}]^{\wedge}_p = C[\frac{e^{\prime}}{p}]^{\wedge}_p,\]
      where the last identity follows as $e\equiv e^{\prime}\mod pC$ (cf. Construction \ref{const:p-dilatation}). So we can conclude the desired isomorphism by using Item (3). 
      
      It remains to show that $\frac{e}{p}$ is contained in the $p$-complete ideal $\calI$ of $D$ generated by $Y_i$'s. Let $\calJ$ be the $p$-complete ideal of $D$ generated by $\frac{J}{p}$. By Remark \ref{rmk: On Ker(Delta)}, one can find $a_i\in R\widehat \otimes_AR$ such that
      \[e^2-e = \sum_{i=1}^da_i(T_i-S_i).\]
      From this, we have
      \[\frac{e}{p} = \sum_{i=1}^da_i\frac{T_i-S_i}{p}-p(\frac{e}{p})^2 = \sum_{i=1}^da_iY_i-p(\frac{e}{p})^2\in \calI+p\calJ^2.\]
      Therefore, we have $\calJ\subset \calI+p\calJ^2$. So one can conclude by Nakayama's lemma as $p$ is topologically nilpotent on $D$.
  \end{proof}

 Now, let $\delta_1$ and $\delta_2$ be (not necessarily different) $\delta$-structures on $R$ compatible with that on $A$ and let $\varphi_1$ and $\varphi_2$ be the induced Frobenius endomorphisms of $R$ respectively. So we get two structures of crystalline prism on $R$. To distinguish them, we denote by $(R,(p),\delta_i)$ the prism induced by $\delta_i$. When contexts are clear, we often write $(R_i,(p))$ instead of $(R,(p),\delta_i)$ for short.
 
  \begin{cor}\label{cor:self product for two delta structure}
      The coproduct of $(R_1,(p))$ with $(R_2,(p))$ exists in $(R_0/A)_{\Prism}$, and is given by
      \[R_1\widehat \otimes_{A}R_2\{\frac{J}{p}\}^{\wedge}_p,\]
      where $J=\Ker(\Delta:R_1\widehat \otimes_A R_2 = R\widehat \otimes_A R \to R)$ as before.
  \end{cor}
  \begin{proof}
      This is a special case of Lemma \ref{lem:cover}
  \end{proof}

  The following assumption is essential for our calculation.
  
  \begin{assumption}\label{assumpsion:key}
      Let $R$ be $p$-completely smooth $A$-algebra equipped with $\delta$-structures $\delta_1$ and $\delta_2$ as above. Assume that $\delta_1$ and $\delta_2$ satisfy one of the following equivalent conditions:
      \begin{enumerate}
          \item[(1)] For any $x\in R$, we have $\varphi_1(x) \equiv \varphi_2(x)\mod p^2$.

          \item[(2)] For any $x\in R$, we have $\delta_1(x)\equiv \delta_2(x)\mod p$.
      \end{enumerate}
  \end{assumption}

   To understand the structure of $R_1\widehat \otimes_{A}R_2\{\frac{J}{p}\}^{\wedge}_p$, we need the following lemma.
  
  \begin{lem}\label{lem:difference of delta}
      Assume $R$ is equipped with a $\delta$-structure compatible with that on $A$ and let $(R,(p))$ and $(R\widehat \otimes_AR,(p))$ be the induced crystalline prisms. Then for any $1\leq i\leq d$, the difference $\delta(T_i)-\delta(S_i)$ belongs to $J=\Ker(\Delta:R\widehat \otimes_{A}R\to R)$.
  \end{lem}
  \begin{proof}
      Note that both $\varphi(T_i)-\varphi(S_i)$ and $T_i^p-S_i^p$ belong to $J$ as they are killed by the diagonal map $\Delta$. We deduce from 
      \[p(\delta(T_i)-\delta(S_i)) = \big(\varphi(T_i)-\varphi(S_i)\big)-\big(T_i^p-S_i^p\big)\]
      that $p(\delta(T_i)-\delta(S_i))\in J$. So we can finally conclude as $R$ is $p$-torsion free.
  \end{proof}

  Now, let $\delta_1$ and $\delta_2$ be (not necessarily different) $\delta$-structures on $R$ compatible with that on $A$ and let $(R_1,(p))$ and $(R_2,(p))$ be the corresponding prisms.
  Under Assumption \ref{assumpsion:key}, we are able to describe the structure of $R_1\widehat \otimes_{A}R_2\{\frac{J}{p}\}^{\wedge}_p$. The next proposition is a slight generalization of a theorem of Ogus \cite[Thm. 2.3.13(1)]{Ogu22} when $A_0$ is \emph{non-reduced}, and we will explain how to deduce it from loc.cit. in the case where $A_0$ is reduced (cf. Remark \ref{rmk:apply Ogus}).

  \begin{prop}[Ogus]\label{prop:key}
      Keep notations of Corollary \ref{cor:self product for two delta structure} and suppose Assumption \ref{assumpsion:key} holds true. Then for any $j\in\{1,2\}$, sending $Y_i$'s to $\frac{T_i-S_i}{p}$'s induces an isomorphism of $R\widehat \otimes_AR$-algebras
      \[\iota:R_j[Y_1,\dots,Y_d]^{\wedge}_{\pd}\xrightarrow{\cong} R_1\widehat \otimes_{A}R_2\{\frac{J}{p}\}^{\wedge}_p,\]
      where $R_j[Y_1,\dots,Y_d]^{\wedge}_{\pd}$ is the $p$-complete free pd-algebra over $R_j$ generated by $Y_i$'s.
  \end{prop}
  \begin{proof}
      Let $\delta$ and $\varphi$ respectively be the induced $\delta$-structure and Frobenius endomorphism on $R_1\widehat \otimes_{A}R_2\{\frac{J}{p}\}^{\wedge}_p$. Keep notations of Remark \ref{rmk: On Ker(Delta)} and Lemma \ref{lem:Structure of D}.
      We start with the ring $D=R\widehat \otimes_AR[\frac{J}{p}]^{\wedge}_p$. By Lemma \ref{lem:Structure of D}(4), for any $j\in \{1,2\}$, we have
      \[D = R_j[Y_1,\dots,Y_d]^{\wedge}_p.\]
      
      Recall that $R_1\widehat \otimes_{A}R_2\{\frac{J}{p}\}^{\wedge}_p$ is the $p$-adic completion $D[S]^{\wedge}_p$ of the ring $D[S]$, where 
      \[S:=\{\delta^n(Y_1),\dots,\delta^n(Y_d),\delta^n(\frac{e}{p})\mid n\geq 0\}.\]
      As $e$ belongs to the $p$-complete ideal of $D$ generated by $Y_i$'s (cf. Lemma \ref{lem:Structure of D}(4)), by \cite[Lem. 2.17]{BS22}, we see that $D[S]^{\wedge}_p$ is the $p$-adic completion of the ring $D[\Lambda]$, where 
      \[\Lambda = \cup_{n\geq 1}\Lambda_n \text{ and }\Lambda_n=\{\delta^m(Y_1),\dots,\delta^m(Y_d)\mid 1\leq m\leq n\}.\]
      
      On the other hand, $R_j[Y_1,\dots,Y_d]^{\wedge}_p$ is the $p$-adic completion of the ring $D[\Lambda^{\prime}]$, where 
      \[\Lambda^{\prime} = \cup_{n\geq 1}\Lambda^{\prime}_n \text{ and }\Lambda^{\prime}_n = \{\frac{Y_1^{p^m}}{p^{1+p+\cdots+p^{m-1}}},\dots,\frac{Y_d^{p^m}}{p^{1+p+\cdots+p^{m-1}}}\mid 1\leq m\leq n\}.\]
      Therefore, to conclude the result, it suffices to show that for any $n\geq 1$, we have
      \[D[\Lambda_n] = D[\Lambda_n^{\prime}].\]
      But this follows from the lemma below.

    \begin{lem}
       Keep notations as above. For any $n\geq 1$, we have that
      \[D[\Lambda_n] = D[\Lambda_n^{\prime}],\]
      and that for any $1\leq i\leq d$, there exists some $z_{i,n-1}\in D_{n-1}$ such that
      \[\delta^n(Y_i)= z_{i,n-1}+(-1)^n\frac{Y_i^{p^n}}{p^{1+p+\cdots+p^{n-1}}},\]
      where $D_0:=D$.
    \end{lem}
    \begin{proof}
        We are going to prove the result by induction on $n$. We first deal with the $n=1$ case. 

        By Lemma \ref{lem:difference of delta} (for $\delta = \delta_1$), for any $1\leq i\leq d$, we have
      \[\delta_1(T_i)-\delta_2(S_i) = \big(\delta_1(T_i)-\delta_1(S_i)\big)+\big(\delta_1(S_i)-\delta_2(S_i)\big) \in JD+pD\subset pD.\]
      Here, we use that $JD\subset pD$. Therefore, there exists some $b_i\in D$ such that
      \[\delta_1(T_i)-\delta_2(S_i) = pb_i.\]
      Now, for any $1\leq i\leq d$, we have
      \begin{equation*}\label{equ:key}
      \begin{split}
         \varphi(Y_i) = & \delta_1(T_i)-\delta_2(S_i)+\frac{T_i^p-S_i^p}{p}\\
         = & pb_i+Y_i\sum_{k=0}^{p-1}T_i^kS_i^{p-1-k}\\
         = & pb_i+Y_i\big(pS_i^{p-1}+\sum_{k=0}^{p-1}(T_i^k-S_i^k)S_i^{p-1-k}\big)\\
         = &  pb_i+Y_i\big(pS_i^{p-1}+pY_i\sum_{k=0}^{p-1}\frac{T_i^k-S_i^k}{T_i-S_i}S_i^{p-1-k}\big)\\
         = & p\big(b_i+Y_iS_i^{p-1}+Y_i^2\sum_{k=0}^{p-1}\frac{T_i^k-S_i^k}{T_i-S_i}S_i^{p-1-k}\big).
      \end{split}
      \end{equation*}
      Therefore, if we put 
      \[z_{i,0}:=b_i+Y_iS_i^{p-1}+Y_i^2\sum_{k=0}^{p-1}\frac{T_i^k-S_i^k}{T_i-S_i}S_i^{p-1-k}\]
      then it is an element in $D$ such that $\varphi(Y_i) = pz_{i,0}$, yielding that
      \[\delta(Y_i) = \frac{\varphi(Y_i)-Y_i^p}{p} = z_{i,0}-\frac{Y_i^p}{p}.\]
      So we have
      \[D[\Lambda_1^{\prime}] = D[\frac{Y_1^p}{p},\dots,\frac{Y_d^p}{p}] = D[z_{1,0}-\delta(Y_1),\dots,z_{d,0}-\delta(Y_d)] = D[\delta(Y_1),\dots,\delta(Y_d)] = D[\Lambda_1]\]
      as desired. This completes the proof for $n=1$.

      Now, assume that the result holds true for some $n\geq 1$. Then for any $1\leq i\leq d$, we have
      \[\begin{split}
          \delta^{n+1}(Y_i) = & \delta(z_{i,n-1}+(-1)^n\frac{Y_i^{p^n}}{p^{1+p+\cdots+p^{n-1}}})\\
          = & \delta(z_{i,n-1})+\delta((-1)^n\frac{Y_i^{p^n}}{p^{1+p+\cdots+p^{n-1}}})-\sum_{k=1}^{p-1}\frac{\binom{p}{k}}{p}z_{i,n-1}^k((-1)^n\frac{Y_i^{p^n}}{p^{1+p+\cdots+p^{n-1}}})^{p-k}.
      \end{split}\]
      As $z_{i,n-1}\in D[\Lambda_{n-1}]$, we have $\delta(z_{i,n-1})\in D[\Lambda_n]$. By inductive hypothesis, 
      \[(-1)^n\frac{Y_i^{p^n}}{p^{1+p+\cdots+p^{n-1}}}\in D[\Lambda_n^{\prime}] = D[\Lambda_n].\]
      We then have
      \[\delta^{n+1}(Y_i)- \delta((-1)^n\frac{Y_i^{p^n}}{p^{1+p+\cdots+p^{n-1}}}) = \delta(z_{i,n-1})-\sum_{k=1}^{p-1}\frac{\binom{p}{k}}{p}z_{i,n-1}^k((-1)^n\frac{Y_i^{p^n}}{p^{1+p+\cdots+p^{n-1}}})^{p-k}\in D[\Lambda_n].\]
      Using $\varphi(Y_i) = pz_{i,0}$ again, we have
      \[\begin{split}
          \delta((-1)^n\frac{Y_i^{p^n}}{p^{1+p+\cdots+p^{n-1}}}) = & \frac{1}{p}\big((-1)^n\varphi(\frac{Y_i^{p^n}}{p^{1+p+\cdots+p^{n-1}}})-((-1)^n\frac{Y_i^{p^n}}{p^{1+p+\cdots+p^{n-1}}})^p\big)\\
          = & \frac{1}{p}\big((-1)^nz_{i,0}^{p^n}p^{p^n-1-p-\cdots-p^{n-1}}-(-1)^{n}\frac{Y_i^{p^{n+1}}}{p^{p+p^2+\cdots+p^n}}\big)\\
          = & (-1)^nz_{i,0}^{p^n}p^{p^n-2-p-\cdots-p^{n-1}}+(-1)^{n+1}\frac{Y_i^{p^{n+1}}}{p^{1+p+\cdots+p^{n}}},
      \end{split}\]
      yielding that
      \[\delta((-1)^n\frac{Y_i^{p^n}}{p^{1+p+\cdots+p^{n-1}}})-(-1)^{n+1}\frac{Y_i^{p^{n+1}}}{p^{1+p+\cdots+p^{n}}}\in D.\]
      Therefore, if we put 
      \[z_{i,n} = \big(\delta^{n+1}(Y_i)- \delta((-1)^n\frac{Y_i^{p^n}}{p^{1+p+\cdots+p^{n-1}}}) \big)+\big(\delta((-1)^n\frac{Y_i^{p^n}}{p^{1+p+\cdots+p^{n-1}}})-(-1)^{n+1}\frac{Y_i^{p^{n+1}}}{p^{1+p+\cdots+p^{n}}}\big),\]
      then it is an element in $D[\Lambda_n] = D[\Lambda_n^{\prime}]$ such that
      \[\begin{split}&\delta^{n+1}(Y_i)=(-1)^{n+1}\frac{Y_i^{p^{n+1}}}{p^{1+p+\cdots+p^{n}}}+z_{i,n}.\end{split}\]
      Finally, we have
      \[\begin{split}D[\Lambda_{n+1}] =& D[\Lambda_n][\delta^{n+1}(Y_1),\dots,\delta^{n+1}(Y_d)]\\
      =&D[\Lambda_n^{\prime}][(-1)^{n+1}\frac{Y_1^{p^{n+1}}}{p^{1+p+\cdots+p^{n}}}+z_{1,n},\dots,(-1)^{n+1}\frac{Y_d^{p^{n+1}}}{p^{1+p+\cdots+p^{n}}}+z_{d,n}]\\
      =&D[\Lambda_n^{\prime}][\frac{Y_1^{p^{n+1}}}{p^{1+p+\cdots+p^{n}}},\dots,\frac{Y_d^{p^{n+1}}}{p^{1+p+\cdots+p^{n}}}]\\
      =&D[\Lambda_{n+1}^{\prime}].\end{split}\]
      So we see the lemma holds true for $n+1$ as desired.
    \end{proof}Thanks to the lemma, we have
    \[R_j[Y_1,\dots,Y_d]^{\wedge}_p = D[\Lambda^{\prime}]^{\wedge}_p = (\colim_nD[\Lambda_n^{\prime}])^{\wedge}_p = (\colim_nD[\Lambda_n])^{\wedge}_p=D[\Lambda]^{\wedge}_p = R_1\widehat \otimes_AR_2\{\frac{J}{p}\}^{\wedge}_p\]
    as desired. This completes the proof.
  \end{proof}

\begin{rmk}
    Compared with \cite[Lem. 5.2]{Tia23}, the $\delta$-structure on $R$ is not induced from the given chart. So, one cannot obtain Proposition \ref{prop:key} from the argument in loc.cit. directly.
\end{rmk}

\begin{rmk}\label{rmk:apply Ogus}
    Let us explain how to deduce Proposition \ref{prop:key} from \cite[Thm. 2.3.13(1)]{Ogu22} directly. In the rest of this remark, we follow the language and notation in \cite{Ogu22} and assume $A_0$ is \emph{reduced}. Consider the $\phi$-scheme $Y:=\Spf(R\widehat \otimes_AR)$ with the Frobenius structure induced by $\delta_1$ and $\delta_2$ as in Proposition \ref{prop:key}. Then the closed embedding $i:X =\Spec(R_0)\to \overline Y:=\Spec(R_0\widehat \otimes_{A_0}R_0)$ admits a lifting $j:\widetilde X=\Spf(R\widehat \otimes_{A[\underline T]}R)\to Y$ defined by the ideal $J =\Ker(\Delta:R\widehat \otimes_AR\to R)$. Under Assumption \ref{assumpsion:key}, it can be checked that $\widetilde X_1 = \widetilde X\times_{\Spf(A)}\Spec(A_1)$ is invariant under the Frobenius $\rF_{Y_1}$ of $Y_1=Y\times_{\Spf(A)}\Spec(A_1)$ as required in \cite[Thm. 2.3.13(1)]{Ogu22}. Therefore, using \cite[Thm. 2.3.13(1)]{Ogu22}, the prismatic envelope $\Prism_X(Y) = \Spf(R_1\widehat \otimes_{A}R_2\{\frac{J}{p}\}^{\wedge}_p)$ coincides with the pd-dilatation of the natural morphism $\widetilde X\to \bD_X(Y) = \Spf(R_1\widehat \otimes_AR_2[\frac{J}{p}]^{\wedge}_p) = \Spf(D)$ defined by the sequence $\frac{T_1-S_1}{p},\dots,\frac{T_d-S_d}{p}$, which is exactly $\bP\bD_X(\bD_X(Y)) = \Spf(R[Y_1,\dots,Y_d]^{\wedge}_{\pd})$ by \cite[Prop. 2.1.3(2)]{Ogu22} (It is where we need that $A_0$ is reduced).
\end{rmk}

   \begin{conv}\label{conv:different delta}
       Let $\delta_1$ and $\delta_2$ be two $\delta$-structures on $R$, $(R_1,(p))$ and $(R_2,(p))$ be the induced crystalline prisms, and $\underline T$ and $\underline S$ be the coordinates on $R_1$ and $R_2$, respectively. Let $(R_{12},(p))$ be the coproduct $(R_1,(p))\times(R_2,(p))$ in $(R_0/A)_{\Prism}$. When Assumption \ref{assumpsion:key} is true, we have 
       \[R_{12} =  R_1[\underline Y]_{\pd}^{\wedge} = R_2[\underline Y]^{\wedge}_{\pd}.\]
   \end{conv}

   \begin{rmk}\label{rmk:R12}
      Clearly, $(R_{12},(p))$ is a covering of $(R_1,(p))$ in $(R_0/A)_{\Prism}$. Denote by $(\widetilde R^{\bullet},(p))$ the \v Cech nerve associated to this covering; namely, for any $\bullet\geq 0$, $(\widetilde R^{\bullet},(p))$ is the fiber coproduct
      \[(R_{12},(p))\times_{(R_1,(p))}(R_{12},(p))\times_{(R_1,(p))}\cdots\times_{(R_1,(p))}(R_{12},(p))\]
      of $(\bullet+1)$-copies of $(R_{12},(p))$ over $(R_1,(p))$. By definition of $(R_{12},(p))$, for any $\bullet\geq 0$, $(\widetilde R^{\bullet},(p))$ is also the coproduct 
      \[(R_1,(p))\times(R_2,(p))\times\cdots\times(R_2,(p))\]
      of $(R_1,(p))$ with $(\bullet+1)$-copies of $(R_2,(p))$. Therefore, if $(R^{\bullet}_2,(p))$ stands for the \v Cech nerve associated to the covering $(R_2,(p))$ of the initial object of $(R_0/A)_{\Prism}$ (cf. Lemma \ref{lem:cover}), then there is a natural map of cosimplicial rings $R_2^{\bullet}\to \widetilde R^{\bullet}$, which gives rise to a functor 
      \[\Strat(R^{\bullet}_2)\to \Strat(\widetilde R^{\bullet}).\]
      We will see that, in the case that we are interested in, any stratification in $\Strat(\widetilde R^{\bullet})$ gives rise to an $R_1$-module. This allows us to relate an $R_2$-module (underlying a stratification with respect to $R_2^{\bullet}$) to an $R_1$-module, and finally helps us to compare some constructions corresponding to the different $\delta$-structures $\delta_1$ and $\delta_2$ on $R$.
  \end{rmk}

  \begin{cor}\label{cor:key}
      Suppose Assumption \ref{assumpsion:key} holds true and keep notations of Convention \ref{conv:different delta}. 
      \begin{enumerate}
          \item[(1)] For any $*\in \{1,2\}$, let $(R_{*}^{\bullet},(p))$ be the self-coproduct
          \[(R_*,(p))\times(R_*,(p))\times\cdots\times(R_*,(p))\]
          of $(\bullet+1)$-copies of $(R_*,(p))$ in $(R_0/A)_{\Prism}$. For any $0\leq j\leq \bullet$, let $\underline T_j$ (resp. $\underline S_j$) be the coordinate on the $(j+1)$-th component of $(R_1^{\bullet},(p))$ (resp. $(R_2^{\bullet},(p))$); for example, $\underline T_0$ represents $T_{0,1},\cdots, T_{0,d}$ and etc.. For any $1\leq i\leq \bullet$, set $\underline X_i = \frac{\underline T_0-\underline T_i}{p}$ (resp. $\underline Z_i = \frac{\underline S_0-\underline S_i}{p}$) and put $\underline T:=\underline T_0$ for short. 
          Then we have an isomorphism of cosimplicial rings
          \[R_1^{\bullet}\simeq R_1[\underline X_1,\dots,\underline X_{\bullet}]^{\wedge}_{\pd},\]
          where the $R_1$-linear structures on both sides are induced by the first component of $R_1^{\bullet}$, and the face maps $p_i$'s on $R_1[\underline X_1,\dots,\underline X_{\bullet}]^{\wedge}_{\pd}$ are induced by 
          \begin{equation}\label{equ:face-I}
              \begin{split}
                  &p_i(\underline T) = \left\{
                  \begin{array}{rcl}
                      \underline T-p\underline X_1, & i=0 \\
                      \underline T, & i\geq 1
                  \end{array}
                  \right.\\
                  &p_i(\underline X_j) = \left\{
                  \begin{array}{rcl}
                      \underline X_{j+1}-\underline X_1, & i=0 \\
                      \underline X_{j+1}, & 0<i\leq j\\
                      \underline X_j, & i>j
                  \end{array}
                  \right.
              \end{split}
          \end{equation}
          while the degeneracy maps $\sigma_i$'s are $R_1$-linear and induced by
          \begin{equation}\label{equ:degeneracy-I}
              \begin{split}
                  \sigma_i(\underline X_j) = \left\{
                  \begin{array}{rcl}
                      0, & (i,j) = (0,1) \\
                      \underline X_{j-1}, & i < j \text{ and }(i,j)\neq (0,1)\\
                      \underline X_j, & i\geq j.
                  \end{array}
                  \right.
              \end{split}
          \end{equation}
          The same result holds true for $R_2$ after replacing $\underline T$ and $\underline X_i$'s by $\underline S$ and $\underline Z_i$'s respectively.

          \item[(2)] Let $(\widetilde R^{\bullet},(p))$ be the coproduct 
          \[(R_1,(p))\times (R_2,(p))\times(R_2,(p))\times\cdots\times(R_2,(p))\]
          of $(R_1,(p))$ with $(\bullet+1)$-copies of $(R_2,(p))$ in $(R_0/A)_{\Prism}$. 
          For any $0\leq j\leq \bullet$, let $\underline S_j$ be the coordinate on the $(j+1)$-th $(R_2,(p))$ in $(\widetilde R^{\bullet},(p))$ and let $\underline T$ be the coordinate on $(R_1,(p))$. 
          For any $1\leq i\leq \bullet$, set $\underline Z_i = \frac{\underline S_0-\underline S_i}{p}$ and $\underline Y = \frac{\underline T_0-\underline S_0}{p}$. 
          Then we have an isomorphism of cosimplicial rings
          \[\widetilde R^{\bullet}\simeq R_{12}[\underline Z_1,\dots,\underline Z_{\bullet}]^{\wedge}_{\pd}\]
          where the face maps $p_i$'s on $R_{12}[ \underline Z_1,\dots,\underline Z_{\bullet}]^{\wedge}_{\pd} = R_{1}[ \underline Y, \underline Z_1,\dots,\underline Z_{\bullet}]^{\wedge}_{\pd}$ are $R_1$-linear and induced by
          \begin{equation}\label{equ:face-II}
              \begin{split}
                  &p_i(\underline Y) = \left\{
                  \begin{array}{rcl}
                      \underline Y+\underline Z_1, & i=0 \\
                      \underline Y, & i\geq 1
                  \end{array}
                  \right.\\
                  &p_i(\underline Z_j) = \left\{
                  \begin{array}{rcl}
                      \underline Z_{j+1}-\underline Z_1, & i=0 \\
                      \underline Z_{j+1}, & 0<i\leq j\\
                      \underline Z_j, & i>j
                  \end{array}
                  \right.
              \end{split}
          \end{equation}
          while the degeneracy maps $\sigma_i$'s are $R_{12}$-linear induced by \eqref{equ:degeneracy-I} for $\underline Z_j$'s instead of $\underline X_j$'s.
        \end{enumerate}
  \end{cor}
  \begin{proof}
      This can be easily deduced from the same calculation as in the proof of Proposition \ref{prop:key}.
  \end{proof}

\section{Prismatic crystal as $p$-connection: Local construction}\label{sec:local construction}
  Throughout this section, let $R_0$ be a smooth $A_0$-algebra with a fixed chart $\Box$, and let $R$ be a smooth lifting of $R_0$ to $A$ with the coordinates $\underline T$ induced from $\Box$ as in the previous section. For any $n\in\bN_*$, define $R_n = R/p^{n+1}R$ and $\frakX_{n}:=\Spf(R_n)$. We also equip $R$ with a $\delta$-structure so that $(R,(p))$ is a prism in $(R_0/A)_{\Prism}$.

\subsection{Prismatic crystal as stratifications}
  We start with some well-known results, which allow us to study prismatic crystals via certain stratifications.
  
  Denote by $(R^{\bullet},(p))$ the cosimplicial prism which represents the \v Cech nerve associated to the covering $(R,(p))$ of the final object of $\Sh((R_0/A)_{\Prism})$ (cf. Lemma \ref{lem:cover}). By Corollary \ref{cor:key}(1), we have an isomorphism of cosimplicial $R$-algebras 
  \[R^{\bullet} \cong R[\underline X_1,\dots,\underline X_{\bullet}]^{\wedge}_{pd}.\]
  \begin{lem}\label{lem:crystal vs stratification}
      For any $n\in \bN_*$, the evaluation along $(R_{\bullet},(p))$ induces an equivalence of categories
      \[\Vect((R_0/A)_{\Prism},\calO_{\Prism,n})\xrightarrow{\simeq}\Strat(R^{\bullet}/p^n) = \Strat(R[\underline X_1,\dots,\underline X_{\bullet}]^{\wedge}_{\pd}/p^n),\]
      which preserves ranks, tensor products and dualities.
  \end{lem}
  \begin{proof}
      This follows from \cite[Prop. 4.8]{Tia23}
  \end{proof}
  \begin{cor}\label{cor:crystal vs stratification}
      Keep assumptions and notations of Corollary \ref{cor:key}. For any $\bM\in \Vect((R_0/A)_{\Prism},\calO_{\Prism,n})$, we have that $\bM(R_1,(p)) = \lim_{\Delta}\bM(\widetilde R^{\bullet},(p))$.
  \end{cor}
  \begin{proof}
      As $(\widetilde R^{\bullet},(p))$ is the \v Cech nerve associated to the covering $(R_1,(p))\to (R_{12},(p))$ in $(R_0/A)_{\Prism}$ (cf. Remark \ref{rmk:R12}), this follows from \cite[Lem. 4.12]{Tia23}.
  \end{proof}

  Thanks to Lemma \ref{lem:crystal vs stratification}, in order to study prismatic crystals of truncation $n$, i.e. objects in $\Vect((R_0/A)_{\Prism},\calO_{\Prism,n})$, we only need to study $R$-modules with a stratification with respect to $R^{\bullet} = R[\underline X_1,\dots,\underline X_{\bullet}]^{\wedge}_{\pd}/p^{n+1}$.

\subsection{Prismatic crystals $p$-connections: Local equivalence}\label{ssec:local equivalence}
  We study $R$-modules with a stratification with respect to $R[\underline X_1,\dots,\underline X_{\bullet}]^{\wedge}_{\pd}/p^{n+1}$ in this section, following the strategy in \cite[\S5]{Tia23}. It is worth pointing out that the main result of this section, i.e. Proposition \ref{prop:local equivalence}, can be obtained by combining \cite[Prop. 4.8]{Tia23} and \cite[Prop. 2.9]{Shiho}. But for the further use, let us recall the details.
  
  Let $M$ be a finite projective $R/p^n$-module and $\varepsilon_M:M\otimes_{R,p_1}R^1\to M\otimes_{R,p_0}R^1$ be an $R^1$-linear isomorphism.
  Then one can write the restriction of $\varepsilon_M$ to $M$ as
  \[\varepsilon_M = \sum_{\underline n\in\bN^d}\theta_{\underline n}\underline X_1^{[\underline n]}\]
  with $\theta_{\underline n}\in\End_A(M)$ satisfying the following nilpotence condition:
  \begin{equation}\label{equ:nilpoteny}
      \lim_{|\underline n|\to+\infty}\theta_{\underline n} = 0.
  \end{equation}
  Here and in what follows, for any $\underline n=(n_1,\dots,n_d)\in\bN^d$, $\underline X_1^{[\underline n]} = X_{1,1}^{[n_1]}\cdots X_{d,1}^{[n_d]}$ and $|\underline n| = n_1+\dots+n_d$ (cf. \S\ref{sec:notations}).
  We remark that as $M$ is finite projective, the condition (\ref{equ:nilpoteny}) is equivalent to the condition that for any $x\in M$, we have
  \[\lim_{|\underline n|\to+\infty}\theta_{\underline n}(x) = 0.\]
  
  From \cite[Eq. (5.11.2)]{Tia23}, we have
    \begin{align*}
      &\begin{aligned}
          \bullet\quad p_2^*(\varepsilon_M)\circ p_0^*(\varepsilon_M) = {} & p_2^*(\varepsilon_M)\circ\sum_{\underline l\in\bN^d}\theta_{\underline l}(\underline X_2-\underline X_1)^{[\underline l]}\\
          = & \sum_{\underline l,\underline k\in\bN^d}\theta_{\underline k}\circ\theta_{\underline l}X_1^{[\underline k]}(\underline X_2-\underline X_1)^{[\underline l]}\\
          = & \sum_{\underline l,\underline k,\underline m\in\bN^d}(-1)^{|\underline l|}\theta_{\underline k}\circ\theta_{\underline l+\underline m}\underline X_1^{[\underline k]}\underline X_1^{[\underline l]}\underline X_2^{[\underline m]},
      \end{aligned}\\
      &\begin{aligned}
          \bullet\quad p_1^*(\varepsilon_M) = \sum_{\underline m\in\bN^d}\theta_{\underline m}X_2^{[\underline m]},
      \end{aligned}\\
      &\begin{aligned}
          \bullet \quad \sigma_0^*(\varepsilon_M) = \theta_{\underline 0}.
      \end{aligned}
  \end{align*}
  Therefore, we have the following result:
  \begin{lem}\label{lem:p_2p_0=p_1}
      The pair $(M,\varepsilon_M = \sum_{\underline n\in\bN^d}\theta_{\underline n}\underline X_1^{[\underline n]})$ lies in $\Strat(R[\underline X_1,\dots,\underline X_{\bullet}]^{\wedge}_{\pd}/p^n)$ if and only if the following conditions hold true for $(M,\varepsilon_M)$:
      \begin{enumerate}
          \item[(1)] $\theta_{\underline 0} = \id_M$,

          \item[(2)] For any $\underline m\in\bN^d$, we have
          \[\theta_{\underline m} = \sum_{\underline l,\underline k,\in\bN^d}(-1)^{|\underline l|}\theta_{\underline k}\circ\theta_{\underline l+\underline m}\underline X_1^{[\underline k]}\underline X_1^{[\underline l]}.\]

          \item[(3)] $\lim_{|\underline n|\to+\infty}\theta_{\underline n} = 0$.
      \end{enumerate}
  \end{lem}
  The next lemma gives an equivalent description of Condition (2) in Lemma \ref{lem:p_2p_0=p_1}.
  \begin{lem}\label{lem:solve p_2p_0=p_1}
      For any $1\leq i\leq d$, let $\phi_i = \theta_{\underline 1_i}$. Assume $\theta_{\underline 0} = \id_M$. Then the following conditions are equivalent:
      \begin{enumerate}
          \item[(1)] For any $\underline m\in\bN^d$, we have $\theta_{\underline m} = \sum_{\underline l,\underline k,\in\bN^d}(-1)^{|\underline l|}\theta_{\underline k}\circ\theta_{\underline l+\underline m}\underline X_1^{[\underline k]}\underline X_1^{[\underline l]}$.

          \item[(2)] The $\phi_i$ commutes with $\phi_j$ for any $1\leq i,j\leq d$ such that for any $\underline m\in\bN^d$, we have $\theta_{\underline m} = \underline \phi^{\underline m}$.
      \end{enumerate}
  \end{lem}
  \begin{proof}
      For $(1)\Rightarrow(2)$: We will prove the result by comparing the coefficients of $\underline X_1^{[\underline h]}$ for certain $\underline h\in \bN^d$ on both sides of 
      \[\begin{split}\theta_{\underline m} &= \sum_{\underline l,\underline k,\in\bN^d}(-1)^{|\underline l|}\theta_{\underline k}\circ\theta_{\underline l+\underline m}\underline X_1^{[\underline k]}\underline X_1^{[\underline l]}\\
      &=\sum_{\underline l=(l_1,\dots,l_d),\underline k=(k_1,\dots,k_d),\in\bN^d}(-1)^{|\underline l|}\theta_{\underline k}\circ\theta_{\underline l+\underline m}X_{1,1}^{[k_1]}\cdots X_{1,d}^{[k_d]}\cdot X_{1,1}^{[l_1]}\cdots X_{1,d}^{[l_d]}.\end{split}\]
      For any $1\leq i\leq d$, consider the coefficients of $X_{1,i}$, then we have $0 = \phi_i\circ\theta_{\underline m}-\theta_{\underline m+\underline 1_i}$; that is, for any $\underline m\in\bN^d$, we have
      \[\theta_{\underline m+\underline 1_i} = \phi_i\circ\theta_{\underline m}.\]
      In particular, for any $1\leq i,j\leq d$, we have
      \[\phi_i\circ\phi_j = \theta_{\underline 1_i+\underline 1_j} = \phi_j\circ\phi_i.\]
      Using this, by iteration, for any $\underline m = (m_1,\dots,m_d)\in \bN^d$, we have $\theta_{\underline m} = \underline \phi^{\underline m} := \phi_1^{m_1}\cdots\phi_d^{m_d}$ as desired.

      For $(2)\Rightarrow(1)$: In this case, we have
      \begin{equation*}
          \begin{split}
              \sum_{\underline l,\underline k,\in\bN^d}(-1)^{|\underline l|}\theta_{\underline k}\circ\theta_{\underline l+\underline m}\underline X_1^{[\underline k]}\underline X_1^{[\underline l]} = & \sum_{\underline l,\underline k,\in\bN^d}(-1)^{|\underline l|}\underline \phi^{\underline k}\circ\underline \phi^{\underline l+\underline m}\underline X_1^{[\underline k]}\underline X_1^{[\underline l]}\\
              = &\underline \phi^{\underline m}\sum_{\underline h\in\bN^d}\sum_{\underline k+\underline l=\underline h}(-1)^{|\underline l|}\underline \phi^{\underline k}\circ\underline \phi^{\underline l}\underline X_1^{[\underline k]}\underline X_1^{[\underline l]}\\
              = & \theta_{\underline m}
          \end{split}
      \end{equation*}
      as desired, where the last equality follows from the identity 
      \[\sum_{k+l = h}(-1)^l\phi^k\phi^lX^{[k]}X^{[l]} = (\phi X-\phi X)^h = 0\]
      when $h \geq 1$.
  \end{proof}
  \begin{cor}\label{cor:stratification}
      The following are equivalent:
      \begin{enumerate}
          \item[(1)] The pair $(M,\varepsilon_M)$ is an object in the category $\Strat(R[\underline X_1,\dots,\underline X_d]^{\wedge}_{\pd}/p^{n+1})$.

          \item[(2)] There are endomorphisms $\phi_1,\dots,\phi_d\in \End_A(M)$ which are topologically quasi-nilpotent and commute with each other such that
          \[\varepsilon_M = \sum_{\underline m\in\bN^d}\underline \phi^{\underline m}\underline X_1^{[\underline m]} = \exp(\sum_{i=1}^d\phi_iX_{i,1}).\]
      \end{enumerate}
  \end{cor}
  \begin{proof}
      It is easy by Lemma \ref{lem:p_2p_0=p_1} combined with Lemma \ref{lem:solve p_2p_0=p_1}.
  \end{proof}
  As in \cite[\S5.10]{Tia23}, we can give an intrinsic construction of $\phi_i$'s.
  \begin{const}\label{const:intrinsic}
      Let $J_{\pd}$ be the kernel of the degeneracy map $\sigma_0:R^1\to R$. Via the isomorphism in Corollary \ref{cor:key}, we see that $J_{\pd} = (X_{1,1},\dots,X_{1,d})_{\pd}$ is the $p$-complete pd-ideal of $R[\underline X_1]^{\wedge}_{\pd}$ generated by $\underline X_1 = \frac{\underline T_0-\underline T_1}{p}$. Let $J_{\pd}^{[2]}$ be the $p$-complete pd-square of $J_{\pd}$ and $J = \Ker(R\widehat \otimes_AR\to R)$ be the ideal as in the very beginning of \S\ref{sec:key proposition}. Then the natural inclusion $R\widehat \otimes_AR\to R^{1}$ induces a natural morphism $J^*\to J_{\pd}^{[*]}$ by identifying $\underline T_0-\underline T_1$ with $p\underline X_1$ for $*\in\{1,2\}$. So we get a canonical map
      \begin{equation}\label{equ:intrinsicI}
          \oplus_{i=1}^dR\cdot pX_{1,i} = J/J^2\hookrightarrow J_{\pd}/J_{\pd}^{[2]} = \oplus_{i=1}^dR\cdot X_{1,i}.
      \end{equation}
     As $J/J^2 = \widehat \Omega^1_{R/A}$, we obtain a canonical isomorphism 
      \begin{equation}\label{equ:intrinsicII}
     J_{\pd}/J_{\pd}^{[2]}\xrightarrow{\cong}\widehat \Omega^1_{R/A}\cdot p^{-1}=:\widehat \Omega^1_{R/A}\{-1\},
      \end{equation}
      by identifying $X_{1,i}$'s with $\rd T_i\{-1\} = \frac{\rd T_i}{p}$.
      For any $(M,\varepsilon_M)\in \Strat(R^{\bullet})$, let $\iota_M:M\to M\otimes_{R,p_1}R^1$ be the map sending $x\to x\otimes 1$. As $\sigma_0^*(\varepsilon_M) = \id_M$, the difference $\varepsilon_M-\iota_M:M\to M\otimes_{R,p_1}R^1$ takes values in $M\otimes_{R,p_1}J_{\pd}$, and thus we have a canonical morphism
      \[\nabla_M:=\overline{\iota_M-\varepsilon_M}:M\to M\otimes_{R,p_1}J_{\pd}/J_{\pd}^{[2]}\cong M\otimes_R\widehat \Omega^1_{R/A}\{-1\}.\]
      By Corollary \ref{cor:stratification}(2), we conclude that 
      \[\nabla_M = -\sum_{i=1}^d\phi_i\otimes\frac{\rd T_i}{p}.\]
  \end{const}
  \begin{dfn}\label{dfn:p-connection}
      For any $n\in\bN_*$, let $\frakX_n$ be a formally smooth formal scheme over $A_n$. By a \emph{$p$-connection} on $\frakX_n$, we mean a pair $(\calM,\nabla_{\calM})$ of a vector bundle $\calM$ on $\frakX_{n,\et}$ and an $A_n$-linear morphism $\nabla_{\calM}:\calM\to \calM\otimes_{\calO_{\frakX_n}}\widehat \Omega_{\frakX_n/A_n}^1\{-1\}$ satisfying the following \emph{Leibniz rule}: 
      
      For any local sections $f\in \calO_{\frakX_n}$ and $x\in \calM$, we have
      \[\nabla_{\calM}(fx) = f\nabla_{\calM}(x)+p\rd(f)\otimes x,\]
      where we regard $\rd(f)$ as an element in $\widehat \Omega^1_{R_n/A_n}\{-1\}$ and thus it is of the form $\rd(f) = \sum_{i=1}^d\frac{\partial f}{\partial T_i}\otimes\frac{\rd T_i}{p}$.

      We say a $p$-connection is \emph{flat}, if $\nabla_{\calM}\wedge\nabla_{\calM} = 0$. A flat $p$-connection is called \emph{topologically quasi-nilpotent}, if $\lim_{n\to+\infty}\nabla_{
          \calM}^n = 0$; that is, when $\frakX_n$ admits a chart with the induced coordinates $\underline T$, if we write $\nabla_{\calM} = \sum_{i=1}^d\nabla_{i}\otimes\frac{\rd T_i}{p}$, then we have $\lim_{|\underline n|\to+\infty}\underline \nabla^{\underline n}= 0$. This condition is independent of the choice of charts on $\frakX_n$ (cf. \cite[Lem. 1.6]{Shiho})
      If $(\calM,\nabla_{\calM})$ is flat, we denote by $\rD\rR(\calM,\nabla_{\calM})$ the induced \emph{de Rham complex}: 
      \[\calM\xrightarrow{\nabla_{\calM}}\calM\otimes\widehat \Omega^1_{\frakX_n/A_n}\{-1\}\xrightarrow{\nabla_{\calM}}\calM\otimes\widehat \Omega^2_{\frakX_n/A_n}\{-2\}\xrightarrow{\nabla_{\calM}}\cdots\xrightarrow{\nabla_{\calM}}\calM\otimes\widehat \Omega^d_{\frakX_n/A_n}\{-d\}.\]
      Denote by $\MIC_p^{\rm tqn}(\frakX_n)$ the category of topologically quasi-nilpotent flat $p$-connections on $\frakX_n$. When $n = 0$, the $\MIC_p^{\rm tqn}(\frakX_0)$ coincides with $\HIG^{\nil}(\frakX_0)$, the category of nilpotent Higgs bundles on $\frakX_0$.
  \end{dfn}

  \begin{lem}\label{lem:stratification vs p-connection}
      For any $(M,\varepsilon_M)\in\Strat(R^{\bullet})$, the pair $(M,\nabla_M)$ as described above defines a topologically quasi-nilpotent flat $p$-connection on $\frakX_n$.
  \end{lem}
  \begin{proof}
      It is topologically quasi-nilpotent and flat by Lemma \ref{lem:p_2p_0=p_1}(3) and Lemma \ref{lem:solve p_2p_0=p_1}, respectively. It suffices to show that $\nabla_M$ satisfies the Leibniz rule in Definition \ref{dfn:p-connection}. To do so, we use the intrinsic construction of $\nabla_M$ in Construction \ref{const:intrinsic}. Fix an $f\in R$ and an $x\in M$, then we have
      \[\begin{split}
      \nabla_M(fx) = & \overline{\iota_M(fx)-\varepsilon_M(fx)}\\
      = & \overline{p_1(f)\iota_M(x)-p_0(f)\varepsilon_M(x)}\\
      = &\overline{(p_1(f)-p_0(f))\iota_M(x)+p_0(f)(\iota_M(f)-\varepsilon_M(x))}\\
      = & \overline{p_1(f)-p_0(f)}\otimes x+f\nabla_M(x)\\
      = & p\rd(f)\otimes x+f\nabla_M(x),
      \end{split}\]
      as desired, where the last equality follows as $\overline{p_1(f)-p_0(f)}$ coincides with the $\rd(f)\in J/J^2\cong \widehat \Omega_{R_n/A_n}^1$ and hence with $p\rd(f)\in J_{\pd}/J_{\pd}^{[2]}\cong \Omega_{R_n/A_n}^1\{-1\}$ by (\ref{equ:intrinsicI}) and (\ref{equ:intrinsicII}).
  \end{proof}
  Now, we have the main result in this section, which is the local version of Theorem \ref{thm:main}
  \begin{prop}\label{prop:local equivalence}
      Let $\frakX = \Spec(R_0)$ be a smooth scheme over $A_0$ of dimension $d$ with a fixed chart. Let $\frakX = \Spf(R)$ be a smooth lifting of $\frakX_0$ over $A$ with $\frakX_n = \Spec(R_n)$. Fix a $\delta$-structure on $R$. Then there exists an equivalence of categories
      \[\Vect((\frakX/A)_{\Prism},\calO_{\Prism,n})\simeq \MIC^{\rm tqn}_p(\frakX_n)\]
      for any $n\in\bN_*$, which preserves ranks, tensor products and dualitis.
  \end{prop}
  \begin{proof}
      The equivalence of categories follows from Lemma \ref{lem:crystal vs stratification}, Corollary \ref{cor:stratification}, and Lemma \ref{lem:stratification vs p-connection}. By standard linear algebra, the equivalence preserves ranks, tensor products, and dualities.
  \end{proof}      
  \begin{rmk}\label{rmk:devissage}
      For any $0\leq m<n\leq \infty$, if $\bM\in \Vect((\frakX_0/A),\calO_{\Prism,n})$, then $\bM/p^{m+1}$ and $p^m\bM$ are objects in $\Vect((\frakX_0/A),\calO_{\Prism,m})$ and $\Vect((\frakX_0/A),\calO_{\Prism,n-m})$, respectively. Let $(\calM,\nabla_{\calM})$ be the $p$-connection corresponding to $\bM$. Then it is easy to see that $(\calM/p^{m+1},\nabla_{\calM/p^{m+1}})$ and $(p^m\calM,\nabla_{p^m\calM})$ are respectively $p$-connections corresponding to $\bM/p^{m+1}$ and $p^m\bM$, where $\nabla_{p^m\calM}$ and $\nabla_{\calM/p^{m+1}}$ denote the restrictions of $\nabla_{\calM}$ to $p^m\calM$ and $\calM/p^{m+1}$ respectively.
  \end{rmk}

\subsection{Cohomological comparison}\label{ssec:cohomological comparison}

  Now, we are going to prove that the equivalence in Proposition \ref{prop:local equivalence} is also compatible with cohomologies, following the strategy in \cite[\S5]{Tia23}. In other words, this subsection is devoted to proving the following proposition:
  \begin{prop}\label{prop:cohomology comparison}
      Let $\frakX_0 = \Spec(R_0)$ be a smooth scheme over $A_0$ of dimension $d$ with a fixed chart, $\frakX = \Spf(R)$ be a smooth lifting of $\frakX_0$ over $A$ with the reduction $\frakX_n = \Spec(R_n)$ modulo $p^{n+1}$ for any $n\geq 0$, and $\nu:(\frakX_0/A)_{\Prism}\to \frakX_{0,\et}\simeq\frakX_{n,\et}$ be the natural morphism of sites (cf. \cite[Const. 4.4]{BS22}). Fix a $\delta$-structure on $R$ compatible with that on $A$. Then for any $\bM \in \Vect((\frakX_0/A)_{\Prism},\calO_{\Prism,n})$ with the induced $p$-connection $(\calM,\nabla_{\calM})\in \MIC^{\rm tqn}_p(\frakX_n)$ via the equivalence in Proposition \ref{prop:local equivalence}, there is an explicitly constructed quasi-isomorphism of complexes of sheaves of $A_n$-modules
      \[\rho_{\delta}(\bM):\rR\nu_*\bM\simeq \rD\rR(\calM,\nabla_{\calM}),\]
      which is functorial in $\bM$ and compatible with \'etale localizations of $\frakX_0$.
  \end{prop}
  Before giving the proof (at the end of this subsection), let us remark that as $\frakX_0$ is affine, if we have $M=\Gamma(\frakX_0,\calM)$ with the induced flat topologically quasi-nilpotent $p$-connection $\nabla_M:M\to M\otimes_{R}\widehat \Omega^1_{R/A}\{-1\}$ from $\nabla_{\calM}$ and the induced de Rham complex $\rD\rR(M,\nabla_M)$, then we can reduce Proposition \ref{prop:cohomology comparison} to showing that there is an explicit quasi-isomorphism
  \begin{equation}\label{equ:prismatic vs dR complex:local}
      \rho_{\delta}(M):\rR\Gamma((\frakX_0/A)_{\Prism},\bM)\simeq \rD\rR(M,\nabla_M),
  \end{equation}
  which is natural in the sense that it (up to now) only depends on the choice of $\delta$-structures on $R$. 
  
  The rest of this subsection is devoted to constructing such a $\rho_{\delta}$ and showing that it is a quasi-isomorphism. We start with some basic constructions. Recall that the given $\delta$-structure on $R$ gives rise to a crystalline prism $(R,(p))\in (R_0/A)_{\Prism}$ which is a covering of final object of $\Sh((R_0/A)_{\Prism})$. The cosimplicial prism $(R^{\bullet},(p))$ represents the \v Cech nerve associated to this covering, and by Corollary \ref{cor:key}(1), we have an isomorphism of cosimplicial rings 
  \[R^{\bullet}\cong R[\underline X_1,\dots,\underline X_{\bullet}]^{\wedge}_{\pd}.\]

  For any $m\geq 1$, we put 
  \[\widehat \Omega^1_{R^m} := \widehat \Omega^1_{R/A}\{-1\}\otimes_{R}R^m\oplus\widehat \Omega^1_{R^m/R},\]
  which admits an $R$-basis $\frac{\rd T_1}{p},\dots,\frac{\rd T_d}{p}, \rd X_{1,1},\dots,\rd X_{1,d},\dots,\rd X_{m,1},\dots,\rd X_{m,d}$. By Proposition \ref{cor:key}(1), we have the identifications 
  \begin{equation}\label{equ:identification on differential}
      \frac{\rd T_{i}}{p} = \frac{\rd T_{0,i}}{p}\text{ and }\frac{\rd T_{j,i}}{p} = \frac{\rd T_{i}}{p}-\rd X_{j.i}
  \end{equation}
  for any $1\leq i\leq d$ and $1\leq j\leq m$. As in \cite[the paragraph above Lem. 5.15]{Tia23}, for any $n,m\geq 0$ and any order-preserving map $f:\{0,1,\dots,m\}\to \{0,1,\dots,n\}$,
  the map 
  \[f_*:\widehat \Omega^1_{R^m}\to \widehat \Omega^1_{R^n}\]
  sending each $\frac{\rd T_{j,i}}{p}$ to $\frac{\rd T_{f(j),i}}{p}$ equips $\widehat \Omega^1_{R^{\bullet}}$ with a structure of cosimplicial $R^{\bullet}$-module, and thus we get a cosimplicial $R^{\bullet}$-module $\widehat \Omega^j_{R^{\bullet}}$ for any $j\geq 1$ by taking $j$-fold wedge products. By \cite[Lem. 5.15]{Tia23}, for each $j\geq 1$, the cosimplicial $R^{\bullet}$-module $\widehat \Omega^j_{R^{\bullet}}$ is homotopic to zero.
 % \begin{lem}\label{lem:homotopy to zero}
 %     For each $j\geq 1$, the cosimplicial $R^{\bullet}$-module $\widehat \Omega^j_{R^{\bullet}}$ is homotopic to zero.
%  \end{lem}
%  \begin{proof}
%      This is exactly \cite[Lem. 5.15]{Tia23}.
%  \end{proof}

  Now, for any $j\geq 0$, we consider the $A$-linear morphism $\rd_R:\widehat \Omega^j_{R^m}\to \widehat \Omega^{j+1}_{R^m}$ satisfying $\rd_{R}\wedge\rd_{R} = 0$ defined as follows: Write
  \[\widehat \Omega^j_{R^m} =\bigoplus_{k=0}^j \widehat \Omega^k_{R}\{-k\}\otimes_R\widehat \Omega^{j-k}_{R^m/R},\]
  and then for any $\omega_k\in \widehat \Omega^k_{R}\{-k\}$ and $\eta_{j-k}\in \widehat \Omega^{j-k}_{R^m/R}$, put
  \[\rd_R(\omega_k\otimes\eta_{j-k}) = p\rd(\omega_k)\otimes \eta_{j-k}+(-1)^k\omega_k\otimes\rd(\eta_{j-k}),\]
  where $p\rd:\widehat \Omega^k_{R}\{-k\}\to \widehat \Omega^{k+1}_{R}\{-k-1\}$ (resp. $\rd:\widehat \Omega^{j-k}_{R^m/R}\to \widehat \Omega^{j-k+1}_{R^m/R}$) is the usual $p$-differential (resp. relative differential) on $\widehat \Omega^k_{R}\{-k\}$ (resp. $\widehat \Omega^{j-k}_{R^m/R}$).
\begin{rmk}
    The differential $\rd_R$ defined above coincides with that in \cite[the paragraph above Lem. 5.15]{Tia23} (in the crystalline case). Indeed, by definition, $R^{\bullet}/p$ is exactly the cosimplicial ring $R_0(\bullet)$ in loc.cit.\footnote{The notation we adapt here is different from that in \cite{Tia23} (in the crystalline case). Here, we work with the smooth scheme $\Spec(R_0)$ over $A/p$, denote by $R$ the $p$-completely flat lifting of $R_0$ over $A$, and denote by $(R^{\bullet},(p))$ the cosimplicial prism representing the \v Cech nerve associated to the crystalline prism $(R,(p))\in (R_0/A)_{\Prism}$. However, in the crystalline case of \cite{Tia23}, the smooth scheme that we consider is denoted by $\Spec(R)$, the $p$-completely flat lifting of $R$ is denoted by $\widetilde R$, the cosimplicial prism representing the \v Cech nerve associated to the crystalline prism $(\widetilde R,(p))\in (R/A)_{\Prism}$ is denoted by $(\widetilde R(\bullet),(p))$, and the quotient $\widetilde R(\bullet)/p$ is denoted by $R(\bullet)$. To avoid confusion, in the rest of this remark, we use the notation $R_0(\bullet)$ to denote ``$R(\bullet)$'' in the sense of \cite{Tia23} (which is exactly the cosimplicial ring $R^{\bullet}/p$ in our setting).}. Thus, the reduction of $\widehat \Omega^1_{R^m}$ modulo $p$ coincides with $\Omega^1_{R_0(m)}$ (cf. the beginning of the last paragraph in \cite[Page 250]{Tia23}), and the reduction of $\widehat \Omega^j_{R^m}$ modulo $p$ coincides with $\Omega^j_{R_0(m)}$ (cf. \cite[(5.15.1)]{Tia23}). For any $0\leq k\leq j$, any $\omega_k\in \widehat \Omega^k_{R}\{-k\}$, and any $\eta_{j-k}\in \widehat \Omega^{j-k}_{R^m/R}$, by definition of our $\rd_R$, we have
    \[\rd_R(\sum_{k=0}^j\omega_k\otimes\eta_{j-k}) = p\big(\sum_{k=0}^j\rd(\omega_k)\otimes \eta_{j-k}\big)+\big(\sum_{k=0}^j(-1)^k\omega_k\otimes\rd(\eta_{j-k})\big).\]
    Modulo $p$, we have
    \[\rd_R(\sum_{k=0}^j\omega_k\otimes\eta_{j-k}) = \sum_{k=0}^j(-1)^k\omega_k\otimes\rd(\eta_{j-k})\]
    which is exactly the formula given in the first displayed equation below \cite[(5.15.1)]{Tia23}.
\end{rmk}

  By letting $m$ vary, for any $\bullet\geq 0$, we now have an $A$-linear morphism $\rd_R:\widehat \Omega^j_{R^{\bullet}}\to \widehat \Omega^{j+1}_{R^{\bullet}}$. The next lemma shows that $\rd_R$ is compatible with comsimplicial structures on its source and target, and thus is a morphism of cosimplicial $A$-modules.
  \begin{lem}\label{lem:d_R is cosimplicial}
      For any $j\geq 0$, the morphism $\rd_R:\widehat \Omega^j_{R^{\bullet}}\to \widehat \Omega^{j+1}_{R^{\bullet}}$ described above is a morphism of cosimplicial $A$-modules.
  \end{lem}
  \begin{proof}
      Let $f:\{0,1,\dots,m\}\to \{0,1,\dots,n\}$ be an order-preserving map of simplices. By abuse of notations, we also denote by $f$ the associated morphism $M^m\to M^n$ of a cosimplicial $R^{\bullet}$-module $M^{\bullet}$ (for example, $M^{\bullet}\in \{\widehat \Omega^j_{R^{\bullet}},\widehat \Omega^{j+1}_{R^{\bullet}}\}$).
      
      We remark that for any $a = a(\underline T,\underline X_1,\dots,\underline X_m)\in R^m$, we have
      \[f(a(\underline T,\underline X_1,\dots,\underline X_m)) = a(\underline T-p(1-\delta_{0,f(0)})\underline X_{f(0)},\underline X_{f(1)},\dots,\underline X_{f(m)}),\]
      where $\delta_{0,f(0)}$ is Kronecker's $\delta$-function.
      Then we see that
      \begin{equation}\label{equ:d_R is cosimplicial-I}
      \begin{split}
          f(p\rd(a)) = & pf(\sum_{i=1}^d\frac{\partial a}{\partial T_i}\otimes\frac{\rd T_{0,i}}{p})\\ 
          = & p\sum_{i=1}^d\frac{\partial f(a)}{\partial T_i}\otimes\frac{\rd T_{f(0),i}}{p} \\
          = & p\sum_{i=1}^d\frac{\partial f(a)}{\partial T_i}\otimes\frac{\rd T_{0,i}}{p}-(1-\delta_{0,f(0)})p\sum_{i=1}^d\frac{\partial f(a)}{\partial T_i}\otimes\rd X_{f(0),i}\\
          = & p\rd(f(a))-(1-\delta_{0,f(0)})p\sum_{i=1}^d\frac{\partial f(a)}{\partial T_i}\otimes\rd X_{f(0),i},
      \end{split}
      \end{equation}
      and that
      \begin{equation}\label{equ:d_R is cosimplicial-II}
          \begin{split}
              f(\rd(a)) = & f(\sum_{1\leq r\leq d,1\leq s\leq m}\frac{\partial a}{\partial X_{s,r}}\otimes \rd X_{s,r})\\
              = & \sum_{1\leq r\leq d,1\leq s\leq m}f(\frac{\partial a}{\partial X_{s,r}})\otimes \rd X_{f(s),r}\\
              = & \sum_{1\leq r\leq d,1\leq s\leq m}\frac{\partial f(a)}{\partial X_{f(s),r}}\otimes \rd X_{f(s),r}+p(1-\delta_{0,f(0)})\sum_{i=1}^d\frac{\partial f(a)}{\partial T_i}\otimes\rd X_{f(0),i}\\
              = & \rd(f(a))+p(1-\delta_{0,f(0)})\sum_{i=1}^d\frac{\partial f(a)}{\partial T_i}\otimes\rd X_{f(0),i}
          \end{split}
      \end{equation}

      Consider an element 
      \[x = a(\frac{\rd T_{0,i_1}}{p}\wedge\cdots\wedge\frac{\rd T_{0,i_k}}{p})\otimes(\rd X_{s_1,r_1}\wedge\cdots\wedge\rd X_{s_{j-k},r_{j-k}})\in \widehat \Omega^k_{R}\{-k\}\otimes_R\widehat \Omega^{j-k}_{R^m/R},\]
      where $a\in R^m$, $1\leq i_1<\cdots<i_k\leq d$ and $1\leq s_1\leq \cdots\leq s_{j-k}\leq m$. We have
      \begin{equation}\label{equ:d_R is cosimplicial-III}
          \begin{split}
          \rd_R(f(x))
          = & \rd_R(f(a)(\frac{\rd T_{f(0),i_1}}{p}\wedge\cdots\wedge\frac{\rd T_{f(0),i_k}}{p})\otimes(\rd X_{f(s_1),r_1}\wedge\cdots\wedge\rd X_{f(s_{j-k}),r_{j-k}}))\\
          = & (p\rd(f(a))\wedge\frac{\rd T_{f(0),i_1}}{p}\wedge\cdots\wedge\frac{\rd T_{f(0),i_k}}{p})\otimes(\rd X_{f(s_1),r_1}\wedge\cdots\wedge\rd X_{f(s_{j-k}),r_{j-k}})\\
          & + (-1)^k(\frac{\rd T_{f(0),i_1}}{p}\wedge\cdots\wedge\frac{\rd T_{f(0),i_k}}{p})\otimes(\rd(f(a))\wedge\rd X_{f(s_1),r_1}\wedge\cdots\wedge\rd X_{f(s_{j-k}),r_{j-k}}).
          \end{split}
      \end{equation}
      
      On the other hand, we have
      \begin{equation}\label{equ:d_R is cosimplicial-IV}
          \begin{split}
          f(\rd_R(x))
          = & f(p\rd(a)\wedge\frac{\rd T_{0,i_1}}{p}\wedge\cdots\wedge\frac{\rd T_{0,i_k}}{p})\otimes(\rd X_{s_1,r_1}\wedge\cdots\wedge\rd X_{s_{j-k},r_{j-k}}))\\
          &+  f((-1)^k(\frac{\rd T_{0,i_1}}{p}\wedge\cdots\wedge\frac{\rd T_{0,i_k}}{p})\otimes(\rd a\wedge\rd X_{s_1,r_1}\wedge\cdots\wedge\rd X_{s_{j-k},r_{j-k}}))\\
          = & (f(p\rd(a))\wedge\frac{\rd T_{f(0),i_1}}{p}\wedge\cdots\wedge\frac{\rd T_{f(0),i_k}}{p})\otimes(\rd X_{f(s_1),r_1}\wedge\cdots\wedge\rd X_{f(s_{j-k}),r_{j-k}})\\
          & + (-1)^k(\frac{\rd T_{f(0),i_1}}{p}\wedge\cdots\wedge\frac{\rd T_{f(0),i_k}}{p})\otimes(f(\rd(a))\wedge\rd X_{f(s_1),r_1}\wedge\cdots\wedge\rd X_{f(s_{j-k}),r_{j-k}}).
          \end{split}
      \end{equation}
      Now, one can conclude by combining equations (\ref{equ:d_R is cosimplicial-I}), (\ref{equ:d_R is cosimplicial-II}), (\ref{equ:d_R is cosimplicial-III}) and (\ref{equ:d_R is cosimplicial-IV}) together.
  \end{proof}

  Now, let $\bM\in \Vect((\frakX_0/A)_{\Prism},\calO_{\Prism,n})$ be a prismatic crystal of truncation $n$. Then $\bM(R^{\bullet},(p))$ is a cosimplicial $R^{\bullet}$-module. Let $q_0:R = R^0\to R^n$ be the map induced by the natural inclusion $\{0\}\subset\{0,1,\dots,n\}$ of simplices and $(M,\varepsilon_M = \exp(\sum_{i=1}^d\phi_iX_{1,i}))$ be the stratification with respect to the cosimplicial ring 
  \[R^{\bullet}/p^{n+1}\cong R[\underline X_1,\dots,\underline X_{\bullet}]^{\wedge}_{\pd}/p^{n+1}\]
  induced by $\bM$ in the sense of Corollary \ref{cor:stratification}. Then we have a canonical isomorphism 
  \begin{equation}\label{equ:q_0-isomorphism}
      M\otimes_{R,q_0}R^{\bullet}\cong \bM(R^{\bullet},(p)).
  \end{equation}
  Recall that for any $n\geq 0$, the tensor product $\bM(R^{\bullet},(p))\otimes_{R^{\bullet}}\widehat \Omega^n_{R^{\bullet}}$ is a cosimplicial $R^{\bullet}$-module, and we denote by
  \[\rd^{m,n}_1:=\sum_{i=0}^{m+1}(-1)^ip_i:\bM(R^{m},(p))\otimes_{R^{m}}\widehat \Omega^n_{R^{m}}\to \bM(R^{m+1},(p))\otimes_{R^{m+1}}\widehat \Omega^n_{R^{m+1}}\]
  the induced differential map on the totalization of $\bM(R^{\bullet},(p))\otimes_{R^{\bullet}}\widehat \Omega^n_{R^{\bullet}}$. 

  For any $m,n\geq 0$, we also define a map
  \[\rd_2^{m,n}:\bM(R^{m},(p))\otimes_{R^{m}}\widehat \Omega^n_{R^{m}}\to \bM(R^{m},(p))\otimes_{R^{m}}\widehat \Omega^{n+1}_{R^{m}}\]
  as the composite of the following maps
  \[\bM(R^{m},(p))\otimes_{R^{m}}\widehat \Omega^n_{R^{m}}\xrightarrow[(\ref{equ:q_0-isomorphism})]{\cong}M\otimes_{R,q_0}\widehat \Omega^n_{R^{m}}\to M\otimes_{R,q_0}\widehat \Omega^{n+1}_{R^{m}}\xrightarrow[(\ref{equ:q_0-isomorphism})]{\cong}\bM(R^{m},(p))\otimes_{R^{m}}\widehat \Omega^{n+1}_{R^{m}}.\]
  In the above, the first and last arrows are given, respectively, by the isomorphism \eqref{equ:q_0-isomorphism} in the following way:
  \begin{enumerate}
      \item[$\bullet$] $\bM(R^{m},(p))\otimes_{R^{m}}\widehat \Omega^n_{R^{m}}\cong M\otimes_{R,q_0}R^m\otimes_{R^{m}} \widehat \Omega^n_{R^{m}} = M\otimes_{R,q_0}\widehat \Omega^n_{R^{m}}$, and

      \item[$\bullet$] $M\otimes_{R,q_0}\widehat \Omega^{n+1}_{R^{m}} = M\otimes_{R,q_0}R^m\otimes_{R^{m}} \widehat \Omega^{n+1}_{R^{m}}\cong \bM(R^{m},(p))\otimes_{R^{m}}\widehat \Omega^{n+1}_{R^{m}}$.
  \end{enumerate}
  The middle arrow is the map such that for any $x\in M$ and any $\omega\in\widehat \Omega^n_{R^{m}}$, it sends $x\otimes\omega$ to $\nabla_M(x)\wedge\omega+x\otimes\rd_R(\omega)$. 
  In other words, via the above two isomorphisms, we have
  \[\rd_2^{m,n}(x\otimes \omega) = \nabla_M(x)\wedge\omega+x\otimes\rd_R(\omega).\]

  Now, for any $n\geq 0$, by letting $m$ vary, we have a morphism
  \[\rd_2^{\bullet,n}:\bM(R^{\bullet},(p))\otimes_{R^{\bullet}}\widehat \Omega^n_{R^{\bullet}}\to \bM(R^{\bullet},(p))\otimes_{R^{\bullet}}\widehat \Omega^{n+1}_{R^{\bullet}}\]
  for any $\bullet\geq 0$. Similar to Lemma \ref{lem:d_R is cosimplicial}, we have the following result.
  
  \begin{lem}\label{lem:bicomplex}
      For any $n\geq 0$, the map $\rd_2^{\bullet,n}:\bM(R^{\bullet},(p))\otimes_{R^{\bullet}}\widehat \Omega^n_{R^{\bullet}}\to \bM(R^{\bullet},(p))\otimes_{R^{\bullet}}\widehat \Omega^{n+1}_{R^{\bullet}}$ is a morphism of cosimplicial $A$-modules. In particular, we get a bicomplex $(\bM(R^{\bullet},(p))\otimes_{R^{\bullet}}\widehat \Omega^{\bullet}_{R^{\bullet}},\rd_1^{\bullet,\bullet},\rd_2^{\bullet,\bullet})$.
  \end{lem}

   Note that when $\bM = \calO_{\Prism}$, the above lemma reduces to Lemma \ref{lem:d_R is cosimplicial}.
  \begin{proof}
      Let $f:\{0,1,\dots,l\}\to \{0,1,\dots,m\}$ be an order-preserving map of simplices. By abuse of notations, we still denote by $f$ the induced map \[M\otimes_{R,q_0}\widehat \Omega^n_{R^{l}}\to M\otimes_{R,q_0}\widehat \Omega^{n+1}_{R^{m}}\]
      via the isomorphism (\ref{equ:q_0-isomorphism}). By construction of $f$, for any $x\in M$ and any $\omega\in\widehat \Omega^n_{R^l}$, we have
      \[f(x\otimes\omega) = x\otimes f(\omega)\]
      when $f(0) = 0$ and that
      \[f(x\otimes \omega) = p_{0,f(0)}^*(\varepsilon_M)(x)\otimes f(\omega)\]
      when $f(0)\neq 0$, where $p_{0,f(0)}: R^1\to R^m$ is the map induced by the order-preserving map $\{0,1\}\to\{0,1,\dots,m\}$ of simplices sending $0$ and $1$ to $0$ and $f(0)$ respectively, and $\varepsilon_M = \exp(\sum_{i=1}^d\phi_iX_{1,i})$ is the stratification on $M$ induced from $\bM$ via Corollary \ref{cor:stratification}. By Corollary \ref{cor:key}(1), in the $f(0)\neq 0$ case, we have
      \begin{equation}\label{equ:f(0) neq 0}
          f(x\otimes \omega) = \exp(\sum_{i=1}^d\phi_iX_{f(0),i})(x)\otimes f(\omega) = \sum_{\underline h\in\bN^d}\underline \phi^{[\underline h]}(x)\otimes\underline X_{f(0)}^{[\underline h]}f(\omega).
      \end{equation}
      
      To conclude, we have to show that $f\circ \rd_2^{l,n} = \rd_2^{m,n}\circ f$. Fix an $x\in M$ and an $\omega\in \widehat \Omega^n_{R^l}$.
      When $f(0) = 0$, we have
      \begin{equation*}
          \begin{split}
              \rd_2^{m,n}(f(x\otimes\omega)) = & \rd_2^{m,n}(x\otimes f(\omega)) = \nabla_M(x)\wedge f(\omega)+x\otimes\rd_R(f(\omega)).
          \end{split}
      \end{equation*}
      Note that $\nabla_M(x) = -\sum_{i=1}^d\phi_i(x)\otimes\frac{\rd T_i}{p}$ and recall we identify $T_i$ with $T_{0,i}$ in Corollary \ref{cor:stratification}. So it follows from $f(0) = 0$ that $f(\nabla_M(x)) = \nabla_M(x)$, which implies that
      \[\rd_2^{m,n}(f(x\otimes\omega))= f(\nabla_M(x))\wedge f(\omega)+x\otimes\rd_R(f(\omega)) = f(\nabla_M\wedge\omega+x\otimes\omega) = f(\rd_2^{l,n}(x\otimes \omega))\]
      as desired. Here, for the second equality, we use that $f$ commutes with $\rd_R$, as shown in Lemma \ref{lem:d_R is cosimplicial}. When $f(0)\neq 0$, then we have
      \begin{equation*}
          \begin{split}
              \rd_2^{m,n}(f(x\otimes\omega)) = &\rd_2^{m,n}(\exp(\sum_{i=1}^d\phi_iX_{f(0),i})(x)\otimes f(\omega))\\
              = & \rd_2^{m,n}(\sum_{\underline h\in\bN^d}\underline \phi^{\underline h}(x)\underline X_{f(0)}^{[\underline h]}\otimes f(\omega))\\
              = & -\sum_{i=1}^d\sum_{\underline h\in\bN^d}\phi_i\circ\underline \phi^{\underline h}(x)\otimes\frac{\rd T_i}{p}\wedge\underline X_{f(0)}^{[\underline h]} f(\omega) + \sum_{\underline h\in\bN^d}\underline \phi^{\underline h}(x)\otimes\rd_R(\underline X_{f(0)}^{[\underline h]} f(\omega))\\
              = & -\sum_{i=1}^d\sum_{\underline h\in\bN^d}\phi_i\circ\underline \phi^{\underline h}(x)\otimes\frac{\rd T_i}{p}\wedge\underline X_{f(0)}^{[\underline h]} f(\omega) + \sum_{\underline h\in\bN^d}\underline \phi^{\underline h}(x)\otimes\underline X_{f(0)}^{[\underline h]} f(\rd_R(\omega))\\
              &+\sum_{i=1}^d\sum_{\underline h\in\bN^d}\underline \phi^{\underline h}(x)\otimes\underline X_{f(0)}^{[\underline h-\underline 1_i]}\rd X_{f(0),i}\wedge f(\omega),
          \end{split}
      \end{equation*}
      where we also use that $f$ commutes with $\rd_R$ for the last equality. Thanks to (\ref{equ:identification on differential}), we have $f(\frac{\rd T_i}{p}) = \frac{T_{f(0),i}}{p}= \frac{\rd T_i}{p}-\rd X_{f(0),i}$. This combined with (\ref{equ:f(0) neq 0}) implies
      \[\begin{split}
          \rd_2^{m,n}(f(x\otimes\omega)) = & -\sum_{i=1}^df(\phi_i(x))\otimes\frac{\rd T_i}{p}\wedge f(\omega)+f(x\otimes\rd_R(x))+\sum_{i=1}^d\sum_{\underline h\in\bN^d}\underline \phi^{\underline h}\circ\phi_i(x)\otimes\underline X_{f(0)}^{[\underline h]}\rd X_{f(0),i}\wedge f(\omega)\\
          = & -\sum_{i=1}^df(\phi_i(x))\otimes\frac{\rd T_i}{p}\wedge f(\omega)+f(x\otimes\rd_R(x))+\sum_{i=1}^df(\phi_i(x))\rd X_{f(0),i}\wedge f(\omega)\\
          = & f(-\sum_{i=1}^d\phi_i(x)\otimes\frac{\rd T_i}{p}\wedge\omega+x\otimes\rd_R(x))\\
          = & f(\rd_2^{l,n}(x\otimes\omega))
      \end{split}\]
      as desired. This completes the proof.
  \end{proof}
    
    Now, we are able to construct $\rho_{\delta}(M)$ and prove that it is a quasi-isomorphism. 

  \begin{const}\label{const:rho_M}
    Note that we have the following observations:
      \begin{itemize}
        \item The complex $(\bM(R^{\bullet},(p)),\rd_1^{\bullet,0})$ is exactly the totalisation of the cosimplicial $A$-module $\bM(R^{\bullet},(p))$.

        \item The complex $(\bM(R^0,(p))\otimes_{R^0}\widehat \Omega^{\bullet}_{R^0},\rd_2^{0,\bullet})$ is exactly the de Rham complex $\rD\rR(M,\nabla_M)$.
    \end{itemize}
    So if we denote by $\Tot_{\delta}(\bM)$ the totalization of the bicomplex $(\bM(R^{\bullet},(p))\otimes_{R^{\bullet}}\widehat \Omega^{\bullet}_{R^{\bullet}},\rd_1^{\bullet,\bullet},\rd_2^{\bullet,\bullet})$, then there are canonical inclusions $\rho_{M}^{\prime}$ and $\rho_{M}^{\prime\prime}$:
    \[\bM(R^{\bullet},(p))\xrightarrow{\rho_{M}^{\prime}}\Tot(\bM(R^{\bullet},(p))\otimes_{R^{\bullet}}\widehat \Omega^{\bullet}_{R^{\bullet}},\rd_1^{\bullet,\bullet},\rd_2^{\bullet,\bullet})\xleftarrow{\rho_{M}^{\prime\prime}}\rD\rR(M,\nabla_M).\]
    We will see from Lemma \ref{lem:construction of rho} below that both $\rho_{M}^{\prime}$ and $\rho_{M}^{\prime\prime}$ are quasi-isomorphisms of complexes of $A$-modules. So the morphism
    \[\rho(M) := \rho_{M}^{\prime\prime-1}\circ\rho_{M}^{\prime}:\bM(R^{\bullet},(p))\to \rD\rR(M,\nabla_M)\]
    is a well-defined quasi-isomorphism in the derived category of $A$-modules as well.
  \end{const}
    
   \begin{lem}\label{lem:construction of rho}
        Fix an $n\in\bN_*$. Let $\bM$ be a prismatic crystal of truncation $n$ with the associated $p$-connection $(M,\nabla_M)$ in the sense of Proposition \ref{prop:local equivalence}.
        Then both $\rho_{M}^{\prime}$ and $\rho_{M}^{\prime\prime}$ in Construction \ref{const:rho_M} are quasi-isomorphisms in the derived category of $A$-modules, which are functorial in $\bM$ and thus $(M,\nabla_M)$.
    \end{lem}
    \begin{proof}
        The functoriality of $\rho_{M}^{\prime}$ and $\rho_M^{\prime\prime}$ follows from their constructions. It remains to prove they are quasi-isomorphisms. In the rest, we denote by $\Tot_{\delta}(\bM)$ the complex $\Tot(\bM(R^{\bullet},(p))\otimes_{R^{\bullet}}\widehat \Omega^{\bullet}_{R^{\bullet}},\rd_1^{\bullet,\bullet},\rd_2^{\bullet,\bullet})$ for short as before.

        It suffices to deal with the case for $n<\infty$. Indeed, as explained in Remark \ref{rmk:devissage}, for any $\bM\in \Vect((R_0/A)_{\Prism},\calO_{\Prism})$ with corresponding $p$-connection $(M,\nabla_M)$, its reduction $\bM_n:=\bM/p^{n+1}$ is a prismatic crystal of truncation $n$ whose associated $p$-connection in the sense of Proposition \ref{prop:local equivalence} is $(M_n = M/p^{n+1},\nabla_{M_n})$. Note that
        \[\bM(R^{\bullet},(p)) = \lim_n\bM_n(R^{\bullet},(p)) \text{ and }\Tot_{\delta}(\bM) = \lim_n\Tot_{\delta}(\bM_n).\]
        We can conclude $\rho_M$ is a quasi-isomorphism by using that $\rho_M^{\prime} = \lim_n\rho_{M_n}^{\prime}$ and that $\rho_{M_n}^{\prime}$ is a quasi-isomorphism for any $n\geq 0$. A similar argument shows that $\rho_M^{\prime\prime}$ is a quasi-isomorphism if so is $\rho_{M_n}^{\prime\prime}$ for any $n\geq 0$.

        Now, we deal with the case for $n<\infty$. When $n = 0$, i.e. $\bM$ is a Hodge--Tate crystal, we can conclude by \cite[Thm. 5.14 and its proof]{Tia23}. Now, suppose that the result holds true for any $0\leq m<n+1$. We are going to show it also holds true for $n+1$. 
        
        Let $\bM$ be a prismatic crystal of truncation $n+1$ with associated $p$-connection $(M,\nabla_M)$. We only prove $\rho_M^{\prime}$ is a quasi-isomorphism as the case for $\rho^{\prime\prime}_M$ follows from the same argument. By Remark \ref{rmk:devissage}, $p\bM$ and $\bM/p$ are prismatic crystals of truncation $n$ and $0$, respectively with associated respective $p$-connections $(pM,\nabla_{pM})$ and $(M/p,\nabla_{M/p})$. Then we have the following commutative diagram of complexes of $A$-modules:
        \begin{equation*}
            \xymatrix@C=0.4cm{
              0\ar[r] &(p\bM)(R^{\bullet},p)\ar[d]^{\rho_{pM}^{\prime}}\ar[r]&\bM(R^{\bullet},p)\ar[d]^{\rho_{M}^{\prime}}\ar[r]&(\bM/p)(R^{\bullet},p)\ar[d]^{\rho_{M/p}^{\prime}}\ar[r]& 0\\
              0\ar[r] &\Tot_{\delta}(p\bM)\ar[r]&\Tot_{\delta}(\bM)\ar[r]&\Tot_{\delta}(\bM/p)\ar[r]& 0.
            }
        \end{equation*}
        By the construction of $\Tot_{\delta}(\bM)$, we see that 
        \[\Tot_{\delta}(p\bM) = p\Tot_{\delta}(\bM)\text{ and }\Tot_{\delta}(\bM/p) = \Tot_{\delta}(\bM)/p\]
        and thus in the above diagram, the top and bottom are both short exact sequences. By inductive hypothesis, $\rho^{\prime}_{pM}$ and $\rho^{\prime}_{M/p}$ are both quasi-isomorphisms, yielding that $\rho_M^{\prime}$ is also a quasi-isomorphism as desired.
    \end{proof}
    \begin{proof}[\textbf{Proof of Proposition \ref{prop:cohomology comparison}:}]
    As argued in the paragraph below Proposition \ref{prop:cohomology comparison}, we need to construct a quasi-isomorphism
    \[\rR\Gamma((\frakX_0/A)_{\Prism},\bM)\simeq \rD\rR(M,\nabla_M).\]
    By \cite[Prop. 4.10]{Tia23}, as $(R,(p))$ is a covering of the final object of $\Sh((R/A)_{\Prism})$ (cf. Lemma \ref{lem:cover}), there exists a quasi-isomorphism
    \[\iota_{\delta}(\bM):\rR\Gamma((\frakX_0/A)_{\Prism},\bM)\xrightarrow{\simeq}\bM(R^{\bullet},(p)).\]
    
      Let $\rho_{\delta}(M):\rR\Gamma((\frakX_0/A)_{\Prism},\bM)\to \rD\rR(M,\nabla_M)$ be the composite
      \[\rR\Gamma((\frakX_0/A)_{\Prism},\bM)\xrightarrow{\iota_{\delta}(\bM)}\bM(R^{\bullet},(p))\xrightarrow{\rho(M)} \rD\rR(M,\nabla_M).\]
      Then it is a quasi-isomorphism by construction. As both $\rho(M)$ and $\iota_{\delta}(\bM)$ are functorial in $\bM$, compatible with \'etale localization of $\frakX_0 = \Spf(R_0)$, and only depend on the $\delta$-structure on $R$, again by their constructions, we know these also hold true for $\rho_{\delta}(\bM)$. This completes the proof.
    \end{proof}

\section{Prismatic crystal as $p$-connection: Globalization}\label{sec:proof of main theorem}
  Let $(A,(p))$ be a crystalline prism as before.
  In this section, we always assume $\frakX_0$ is a smooth scheme over $A_0$ of dimension $d$ with the absolute Frobenius $\rF_{\frakX_0}:\frakX_0\to\frakX_0$. We assume $(\frakX_0,\rF_{\frakX_0})$ admits a lifting $(\frakX_1,\rF_{\frakX_1})$ over $A_1$.

  The purpose of this section is to prove Theorem \ref{thm:main} (cf. \S\ref{ssec:proof-I} and  \S\ref{ssec:proof-II}) by showing that when $(\frakX_0,\rF_{\frakX_0})$ admits a lifting $(\frakX_1,\rF_{\frakX_1})$ over $A_1$, the local constructions in Proposition \ref{prop:compare local equivalence} and Proposition \ref{prop:cohomology comparison} are independent of the (suitable) choice of $\delta$-structures.

\subsection{The compatibility of local equivalence}\label{ssec:compare local equivalence}
  Now, we assume $\frakX_0 = \Spec(R_0)$ is affine with a chart. Let $\frakX = \Spf(R)$ be a smooth lifting of $\frakX_0$ to $A$ and $\frakX_n = \frakX\times_{\Spf(A)}\Spec(A_n)$ for any $n\geq0$. We are going to show if two $\delta$-structures $\delta_1$ and $\delta_2$ on $R$ coincide modulo $p$ (equivalently, Assumption \ref{assumpsion:key} holds true for $\delta_1$ and $\delta_2$), then the equivalences in Proposition \ref{prop:compare local equivalence} with respect to $\delta_i$'s also coincide. From now on, we keep assumptions and notations of Corollary \ref{cor:key}.

  Fix an $n\in\bN_* = \bN\cup\{\infty\}$.
  We first study the stratifications with respect to the cosimplicial ring 
  \[\widetilde R^{\bullet}/p^{n+1} \cong R_{12}[\underline Z_1,\dots,\underline Z_{\bullet}]^{\wedge}_{\pd}/p^{n+1}.\]
  Here and in what follows, for a $p$-adically complete ring $S$, we put $S/p^{n+1}:=S$ when $n = \infty$.
  
  Let $(N,\varepsilon_{N})$ be a pair of a finite projective $R_{12}/p^n$-module and an $R_{12}[\underline Z_1]^{\wedge}_{\pd}$-linear isomorphism 
  \[\varepsilon_{N}:N\otimes_{R_{12},p_0}R_{12}[\underline Z_1]^{\wedge}_{\pd}\to N\otimes_{R_{12},p_1}R_{12}[\underline Z_1]^{\wedge}_{\pd}.\]
  Similar to \S\ref{sec:local construction}, we may assume the restriction of $\varepsilon_{N}$ to $N$ is of the form 
  \[\varepsilon_{N} = \sum_{\underline m\in\bN^d}\theta_{\underline m}\underline Z_1^{[\underline m]},\]
  where $\theta_{\underline m}\in\End_{R_1}(N)$ satisfying $\lim_{|\underline m|\to+\infty}\theta_{\underline m} = 0$. Using (\ref{equ:face-II}), in this case, we still have
  \begin{equation*}
      \begin{split}
          &p_2^*(\varepsilon_N)\circ p_0^*(\varepsilon_N) = \sum_{\underline l,\underline k,\underline m\in\bN^d}(-1)^{|\underline l|}\theta_{\underline k}\circ\theta_{\underline l+\underline m}\underline Z_1^{[\underline k]}\underline Z_1^{[\underline l]}\underline Z_2^{[\underline m]},\\
          &p_1^*(\varepsilon_N) = \sum_{\underline m\in\bN^d}\theta_{\underline m}\underline Z_2^{[\underline m]},\\
          &\sigma_0^*(\varepsilon_N) = \theta_{\underline 0}.
      \end{split}
  \end{equation*}
  By a similar argument as in \S\ref{sec:local construction}, the following result is true.
  \begin{lem}\label{lem:stratification for comparing local data}
      For any $1\leq i\leq d$, put $\phi_i = \theta_{\underline 1_i}\in \End_{A}(N)$. Then the following two statements are equivalent.
      \begin{enumerate}
          \item[(1)] The pair $(N,\varepsilon_{N} = \sum_{\underline m\in\bN^d}\theta_{\underline m}\underline Z_1^{[\underline m]})$ is an object in $\Strat(\widetilde R^{\bullet}/p^{n+1})$.

          \item[(2)] The morphisms $\phi_i$'s are topologically quasi-nilpotent and commute with each other such that for any $\underline m = (m_1,\dots,m_d)\in\bN^d$, we have 
          \[\theta_{\underline m} = \underline \phi^{\underline m}:=\phi_1^{m_1}\cdots\phi_d^{m_d}.\]
          In particular, we have $\varepsilon_N = \exp(\sum_{i=1}^d\phi_iZ_{1,i})$.
      \end{enumerate}
  \end{lem}
  \begin{proof}
      It follows from the same proof for Corollary \ref{cor:stratification}.
  \end{proof}

  Fix an object $\bM\in \Vect((\frakX_0/A)_{\Prism},\calO_{\Prism,n})$. For each $*\in \{1,2\}$, let $(M_*,\varepsilon_*)\in \Strat(R^{\bullet}_*/p^{n+1})$ be the stratification induced by $\bM$ via Lemma \ref{lem:crystal vs stratification} and $(M_{*},\nabla_{*} = -\sum_{i=1}^d\phi_{*,i}\otimes\frac{\rd T_i}{p})\in \MIC_p^{\rm tqn}(\frakX_n)$ be the $p$-connection induced by $\bM$ via Proposition \ref{prop:local equivalence} with respect to the $\delta$-structure $\delta_*$ on $R$ as in \S\ref{sec:local construction}.
  \begin{prop}\label{prop:compare local equivalence}
      Suppose Assumption \ref{assumpsion:key} holds for $\delta_1$ and $\delta_2$. Then there is a natural isomorphism $\iota_{12}:M_1\xrightarrow{\cong} M_2$ satisfying the following conditions:
      \begin{enumerate}
          \item[(1)] The isomorphism $\iota_{12}$ is compatible with connections, i.e. the following diagram is commutative
          \[\xymatrix@C=0.5cm{
            M_1\ar[rr]^{\iota_{12}}\ar[d]_{\nabla_1}&&M_2\ar[d]^{\nabla_2}\\
            M_1\otimes_{R}\widehat \Omega^1_{R/A}\{-1\}\ar[rr]^{\iota_{12}\otimes\id}&&M_2\otimes_{R}\widehat \Omega^1_{R/A}\{-1\},
          }\]
          and $\iota_{12}$ also admits an inverse $\iota_{21}:M_2\to M_1$ which is also compatible with connections.

          \item[(2)] Let $\delta_3$ be another $\delta$-structure on $R$ with induced prism $(R_3,(p))$ and the induced $p$-connection $(M_3,\nabla_3)$. Let $\iota_{*3}:M_*\xrightarrow{\cong} M_3$ be the corresponding natural isomorphism for each $*\in\{1,2\}$. Then we have $\iota_{13} = \iota_{23}\circ\iota_{12}$.
      \end{enumerate}
  \end{prop}
  \begin{proof}
      We first construct the desired $\iota_{12}$.
      Recall by Corollary \ref{cor:crystal vs stratification}, we have a canonical identification
      \[M_1 = \lim_{\Delta}\bM(\widetilde R^{\bullet},(p)).\]
      We identify $N:=\bM(R_{12}=\widetilde R^0,(p))$ with $M_2\otimes_{R_2}R_{12}$. Then it carries the induced stratification 
      \[\varepsilon_N: N\otimes_{R_{12},p_0}R_{12}[\underline Z_1]^{\wedge}_{\pd}\cong M_2\otimes_{R_2}R_2[\underline Y,\underline Z_1]^{\wedge}_{\pd }\xrightarrow{\varepsilon_2\otimes 1}M_2\otimes_{R_2}R_2[\underline Y,\underline Z_1]^{\wedge}_{\pd}\cong  N\otimes_{R_{12},p_1}R_{12}[\underline Z_1]^{\wedge}_{\pd}\]
      such that $(N,\varepsilon_N)\in\Strat(\widetilde R^{\bullet}/p^{n+1})$.
      Consider the automorphism 
      \[\iota_{12}:N= M_2\otimes_{R_2}R_{12} \xrightarrow{\exp(-\sum_{i=1}^d\phi_{2,i}Y_{i})} M_2\otimes_{R_2}R_{12} = N.\]
    We claim that it makes the following diagram of morphisms of $\widetilde R^1$-modules
\begin{equation}\label{diag:descent}
          \xymatrix@C=0.5cm{
(M_2\otimes_{R_2}R_{12})\otimes_{R_{12},p_0}\widetilde R^1\ar@{=}[r]\ar[d]^{\iota_{12}\otimes_{p_0}\id}&M_2\otimes_{R_2}\widetilde R^1\ar[rr]^{\id}&&M_2\otimes_{R_2}\widetilde R^1\ar@{=}[r]&
(M_2\otimes_{R_2}R_{12})\otimes_{R_{12},p_1}\widetilde R^1\ar[d]^{\iota_{12}\otimes_{p_1}\id }\\
(M_2\otimes_{R_2}R_{12})\otimes_{R_{12},p_0}\widetilde R^1\ar[rrrr]^{\varepsilon_N}&&&&(M_2\otimes_{R_2}R_{12})\otimes_{R_{12},p_1}\widetilde R^1
          }
      \end{equation}
      commute. In other words, we have to show that for any $x\in M_2$ (regarded as the element $x\otimes1\otimes1$ in $M_2\otimes_{R_2}R_{12}\otimes_{R_{12},p_0}\widetilde R^1$), 
      \[(\iota_{12}\otimes_{p_1}\id)(x) = \varepsilon_N((\iota_{12}\otimes_{p_0}\id)(x)).\]
      Recall by \eqref{equ:face-II}, we have that $p_0(\underline Y) = \underline Y+\underline Z_1$ and $p_1(\underline Y) = \underline Y$. By the definition of $\iota_{12}$, we have 
      \[(\iota_{12}\otimes_{p_1}\id)(x) = \exp(-\sum_{i=1}^d\phi_{2,i}Y_i)(x)\]
      and 
      \[(\iota_{12}\otimes_{p_0}\id)(x) = \exp(-\sum_{i=1}^d\phi_{2,i}(Y_i+ Z_{1,i}))(x).\]
      By Lemma \ref{lem:stratification for comparing local data} and the construction of $\varepsilon_N$, we have $\varepsilon_N = \exp(\sum_{i=1}^d\phi_{2,i}Z_{1,i})$, and thus
      \[\varepsilon_N((\iota_{12}\otimes_{p_0}\id)(x)) = \exp(\sum_{i=1}^d\phi_{2,i}Z_{1,i})\big(\exp(-\sum_{i=1}^d\phi_{2,i}(Y_i+ Z_{1,i}))(x)\big) = \exp(-\sum_{i=1}^d\phi_{2,i}Y_i)(x),\]
      yielding that $(\iota_{12}\otimes_{p_1}\id)(x) = \varepsilon_N((\iota_{12}\otimes_{p_0}\id)(x))$ as desired. 

    Now, as diagram \eqref{diag:descent} is commutative, the canonical identifications
      
      \begin{equation}\label{equ:two term comparison}
          M_1\otimes_{R_1}R_{12} \cong \bM(R_{12},(p))\cong M_2\otimes_{R_2}R_{12},
      \end{equation}
      induce a canonical isomorphism
      \begin{equation}\label{equ:delta1 vs delta2-I}
          \iota_{12} = \exp(-\sum_{i=1}^d\phi_{2,i}\frac{T_i-S_i}{p}): M_1\to M_2.
      \end{equation}
      It remains to prove that this isomorphism $\iota_{12}$ is desired.
      
      For Item (1): By the symmetry between $\delta_1$ and $\delta_2$ (equivalently, exchanging $1$ (resp. $T_i$'s) and $2$ (resp. $S_i$'s) in (\ref{equ:delta1 vs delta2-I})), the canonical identifications in (\ref{equ:two term comparison}) induce a canonical isomorphism
      \[\iota_{21} = \exp(-\sum_{i=1}^d\phi_{1,i}\frac{S_i-T_i}{p}):M_2\to M_1\]
      which is indeed the inverse of $\iota_{12}$ above. So we can deduce from $\iota_{12}\circ\iota_{21} = \id_{M_1}$ that $\phi_{1,i} = \phi_{2,i}$ for any $1\leq i\leq d$. In other words, the isomorphism $\iota_{12}$ and its inverse $\iota_{21}$ are compatible with connections as expected.

      For Item (2): Now, assume $R$ is equiped with another $\delta$-structure $\delta_3$ with the induced prism $(R_3,(p))\in (R_0/A)_{\Prism}$. Let $\underline U$ be the coordinate corresponding to the given chart on $R_3$ and $(M_3,\nabla_3 = -\sum_{i=1}^d\phi_{3,i}\otimes\frac{\rd T_i}{p})$ be the $p$-connection induced by $\bM$ via the equivalence in Proposition \ref{prop:local equivalence}. Then Item (1) implies that $\phi_{1,i} = \phi_{2,i} = \phi_{3,i}$, which we shall denote by $\phi_i$ for any $1\leq i\leq d$. Let $(R_{123},(p))$ be the coproduct of prisms
      \[(R_1,(p))\times(R_2,(p))\times(R_3,(p)).\]
      Then the above argument implies that \begin{enumerate}
          \item[$\bullet$] $\iota_{12} = \exp(-\sum_{i=1}^d\phi_{i}\frac{T_i-S_i}{p}): M_1\to M_2$,

          \item[$\bullet$] $\iota_{13} = \exp(-\sum_{i=1}^d\phi_{i}\frac{T_i-U_i}{p}): M_1\to M_3$, 

          \item[$\bullet$] $\iota_{23} = \exp(-\sum_{i=1}^d\phi_{i}\frac{S_i-U_i}{p}): M_2\to M_3$.
      \end{enumerate}
      Now, we can directly check $\iota_{13} = \iota_{23}\circ\iota_{12}$ as desired.
  \end{proof}

\subsection{The proof of Theorem \ref{thm:main}(1)}\label{ssec:proof-I}
\begin{proof}[\textbf{Proof of Theorem \ref{thm:main}}(1)]
    Recall that for any $n\in\bN_*$, when $\frakX_0$ admits a smooth lifting $\frakX_n$ over $A_n$, there is a canonical identification of sites $\frakX_{0,\et}\simeq \frakX_{n,\et}$.

    Now, assume $\frakX_0$ lifts to a formally smooth formal scheme over $A_n$ for some $n\geq 1$ (which is allowed to be $\infty$) and $\rF_{\frakX_0}$ admits a lifting $\rF_{\frakX_1}$ over $A_1$. Fix such a lifting $\frakX_n$. Let $\{\frakU_{0,i} = \Spec(R_{0,i})\to\frakX_0\}_{i\in I}$ be an \'etale covering of $
    \frakX_0$ with each $\frakU_{0,i}$ affine smooth with a fixed chart. Let $\frakU_{1,i}=\Spf(R_{1,i})\to\frakX_1$ be the lifting of the \'etale map $\frakU_{0,1}\to \frakX_0$ over $A_1$ induced from the given lifting $\frakX_1$, and let $\varphi_{1,i}:R_{1,i}\to R_{1,i}$ be the morphism over $\phi_A:A_1\to A_1$ determined by the given $\rF_{\frakX_1}$.
    For each $i\in I$, fix a lifting $R_i$ of $R_{1,i}$ over $A$ and equip it with a $\delta$-structure $\delta_i$ of which the induced Frobenius map lifts the $\varphi_{i,1}$ above such that the induced Frobenius map on $R_i$ is a lifting of $\varphi_{1,i}$ on $R_{1,i}$. Thus for each $i\in I$, we get a prism $(R_i,(p))\in(\frakX_0/A)_{\Prism}$. 
    
    Put $\frakU_{0,ij} = \frakU_{0,i}\times_{\frakX_0}\frakU_{0,j}$. As $\frakX_0$ is separated, $\frakU_{0,ij}$ is also affine. Thus it is of the form $\frakU_{0,i} = \Spec(R_{0,ij})$ such that the natural maps $R_{0,i}\to R_{0,ij}$ and $R_{0,j}\to R_{0,ij}$ are \'etale. Let $R_{ij}$ be a $p$-completely smooth lifting of $R_{0,ij}$ over $A$. By \'etaleness, for each $*\in \{i,j\}$, there is a unique $p$-completely \'etale morphism $R_*\to R_{ij}$ lifting $R_{0,*}\to R_{0,ij}$ such that the $\delta$-structure $\delta_*$ on $R_*$ uniquely extends to a $\delta$-structure, still denoted by $\delta_i$ for simplicity, on $R_{ij}$ (cf. \cite[Lem. 2.18]{BS22}). By construction, these two $\delta$-structures on $R_{ij}$ satisfy Assumption \ref{assumpsion:key}; that is, we have $\delta_i\equiv\delta_j\mod p^2$.

     Now, the desired equivalence follows from Proposition \ref{prop:compare local equivalence} as the following presheaves of categories on $\frakX_{0,\et}\simeq\frakX_{n,\et}$
    \[\frakU_0\in\frakX_{0,\et}\mapsto\Vect((\frakU_0/A)_{\Prism},\calO_{\Prism,n})\]
    and 
    \[\frakU_0\in\frakX_{0,\et}\mapsto\MIC^{\rm tqn}_p(\frakU_n)\]
    are sheaves for \'etale topology, where $\frakU_n\to \frakX_n$ is the lifting of $\frakU_0\to \frakX_0$ uniquely induced by the given lifting $\frakX_n$ of $\frakX_0$.
\end{proof}

\subsection{The compatiblility of cohomologies}\label{ssec:compatibility of cohomology}
  As in \S\ref{ssec:compare local equivalence}, we again assume $\frakX_0 = \Spec(R_0)$ is affine with a chart. Let $\frakX = \Spf(R)$ be a smooth lifting of $\frakX_0$ to $A$ and $\frakX_n = \Spf(R_n) = \frakX\times_{\Spf(A)}\Spf(A_n)$ for any $n\in\bN_*=\bN\cup\{\infty\}$. We want to assign to each $\bM\in\Vect((\frakX_0/A)_{\Prism},\calO_{\Prism,n})$ a \emph{canonical} quasi-isomorphism 
  \[\rR\Gamma((\frakX_0/A)_{\Prism},\bM)\simeq \rR\Gamma(\frakX_{n,\et},\rD\rR(\calM,\nabla_{\calM}))\]
  in the sense that it \emph{only} depends on the Frobenius endormorphism of $R_1$,
  where $(\calM,\nabla_{\calM})$ is the $p$-connection induced from $\bM$ via Theorem \ref{thm:main}(1).

  \begin{conv}\label{conv:R_Sigma}
      Let $\Sigma = \{\delta_1,\dots,\delta_h\}$ be a set of $\delta$-structures on $R$ satisfying 
      \[(\ast)\qquad\delta_1\equiv\delta_2\equiv\cdots\equiv\delta_h\mod p.\]
      Let $(R_i,(p))\in(R_0/A)_{\Prism}$ be the prism induced by $\delta_i$ as in the paragraph above Lemma \ref{lem:cover}, on which the coordinate is denoted by $\underline T^{(i)}$, i.e. $T_1^{(i)},\dots,T_d^{(i)}$.
      Let 
      \[(R_{\Sigma},(p)) = (R_1,(p))\times\cdots\times(R_h,(p))\]
      be the coproduct of $(R_i,(p))$'s in $(R_0/A)_{\Prism}$. Write $\underline Y_{i}:=\frac{\underline T^{(1)}-\underline T^{(i+1)}}{p}$ for any $1\leq i\leq h-1$. By Proposition \ref{prop:key},  $R_{\Sigma}$ is an $(R_1\widehat \otimes_A\cdots\widehat \otimes_AR_h)$-algebra such that for any $1\leq j\leq h$, we have an isomorphism of $R_j$-algebras
      \[R_{\Sigma}  \cong R_j[\underline Y_i\mid 1\leq i\leq h-1]^{\wedge}_{\pd}.\]
      Using Lemma \ref{lem:cover}, it is easy to see that $(R_{\Sigma},(p))$ is a covering of the final object of $\Sh((R_0/A)_{\Prism})$.
      For further use, we also define
      \[\widehat \Omega^1_{R_{\Sigma}}\{-1\} = \oplus_{i=1}^hR_{\Sigma}\otimes_{R_i}\widehat \Omega^1_{R_i}\{-1\}.\]
      Then for any $1\leq r\leq d$ and any $1\leq i\leq h-1$, we have 
      \[\frac{\rd T^{(i+1)}_r}{p} = \frac{\rd T^{(1)}_r}{p}-\rd Y^{(i)}_r\] 
      and there is a natural derivation
      \begin{equation}\label{equ:d_Sigma}
          \rd_{\Sigma}:R_{\Sigma}\to \widehat \Omega^1_{R_{\Sigma}}\{-1\}.
      \end{equation}
      Note that for any $1\leq r\leq d$ and any $1\leq i\leq h-1$, we have 
      \[\rd_{\Sigma}(Y_{r}^{(i)}) =\rd Y_{r}^{(i)} = \frac{\rd T^{(1)}_r}{p}-\frac{\rd T^{(i+1)}_r}{p} \text{ and } \rd_{\Sigma}(T_{r}^{(1)}) = p\rd T_r^{(1)}.\]
  \end{conv}
  \begin{lem}\label{lem:coproduct of R_Sigma}
      Keep notations of Convention \ref{conv:R_Sigma}.
      Let $(R_{\Sigma}^{\bullet},(p))$ be the \v Cech nerve associated to $(R_{\Sigma},(p))$. For any $m\geq k\geq 1$, let $\underline T_{k-1}^{(i)}$ be the coordinate on the $k$-th $(R_i,(p))$ in $(R^m_{\Sigma},(p))$. For any $1\leq j\leq m$, let $\underline X_j^{(i)}:=\frac{\underline T_0^{(i)}-\underline T_j^{(i)}}{p}$. Then there is an isomorphism of cosimplicial rings
      \[R^{\bullet}_{\Sigma} \cong R_{\Sigma}[\underline X_{1}^{(i)},\dots,\underline X_{\bullet}^{(i)}\mid 1\leq i\leq h]^{\wedge}_{\pd},\]
      where the $R_{\Sigma}$-linear structures on both sides are induced from the first component of $R_{\Sigma}^{\bullet}$.
       Moreover, via the above isomorphism, the face maps $p_i$'s and the degeneracy maps $\sigma_i$'s on 
       \[R_{\Sigma}[\underline X_{1}^{(i)},\dots,\underline X_{\bullet}^{(i)}\mid 1\leq i\leq h]^{\wedge}_{\pd}\]
       satisfy the following conditions.
       \begin{enumerate}
           \item[(1)] The degeneracy maps $\sigma_i$'s are $R_{\Sigma}$-linear such that for any $1\leq l\leq h$ and for each $j$,
           \[
              \begin{split}
                  \sigma_i(\underline X_j^{(l)}) = \left\{
                  \begin{array}{rcl}
                      0, & (i,j) = (0,1) \\
                      \underline X_{j-1}^{(l)}, & i < j \text{ and }(i,j)\neq (0,1)\\
                      \underline X_j^{(l)}, & i\geq j.
                  \end{array}
                  \right.
              \end{split}
          \]

          \item[(2)] When $i\neq 0$, the face maps $p_i$'s are $R_{\Sigma}$-linear such that for any $1\leq l\leq h$ and for each $j$,
          \begin{equation*}
              \begin{split}
                  p_i(\underline X^{(l)}_j) = \left\{
                  \begin{array}{rcl}
                      \underline X^{(l)}_{j+1}, & 0<i\leq j\\
                      \underline X^{(l)}_j, & i>j
                  \end{array}
                  \right.
              \end{split}
          \end{equation*}

          \item[(3)] For any $1\leq l\leq h$ and for each $j$, we have
          \begin{equation*}
              p_0(\underline T_0^{(l)}) = \underline T_0^{(l)}-p\underline X_1^{(l)} \text{ and } p_0(\underline X_j^{(l)}) = \underline X^{(l)}_{j+1}-\underline X^{(l)}_1.
          \end{equation*}
          In particular, for any $1\leq k\leq h-1$, we have
          \[p_0(\underline Y_k) = \underline Y_k-\underline X_1^{(1)}+\underline X_1^{(k+1)}.\]
       \end{enumerate}
  \end{lem}
  \begin{proof}
      The desired isomorphism follows from an argument similar to the proof of Proposition \ref{prop:key}. The items (1), (2) and (3) follow from the definitions of $\underline T_j^{(i)}$'s.
  \end{proof}
  
  As before, let $M_{\Sigma}$ be a finite projective $R_{\Sigma}/p^{n+1}$-module with an $R_{\Sigma}^1$-linear isomorphism
  \[\varepsilon_{M_{\Sigma}}:M_{\Sigma}\otimes_{R_{\Sigma},p_0}R_{\Sigma}^1\to M_{\Sigma}\otimes_{R_{\Sigma},p_1}R_{\Sigma}^1.\]
  Here, $n\in\bN_*=\bN\cup\{\infty\}$ and for any $p$-complete $R$-algebra $S$ (for example, $S = R_{\Sigma}$),  $S/p^{n+1}:=S$ when $n=\infty$.
  We want to find a necessary and sufficient condition for $(M_{\Sigma},\varepsilon_{M_{\Sigma}})$ to be an object of $\Strat(R_{\Sigma}^{\bullet}/p^{n+1})$.
  Write 
  \[\varepsilon_{M_{\Sigma}} = \sum_{\underline i_1,\dots,\underline i_n\in\bN^d}\theta_{\underline i_1,\dots,\underline i_h}(\underline X^{(1)}_1)^{[\underline i_1]}\cdots(\underline X^{(h)}_1)^{[\underline i_h]}.\]
  As in \S\ref{ssec:local equivalence}, we have that
  \begin{align*}
      &\begin{aligned}
          \bullet \quad p_2^*(\varepsilon_{M_{\Sigma}})\circ p_0^*(\varepsilon_{M_{\Sigma}}) = {} & \sum_{\underline i_1,\underline j_1,\dots,\underline i_{h},\underline j_{h}\in\bN^d} \theta_{\underline j_1,\dots,\underline j_{h}}\circ\underline \theta_{\underline i_1,\dots,\underline i_{h}}\prod_{t=1}^{h}((\underline X^{(t)}_1)^{[\underline j_t]}(\underline X^{(t)}_2-\underline X^{(t)}_1)^{[\underline i_t]})\\
          = & \sum_{\underline i_1,\underline j_1,\underline k_1,\dots,\underline i_{h},\underline j_{h},\underline k_{h}\in\bN^d}\theta_{\underline j_1,\dots,\underline j_{h}}\circ\theta_{\underline i_1+\underline k_1,\dots,\underline i_{h}+\underline k_{h}}\prod_{t=1}^{h}((-1)^{|\underline i_t|}(\underline X^{(t)}_1)^{[\underline i_t]}(\underline X^{(t)}_1)^{[\underline j_t]}(\underline X^{(t)}_2)^{[\underline k_t]}),
      \end{aligned}\\
      &\begin{aligned}
          \bullet\quad p_1^*(\varepsilon_{M_{\Sigma}}) = \sum_{\underline k_1,\dots,\underline k_{h}\in\bN^d}\theta_{\underline k_1,\dots,\underline k_{h}}(\underline X^{(1)}_1)^{[\underline k_1]}\cdots(\underline X^{({h})}_2)^{[\underline k_{h}]},
      \end{aligned}\\
      &\begin{aligned}
          \bullet\quad\sigma_0^*(\varepsilon_{M_{\Sigma}}) = \theta_{\underline 0,\dots,\underline 0}.
      \end{aligned}
  \end{align*}
  Similar to Lemma \ref{lem:p_2p_0=p_1}, we have the following result.
  \begin{lem}\label{lem:p_2p_0=p_1 R_Sigma}
      Keep notations as above. The following are equivalent:
      \begin{enumerate}
          \item[(1)] We have $(M_{\Sigma},\varepsilon_{M_{\Sigma}}=\sum_{\underline i_1,\dots,\underline i_h\in\bN^d}\theta_{\underline i_1,\dots,\underline i_n}(\underline X^{(1)}_1)^{[\underline i_1]}\cdots(\underline X^{(h)}_1)^{[\underline i_h]})\in\Strat(R^{\bullet}_{\Sigma}/p^{n+1})$.

          \item[(2)] There are topologically quasi-nilpotent $\{\phi_i^{(t)}\}_{1\leq i\leq d,1\leq t\leq h}$ in $\End_A(M_{\Sigma})$ commuting with each others such that for any $\underline i_1,\dots,\underline i_h\in\bN^d$, we have
          $\theta_{\underline i_1,\dots,\underline i_h} = (\underline \phi^{(1)})^{\underline i_1}\cdots(\underline \phi^{(h)})^{\underline i_h}$, where for any $1\leq j\leq h$ and $\underline i_{j} = (i_{j,1},\dots,i_{j,d})\in \bN^d$, 
          \[(\underline \phi^{(j)})^{\underline i_j} = (\phi_1^{(j)})^{i_{j,1}}\cdots (\phi_d^{(j)})^{i_{j,d}}.\]
          In particular, we have
          \[\varepsilon_{M_{\Sigma}} = \exp(\sum_{t=1}^h\sum_{i=1}^d\phi_i^{(t)}X_{1,i}^{(t)}).\]
      \end{enumerate}
  \end{lem}
  \begin{proof}
     The result follows from the same argument for the proof of Lemma \ref{lem:solve p_2p_0=p_1}.
  \end{proof}
    Keep notations of Lemma \ref{lem:p_2p_0=p_1 R_Sigma}. For any $1\leq t\leq h$, we define 
    \begin{equation}\label{equ:nabla_Sigma^t}
        \nabla_{M_{\Sigma}}^{(t)}:=-\sum_{i=1}^d\phi_i^{(t)}\otimes\frac{\rd T_i^{(t)}}{p}:M_{\Sigma}\to M_{\Sigma}\otimes_{R_t}\widehat \Omega^1_{R_t}\{-1\}
    \end{equation}
    and define 
    \begin{equation}\label{equ:nabla_Sigma}
        \nabla_{M_{\Sigma}}:=(\nabla_{M_{\Sigma}}^{(1)},\dots,\nabla_{M_{\Sigma}}^{(h)}):M_{\Sigma}\to \oplus_{t=1}^nM_{\Sigma}\otimes_{R_t}\widehat \Omega^1_{R_t}\{-1\} = M_{\Sigma}\otimes_{R_{\Sigma}}\widehat \Omega_{R_{\Sigma}}^1\{-1\}.
    \end{equation}
    As Lemma \ref{lem:stratification vs p-connection}, we will see $(M_{\Sigma},\nabla_{M_{\Sigma}})$ is a \emph{topologically quasi-nilpotent flat connection with respect to $\rd_{\Sigma}$}. Let us first introduce the following definition.
    \begin{dfn}\label{dfn:p-connection wrt d_Sigma}
        By a \emph{connection with respect to $\rd_{\Sigma}$} (\ref{equ:d_Sigma}) over $R_{\Sigma}/p^{n+1}$, we mean a finite projective $R_{\Sigma}/p^{n+1}$-module $M_{\Sigma}$ together with an $A$-linear morphism 
        \[\nabla_{M_{\Sigma}}=(\nabla_{M_{\Sigma}}^{(1)},\dots,\nabla_{M_{\Sigma}}^{(h)}):M_{\Sigma}\to \oplus_{t=1}^hM_{\Sigma}\otimes_{R_t}\widehat \Omega^1_{R_t}\{-1\} = M_{\Sigma}\otimes_{R_{\Sigma}}\widehat \Omega_{R_{\Sigma}}^1\{-1\}\]
        satisfying the following \emph{Leibniz rule}: For any $f\in R_{\Sigma}$ and any $x\in M_{\Sigma}$, we have
        \[\nabla_{M_{\Sigma}}(fx) = f\nabla_{M_{\Sigma}}(x)+\rd_{\Sigma}(f)\otimes x.\]
        We say a connection $(M_{\Sigma},\nabla_{M_{\Sigma}})$ with respect to $\rd_{\Sigma}$ is \emph{flat}, if $\nabla_{M_{\Sigma}}\wedge\nabla_{M_{\Sigma}} = 0$. A flat connection $(M_{\Sigma},\nabla_{M_{\Sigma}})$ is called \emph{topologically quasi-nilpotent}, if $\lim_{m\to+\infty}\nabla_{M_{\Sigma}}^m = 0$.
        When $(M_{\Sigma},\nabla_{M_{\Sigma}})$ is flat, we denote by $\rD\rR(M_{\Sigma},\nabla_{M_{\Sigma}})$ its induced \emph{de Rham complex}:
    \[M_{\Sigma}\xrightarrow{\nabla_{M_{\Sigma}}}M_{\Sigma}\otimes_{R_{\Sigma}}\widehat \Omega^1_{R_{\Sigma}}\{-1\}\xrightarrow{\nabla_{M_{\Sigma}}}\cdots \xrightarrow{\nabla_{M_{\Sigma}}}M_{\Sigma}\otimes_{R_{\Sigma}}\widehat \Omega^{hd}_{R_{\Sigma}}\{-hd\}.\]
    Denote by $\MIC^{\rm tqn}(R_{\Sigma}/p^{n+1},\rd_{\Sigma})$ the category of topologically quasi-nilpotent flat connections over $R_{\Sigma}/p^{n+1}$ with respect to $\rd_{\Sigma}$.
    \end{dfn}

    \begin{rmk}\label{rmk:p-connection is connection}
        When $\Sigma = \{\delta\}$, a connection with respect to $\rd_{\Sigma}$ over $R_{\Sigma}/p^{n+1}$ is exactly a $p$-connection with respect to $\rd$ over $R/p^{n+1}$ as described in Definition \ref{dfn:p-connection}. Indeed, in this case, we have $R_{\Sigma} = R$, $\widehat \Omega^1_{R_{\Sigma}}\{-1\} = \widehat \Omega^1_R\{-1\}$, and $\rd_{\Sigma} = p\rd:R\to \widehat \Omega^1_{R}\{-1\}$. So we get the desired compatibility by noting that a connection with respect to $p\rd$ is exactly a $p$-connection with respect to $\rd$.
    \end{rmk}
    
  \begin{lem}\label{lem:connection with respect to d_Sigma}
      Keep notations as above. For any $(M_{\Sigma},\varepsilon_{M_{\Sigma}})\in \Strat(R_{\Sigma}^{\bullet}/p^{n+1})$, the pair $(M_{\Sigma},\nabla_{M_{\Sigma}})$ defined by (\ref{equ:nabla_Sigma^t}) and (\ref{equ:nabla_Sigma}) is a topologically quasi-nilpotent flat connection with respect to $\rd_{\Sigma}$ over $R_{\Sigma}/p^{n+1}$. 
  \end{lem}
  \begin{proof}
      Let $J_{\pd}$ be the kernel of the degeneracy map $\sigma_0:R^1_{\Sigma}\to R_{\Sigma}$. Then it is the $p$-complete pd-ideal of $R_{\Sigma}^1$ generated by $X^{(t)}_{\underline 1}$'s. As in Construction \ref{const:intrinsic}, there is an isomorphism of $R_{\Sigma}$-modules
      \[J_{\pd}/J_{\pd}^{[2]}\cong \widehat \Omega^1_{R_{\Sigma}}\{-1\}.\]
      If we denote by $\iota_{M_{\Sigma}}:M_{\Sigma}\to M_{\Sigma}\otimes_{R_{\Sigma},p_1}R_{\Sigma}^1$ the map sending $x$ to $x\otimes 1$, then we still have that $\nabla_{M_{\Sigma}}$ coincides with the map
      \[\overline{\iota_{M_{\Sigma}}-\varepsilon_{M_{\Sigma}}}:M_{\Sigma}\to M_{\Sigma}\otimes_{R_{\Sigma},p_1}J_{\pd}/J_{\pd}^{[2]}\cong M_{\Sigma}\otimes_{R_{\Sigma}}\widehat \Omega^1_{R_{\Sigma}}\{-1\}.\]
      Now, the lemma follows from the same argument for the proof of Lemma \ref{lem:stratification vs p-connection}.
  \end{proof}
  \begin{cor}\label{cor:connection with respect to d_Sigma}
      There are equivalence of categories
      \[\Vect((\frakX_0/A)_{\Prism},\calO_{\Prism,n})\simeq \Strat(R_{\Sigma}^{\bullet}/p^{n+1})\simeq \MIC^{\rm tqn}(R_{\Sigma}/p^{n+1},\rd_{\Sigma}).\]
  \end{cor}
  \begin{proof}
      As $(R_{\Sigma},(p))$ is a covering of the final object of $\Sh((\frakX_0/A)_{\Prism})$, the first equivalence follows from \cite[Prop. 4.8]{Tia23}. The second equivalence follows from Lemma \ref{lem:p_2p_0=p_1 R_Sigma} together with Lemma \ref{lem:connection with respect to d_Sigma}.
  \end{proof}
  \begin{rmk}[Compatibility with $\Sigma$]\label{rmk:functorial in Sigma}
      Let $\Sigma$ and $\Sigma^{\prime}$ be two sets of $\delta$-structures satisfying the condition $(\ast)$ in Convention \ref{conv:R_Sigma}. Assume $\Sigma$ is a subset of $\Sigma^{\prime}$. Then there is a canonical morphism of cosimiplicial $A$-algebras $R_{\Sigma}^{\bullet}\to R_{\Sigma^{\prime}}^{\bullet}$. Note that for any $(M_{\Sigma},\varepsilon_{M_{\Sigma}})\in \Strat(R_{\Sigma}^{\bullet}/p^{n+1})$ with the associated $\bM\in \Vect((\frakX_0/A)_{\Prism},\calO_{\Prism,n})$ via the equivalence in Corollary \ref{cor:connection with respect to d_Sigma}, the base-change
      \[(M_{\Sigma^{\prime}}:=M_{\Sigma}\otimes_{R_{\Sigma}}R_{\Sigma^{\prime}},\varepsilon_{M^{\prime}}:=\varepsilon_{M_{\Sigma^{\prime}}}\otimes\id_{R_{\Sigma^{\prime}}^1})\in \Strat(R_{\Sigma^{\prime}}^{\bullet}/p^{n+1})\]
      is associated to $\bM$ via the equivalence in Corollary \ref{cor:connection with respect to d_Sigma} (for $\Sigma^{\prime}$ instead of $\Sigma$). So if $(M_{\Sigma},\nabla_{M_{\Sigma}})\in \MIC^{\rm tqn}(R_{\Sigma}/p^{n+1},\rd_{\Sigma})$ is the connection with respect to $\rd_{\Sigma}$ associated to $\bM$, then 
      \[(M_{\Sigma^{\prime}} = M_{\Sigma}\otimes_{R_{\Sigma}}R_{\Sigma^{\prime}},\nabla_{M_{\Sigma^{\prime}}} = \nabla_{M_{\Sigma}}\otimes\id_{R_{\Sigma}^{\prime}}+\id_{M_{\Sigma}}\otimes\rd_{\Sigma^{\prime}})\] 
      is the object in $\MIC^{\rm tqn}(R_{\Sigma^{\prime}}/p^{n+1},\rd_{\Sigma^{\prime}})$ associated to $\bM$ as well. 
  \end{rmk}

  Now, we follow the strategy in \S\ref{ssec:cohomological comparison} to show there is an explicitly constructed quasi-isomorphism of complexes of $A$-modules
  \[\rho_{\Sigma}(\bM):\rR\Gamma((\frakX_0/A)_{\Prism},\bM)\xrightarrow{\simeq} \rD\rR(M_{\Sigma},\nabla_{M_{\Sigma}})\]
  for any $\bM\in \Vect((\frakX_0/A)_{\Prism},\calO_{\Prism,n})$ with the associated $(M_{\Sigma},\nabla_{M_{\Sigma}})\in \MIC^{\rm tqn}(R_{\Sigma}/p^{n+1},\rd_{\Sigma})$.

  For any $m\geq 1$, define 
  \[\widehat \Omega^1_{R^m_{\Sigma}}:=\widehat \Omega^1_{R_{\Sigma}}\{-1\}\otimes_{R_{\Sigma}}R_{\Sigma}^m\oplus\widehat \Omega^1_{R_{\Sigma}^m/R_{\Sigma}}\]
  and then we get $\widehat \Omega^1_{R^{\bullet}_{\Sigma}}$ by letting $m$ vary.
  The same argument in the paragraph around (\ref{equ:identification on differential}) implies that $\widehat \Omega^1_{R_{\Sigma}^{\bullet}}$ is a cosimiplicial $R_{\Sigma}^{\bullet}$-module. For any $j\geq 1$, let $\widehat \Omega^j_{R_{\Sigma}^{\bullet}}$ be the $j$-fold wedge product of $\widehat \Omega^1_{R_{\Sigma}^{\bullet}}$.
  \begin{lem}\label{lem:zero-homotopy}
      For any $j\geq 1$, the cosimiplicial $R_{\Sigma}^{\bullet}$-module $\widehat \Omega^j_{R_{\Sigma}^{\bullet}}$ is homotopic to zero.
  \end{lem}
  \begin{proof}
      It suffices to prove the $j=1$ case.
      For any $1\leq t\leq n$, let $(R_t^{\bullet},(p))$ be the cosimplicial prism representing the \v Cech nerve associated to the covering $(R_t,(p))$ of the final object of $\Sh((R_0/A)_{\Prism})$. Then there is a canonical morphism $R_t^{\bullet}\to R_{\Sigma}^{\bullet}$ of comsimplicial $A$-algebras.
      By construction, we have an isomorphism of cosimplicial $R_{\Sigma}^{\bullet}$-modules
      \[\oplus_{t=n}^d\widehat \Omega^1_{R_t^{\bullet}}\otimes_{R^{\bullet}_t}R_{\Sigma}^{\bullet}\cong \widehat \Omega^1_{R_{\Sigma}^{\bullet}}.\]
      As each $\widehat \Omega^1_{R_t^{\bullet}}$ is homotopic to zero by \cite[Lem. 5.15]{Tia23}, so is $\widehat \Omega^1_{R_{\Sigma}^{\bullet}}$. 
  \end{proof}
  Now, for any $j\geq 0$, we consider the $A$-linear morphism $\rd_{R_{\Sigma}}:\widehat \Omega^j_{R_{\Sigma}^m}\to \widehat \Omega^{j+1}_{R_{\Sigma}^m}$ satisfying $\rd_{R_{\Sigma}}\wedge\rd_{R_{\Sigma}} = 0$ defined as follows: Write
  \[\widehat \Omega^j_{R_{\Sigma}^m} =\bigoplus_{k=0}^j \widehat \Omega^k_{R_{\Sigma}}\{-k\}\otimes_{R_{\Sigma}}\widehat \Omega^{j-k}_{R_{\Sigma}^m/R_{\Sigma}},\]
  and then for any $\omega_k\in \widehat \Omega^k_{R_{\Sigma}}\{-k\}$ and $\eta_{j-k}\in \widehat \Omega^{j-k}_{R_{\Sigma}^m/R_{\Sigma}}$, put
  \[\rd_{R_{\Sigma}}(\omega_k\otimes\eta_{j-k}) = \rd_{\Sigma}(\omega_k)\otimes \eta_{j-k}+(-1)^k\omega_k\otimes\rd(\eta_{j-k}),\]
  where $\rd:\widehat \Omega^{j-k}_{R_{\Sigma}^m/R_{\Sigma}}\to \widehat \Omega^{j-k+1}_{R_{\Sigma}^m/R_{\Sigma}}$ is the usual  differential on $\widehat \Omega^{j-k}_{R_{\Sigma}^m/R_{\Sigma}}$ and $\rd_{\Sigma}$ is given by (\ref{equ:d_Sigma}).
  We remark that the $\rd_{R_{\Sigma}}$ coincides with $\rd_{R}$ in \S\ref{ssec:cohomological comparison} when $\Sigma = \{\delta\}$ (cf. Remark \ref{rmk:p-connection is connection}).
  \begin{lem}\label{lem:d_Sigma is cosimplicial}
      For any $j\geq 1$, the $\rd_{R_{\Sigma}}:\widehat \Omega^j_{R_{\Sigma}^{\bullet}}\to \widehat \Omega^{j+1}_{R_{\Sigma}^{\bullet}}$ is a morphism of cosimplicial $A$-modules.
  \end{lem}
  \begin{proof}
    The result follows from the same argument in the proof of Lemma \ref{lem:d_R is cosimplicial}.
  \end{proof}

  Now, let $\bM\in \Vect((\frakX_0/A)_{\Prism},\calO_{\Prism,n})$ be a prismatic crystal of truncation $n$. Then $\bM(R_{\Sigma}^{\bullet},(p))$ is a cosimplicial $R_{\Sigma}^{\bullet}$-module. Let $q_0:R_{\Sigma} = R_{\Sigma}^0\to R_{\Sigma}^n$ be the map induced by the natural inclusion $\{0\}\subset\{0,1,\dots,n\}$ of simplices. Let $(M_{\Sigma},\varepsilon_{M_{\Sigma}} = \exp(\sum_{t=1}^{h}\sum_{i=1}^d\phi^{(t)}_iX^{(t)}_{1,i}))$ be the stratification corresponding to $\bM$ in the sense of Lemma \ref{lem:p_2p_0=p_1 R_Sigma}. Then we have a canonical isomorphism 
  \begin{equation}\label{equ:q_0-isomorphism R_Sigma}
      M_{\Sigma}\otimes_{R_{\Sigma},q_0}R_{\Sigma}^{\bullet}\cong \bM(R_{\Sigma}^{\bullet},(p)).
  \end{equation}
  For any $n\geq 0$, consider the cosimplicial $R^{\bullet}_{\Sigma}$-module $\bM(R_{\Sigma}^{\bullet},(p))\otimes_{R_{\Sigma}^{\bullet}}\widehat \Omega^n_{R_{\Sigma}^{\bullet}}$ and denote by \[\rd^{m,n}_1:=\sum_{i=0}^{m+1}(-1)^ip_i:\bM(R_{\Sigma}^{m},(p))\otimes_{R_{\Sigma}^{m}}\widehat \Omega^n_{R^{m}}\to \bM(R_{\Sigma}^{m+1},(p))\otimes_{R_{\Sigma}^{m+1}}\widehat \Omega^n_{R_{\Sigma}^{m+1}}\]
  the induced differential map on the totalization of $\bM(R_{\Sigma}^{\bullet},(p))\otimes_{R_{\Sigma}^{\bullet}}\widehat \Omega^n_{R_{\Sigma}^{\bullet}}$. Let
  \[\rd_2^{m,n}:\bM(R_{\Sigma}^{m},(p))\otimes_{R_{\Sigma}^{m}}\widehat \Omega^n_{R_{\Sigma}^{m}}\to\bM(R_{\Sigma}^{m},(p))\otimes_{R_{\Sigma}^{m}}\widehat \Omega^{n+1}_{R_{\Sigma}^{m}}\]
  be the composite of the following maps:
  \[\bM(R_{\Sigma}^{m},(p))\otimes_{R_{\Sigma}^{m}}\widehat \Omega^n_{R_{\Sigma}^{m}}\xrightarrow[(\ref{equ:q_0-isomorphism R_Sigma})]{\cong}M_{\Sigma}\otimes_{R_{\Sigma},q_0}\widehat \Omega^n_{R_{\Sigma}^{m}}\to M_{\Sigma}\otimes_{R_{\Sigma},q_0}\widehat \Omega^{n+1}_{R_{\Sigma}^{m}}\xrightarrow[(\ref{equ:q_0-isomorphism R_Sigma})]{\cong}\bM(R_{\Sigma}^{m},(p))\otimes_{R_{\Sigma}^{m}}\widehat \Omega^{n+1}_{R_{\Sigma}^{m}}.\]
  In the above, the first and last arrows are respectively given by the isomorphism \eqref{equ:q_0-isomorphism R_Sigma} in the following way:
  \begin{enumerate}
      \item[$\bullet$] $\bM(R_{\Sigma}^{m},(p))\otimes_{R_{\Sigma}^{m}}\widehat \Omega^n_{R_{\Sigma}^{m}}\cong M_{\Sigma}\otimes_{R_{\Sigma},q_0}R^m_{\Sigma}\otimes_{R_{\Sigma}^{m}}\widehat \Omega^n_{R_{\Sigma}^{m}}=M_{\Sigma}\otimes_{R_{\Sigma},q_0}\widehat \Omega^n_{R_{\Sigma}^{m}}$,

      \item[$\bullet$] $M_{\Sigma}\otimes_{R_{\Sigma},q_0}\widehat \Omega^{n+1}_{R_{\Sigma}^{m}} = M_{\Sigma}\otimes_{R_{\Sigma},q_0}R^m_{\Sigma}\otimes_{R_{\Sigma}^{m}}\widehat \Omega^{n+1}_{R_{\Sigma}^{m}}\cong\bM(R_{\Sigma}^{m},(p))\otimes_{R_{\Sigma}^{m}}\widehat \Omega^{n+1}_{R_{\Sigma}^{m}}$.
  \end{enumerate}
  The middle arrow is the map such that for any $x\in M_{\Sigma}$ and any $\omega\in\widehat \Omega^n_{R_{\Sigma}^{m}}$, it sends $x\otimes\omega$ to $\nabla_{M_{\Sigma}}(x)\wedge\omega+x\otimes\rd_{R_{\Sigma}}(\omega)$.
  In other words, via the above two isomorphisms, for any $x\in M_{\Sigma}$ and any $\omega\in\widehat \Omega^n_{R_{\Sigma}^{m}}$, we have
  \begin{equation}\label{equ:d_2 M_Sigma}
      \rd_2^{m,n}(x\otimes \omega) = \nabla_{M_{\Sigma}}(x)\wedge\omega+x\otimes\rd_{R_{\Sigma}}(\omega).
  \end{equation}
  Fix an $n\geq 0$. By letting $m$ vary, we get a map
  \[\rd_2^{\bullet,n}:\bM(R_{\Sigma}^{\bullet},(p))\otimes_{R_{\Sigma}^{\bullet}}\widehat \Omega^n_{R_{\Sigma}^{\bullet}}\to \bM(R_{\Sigma}^{\bullet},(p))\otimes_{R_{\Sigma}^{\bullet}}\widehat \Omega^{n+1}_{R_{\Sigma}^{\bullet}}.\]
  
  \begin{lem}\label{lem:bicomplex R_Sigma}
      The map $\rd_2^{\bullet,n}:\bM(R_{\Sigma}^{\bullet},(p))\otimes_{R_{\Sigma}^{\bullet}}\widehat \Omega^n_{R_{\Sigma}^{\bullet}}\to \bM(R_{\Sigma}^{\bullet},(p))\otimes_{R_{\Sigma}^{\bullet}}\widehat \Omega^{n+1}_{R_{\Sigma}^{\bullet}}$ described as above is a morphism of cosimplicial $A$-modules. In particular, by letting $n$ vary, we get a bicomplex $(\bM(R_{\Sigma}^{\bullet},(p))\otimes_{R_{\Sigma}^{\bullet}}\widehat \Omega^{\bullet}_{R_{\Sigma}^{\bullet}},\rd_1^{\bullet,\bullet},\rd_2^{\bullet,\bullet})$.
  \end{lem}
  \begin{proof}
      This follows from the same argument for the proof of Lemma \ref{lem:bicomplex}.
  \end{proof}

  Note that in this case we still have that 
    \begin{itemize}
        \item The complex $(\bM(R_{\Sigma}^{\bullet},(p)),\rd_1^{\bullet,0})$ is exactly the totalization of  $\bM(R_{\Sigma}^{\bullet},(p))$.

        \item The complex $(\bM(R_{\Sigma},(p))\otimes_{R_{\Sigma}}\widehat \Omega^{\bullet}_{R_{\Sigma}},\rd_2^{0,\bullet})$ is exactly the de Rham complex $\rD\rR(M_{\Sigma},\nabla_{M_{\Sigma}})$.
    \end{itemize}
    So if we denote by $\Tot_{\Sigma}(\bM)$ the totalization of the bicomplex $(\bM(R_{\Sigma}^{\bullet},(p))\otimes_{R_{\Sigma}^{\bullet}}\widehat \Omega^{\bullet}_{R_{\Sigma}^{\bullet}},\rd_1^{\bullet,\bullet},\rd_2^{\bullet,\bullet})$, then we get two natural morphisms $\rho_{M_{\Sigma}}^{\prime}$ and $\rho_{M_{\Sigma}}^{\prime\prime}$ of complexes of $A$-modules:
    \[\bM(R_{\Sigma}^{\bullet},(p))\xrightarrow{\rho_{M_{\Sigma}}^{\prime}}\Tot_{\Sigma}(\bM)\xleftarrow{\rho_{M_{\Sigma}}^{\prime\prime}}\rD\rR(M_{\Sigma},\nabla_{M_{\Sigma}}).\]

    \begin{lem}\label{lem:rho_Sigma is quasi-isomorphism}
        Keep notations as above. Then both $\rho_{M_{\Sigma}}^{\prime}$ and $\rho_{M_{\Sigma}}^{\prime\prime}$ are quasi-isomorphisms of complexes of $A$-modules. In particular, we have a quasi-isomorphism in the derived category of $A$-modules
        \[\rho_{\Sigma}(M_{\Sigma}):=(\rho_{M_{\Sigma}}^{\prime\prime})^{-1}\circ \rho_{M_{\Sigma}}^{\prime}:\bM(R_{\Sigma}^{\bullet},(p))\xrightarrow{\simeq} \rD\rR(M_{\Sigma},\nabla_{M_{\Sigma}}).\]
    \end{lem}
    \begin{proof}
        For $\rho_{M_{\Sigma}}^{\prime}$: By Lemma \ref{lem:zero-homotopy}, for any $j\geq 1$, the cosimplicial $A$-module $\bM(R_{\Sigma}^{\bullet},(p))\otimes_{R_{\Sigma}^{\bullet}}\widehat 
        \Omega^1_{R^{\bullet}_{\Sigma}}$ is homotopic to zero. Therefore, the natural morphism $(\bM(R_{\Sigma}^{\bullet},(p)),\rd_1^{\bullet,0})\to\Tot_{\Sigma}(\bM)$ is a quasi-isomorphism in the derived category of $A$-modules. By definition, this is exactly $\rho_{M_{\Sigma}}^{\prime}$.
    
        For $\rho_{M_{\Sigma}}^{\prime\prime}$: By the same d\'evissage argument as in the proof of Lemma \ref{lem:construction of rho}, we reduce to the case $n=0$. In this case, the $R_0$-linear structure on $R_{\Sigma}/p$ induced from the natural morphism $R_t/p\to R_{\Sigma}/p$ is independent of the choice of $1\leq t\leq h$ (as $\underline T^{(t)}\equiv \underline T^{(1)}\mod pR_{\Sigma}$ for all $t$). In particular, $\varepsilon_{M_{\Sigma}}$ is $R_0$-linear. Therefore, $\nabla_{M_{\Sigma}}$ is also $R_0$-linear and thus so is $\rd_2^{\bullet,\bullet}$.

        Next, we claim that for any $0\leq i\leq m$, the morphism $q_i:\{0\}\xrightarrow{0\mapsto i}\{0,1,\dots,m\}$ induces an isomorphism of complexes
        \[\rD\rR(M_{\Sigma},\nabla_{M_{\Sigma}}) = (\bM(R_{\Sigma}^{\bullet},(p)),\rd_2^{0,\bullet}) \xrightarrow{\simeq}(\bM(R^m_{\Sigma},(p))\otimes_{R_{\Sigma}^m}\widehat \Omega^{\bullet}_{R_{\Sigma}^m},\rd_2^{m,\bullet}).\]
        Indeed, as $\nabla_{M_{\Sigma}}$ is $R_0$-linear, by the definition of $d_2^{\bullet,\bullet}$ (\ref{equ:d_2 M_Sigma}), we see that $(\bM(R^m_{\Sigma},(p))\otimes_{R_{\Sigma}^m}\widehat \Omega^{\bullet}_{R_{\Sigma}^m},\rd_2^{m,\bullet})$ is the tensor product of complexes
        \[(\bM(R^m_{\Sigma},(p))\otimes_{R_{\Sigma}^m}\widehat \Omega^{\bullet}_{R_{\Sigma}^m},\rd_2^{m,\bullet}) = \rD\rR(M_{\Sigma},\nabla_{M_{\Sigma}})\otimes(\widehat \Omega^{\bullet}_{R^{m}_{\Sigma}/R_{\Sigma}},\rd).\]
        As $R_{\Sigma}^m$ is the $p$-adic completion of a free pd-algebra over $R_{\Sigma}$, the claim follows from the same argument in the proof of \cite[Lem. 5.17]{Tia23}.

        Thanks to the above claim, the same argument for the proof of \cite[Thm. 5.14]{Tia23} (i.e. the paragraphs above \cite[Rem. 5.18]{Tia23}) implies that via the quasi-isomorphisms in the claim, the complex $\Tot_{\Sigma}(\bM)$ is quasi-isomorphic to the totalization of the complex
        \[\rD\rR(M_{\Sigma},\nabla_{M_{\Sigma}})\xrightarrow{0}\rD\rR(M_{\Sigma},\nabla_{M_{\Sigma}})\xrightarrow{\id}\rD\rR(M_{\Sigma},\nabla_{M_{\Sigma}})\xrightarrow{0}\rD\rR(M_{\Sigma},\nabla_{M_{\Sigma}})\to\cdots,\]
        and thus is quasi-isomorphic to $\rD\rR(M_{\Sigma},\nabla_{M_{\Sigma}})$ as desired.
    \end{proof}

    \begin{prop}\label{prop:local compatibility of cohomology}
        Fix a lifting $\varphi_{R_1}:R_1\to R_1$ of the absolute Frobenius map $\varphi_{R_0}:R_0\to R_0$ over $\phi_{A_1}:A_1\to A_1$.
        For any $\bM\in \Vect((\frakX_0/A)_{\Prism},\calO_{\Prism,n})$ with corresponding $(\calM,\nabla_{\calM})\in \MIC_p^{\rm tqn}(\frakX_n)$, there exists a quasi-isomorphism in the derived category of $A$-modules
        \[\rho(\bM):\rR\Gamma((\frakX_0/A)_{\Prism},\bM)\simeq \rR\Gamma(\frakX_{n,\et},\rD\rR(\calM,\nabla_{\calM}))\]
        which is functorial in $\bM$ and \emph{canonical} in the sense that it only depends on the lifting $\varphi_{R_1}$.
    \end{prop}
    \begin{proof}
        Let $\frakS$ be the set of all $\delta$-structures on $R$ whose induced Frobenius endomorphism of $R$ coincides with the given $\varphi_{R_1}$ modulo $p^2$. For any non-empty finite subset $\Sigma\subset\frakS$, let $(R_{\Sigma}^{\bullet},(p))$ be as above. By \cite[Prop. 4.10]{Tia23}, there is a quasi-isomorphism of complexes of $A$-modules
        \[\iota_{\Sigma}(\bM):\rR\Gamma((\frakX_0/A)_{\Prism},\bM)\xrightarrow{\simeq}\bM(R_{\Sigma}^{\bullet},(p)).\]
        Let $(M_{\Sigma},\nabla_{M_{\Sigma}})$ be the connection with respect to $\rd_{\Sigma}$ associated to $\bM$ via the equivalence in Corollary \ref{cor:connection with respect to d_Sigma}.
        Define 
        \[\rho_{\Sigma}(\bM):=\rho_{\Sigma}(M_{\Sigma})\circ\iota_{\Sigma}(\bM):\rR\Gamma((\frakX_0/A)_{\Prism},\bM)\xrightarrow{\simeq}\rD\rR(M_{\Sigma},\nabla_{M_{\Sigma}}),\]
        which is a quasi-isomorphism in the derived category of $A$-modules by Lemma \ref{lem:rho_Sigma is quasi-isomorphism}. By Remark \ref{rmk:functorial in Sigma}, for any $\Sigma^{\prime}\in\frakS$ satisfying $\Sigma\subset\Sigma^{\prime}$, we have a natural morphism of complexes of $A$-modules
        \[i_{\Sigma,\Sigma^{\prime}}:\rD\rR(M_{\Sigma},\nabla_{M_{\Sigma}})\to \rD\rR(M_{\Sigma^{\prime}},\nabla_{M_{\Sigma^{\prime}}}).\]
        By construction of $\rho_{\Sigma}(\bM)$ and $\rho_{\Sigma^{\prime}}(\bM)$, the following diagram
        \[\xymatrix@C=0.5cm{
           \rR\Gamma((\frakX_0/A)_{\Prism},\bM)\ar[rrrd]_{\rho_{\Sigma^{\prime}}(\bM)}\ar[rrr]^{\rho_{\Sigma}(\bM)}&&&\rD\rR(M_{\Sigma},\nabla_{M_{\Sigma}})\ar[d]^{\iota_{\Sigma,\Sigma^{\prime}}}\\
           &&&\rD\rR(M_{\Sigma^{\prime}},\nabla_{M_{\Sigma^{\prime}}})
        }\]
        is commutative. So $\iota_{\Sigma,\Sigma^{\prime}}$ is also a quasi-isomorphism of complexes of $A$-modules. 
        Using Proposition \ref{prop:cohomology comparison}, for any $\Sigma\in \frakS$, we obtain a quasi-isomorphism of complexes of $A$-modules
        \[\rR\Gamma(\frakX_{n,\et},\rD\rR(\calM,\nabla_{\calM}))\simeq \rD\rR(M_{\Sigma},\nabla_{M_{\Sigma}})\]
        by considering any $\delta\in \Sigma$ and the map $\iota_{\{\delta\},\Sigma}$. Therefore, $\Sigma$ induces a quasi-isomorphism of complexes of $A$-modules, which is still denoted by
        \[\rho_{\Sigma}(\bM):\rR\Gamma((\frakX_0/A)_{\Prism},\bM)\to\rR\Gamma(\frakX_{n,\et},\rD\rR(\calM,\nabla_{\calM})).\]
        Since the set $\calP(\frakS)$ of all non-empty finite subsets of $\frakS$ is cofiltered, by taking colimit, we then obtain a quasi-isomorphism of complexes of $A$-modules
        \[\rho(\bM):\colim_{\Sigma\in\calP(\frakS)}\rho_{\Sigma}(\bM):\rR\Gamma((\frakX_0/A)_{\Prism},\bM)\to\rR\Gamma(\frakX_{n,\et},\rD\rR(\calM,\nabla_{\calM})).\]
        By construction, this quasi-isomorphism $\rho(\bM)$ is functorial in $\bM$ and independent of the choice of $\delta$-structures on $R$ (but depends on the lifting $\varphi_{R_1}$ of $\varphi_{R_0}$).
    \end{proof}

\subsection{The proof of Theorem \ref{thm:main}(2)}\label{ssec:proof-II}
\begin{proof}[\textbf{Proof of Theorem \ref{thm:main}}(2)]
    Let $\bM\in \Vect((\frakX_0/A)_{\Prism},\calO_{\Prism,n})$ with associated $(\calM,\nabla_{\calM})\in\MIC^{\rm tqn}_p(\frakX_n)$ via the equivalence in Theorem \ref{thm:main}(1). It suffices to show for any $\frakV_0\in\frakX_{0,\et}$ with the induced lifting $\frakV_n\in \frakX_{n,\et}$, we have a canonical quasi-isomorphism in the derived category of $A$-modules
    \[\rR\Gamma((\frakV_0/A)_{\Prism},\bM)\simeq \rR\Gamma(\frakV_{n,\et},\rD\rR(\calM,\nabla_{\calM}))\]
    which depends only on the lifting $\rF_{\frakX_1}$ of $\rF_{\frakX_0}$.
    
    Keep notations as in \S\ref{ssec:proof-I} and for each $i\in I$, let $\frakU_{n,i}\subset\frakX_n$ be the induced lifting of $\frakU_{0,i}$. Put $\frakV_{n,i}:=\frakV_n\times_{\frakX_n}\frakU_{n,i}$. Then $\{\frakV_{n,i}\to\frakV_n\}_{i\in I}$ is a covering of $\frakV_n$ by affine schemes with charts. Let $\frakV_n^{\bullet}$ be the \v Cech nerve associated to the \'etale covering $\{\frakV_{n,i}\to\frakV_n\}_{i\in I}$ of $\frakV_n$. Then by Proposition \ref{prop:local compatibility of cohomology}, for each $\bullet\geq 0$, there is a canonical quasi-isomorphism in the derived category of $A$-modules
    \[\rR\Gamma((\frakV_0^{\bullet}/A)_{\Prism},\bM)\simeq \rR\Gamma(\frakV^{\bullet}_{n,\et},\rD\rR(\calM,\nabla_{\calM}))\]
    depending only on the lifting $\rF_{\frakX_1}$ of $\rF_{\frakX_0}$. By taking limits on both sides, we then obtain quasi-isomorphisms in the derived category of $A$-modules
    \[\rR\Gamma((\frakV_0/A)_{\Prism},\bM)\simeq\lim_{\Delta}\rR\Gamma((\frakV_0^{\bullet}/A)_{\Prism},\bM)\simeq \lim_{\Delta}\rR\Gamma(\frakV^{\bullet}_{n,\et},\rD\rR(\calM,\nabla_{\calM})) \simeq \rR\Gamma(\frakV_{n,\et},\rD\rR(\calM,\nabla_{\calM}))\]
    which depends only on $\rF_{\frakX_1}$ as desired. This completes the proof.
\end{proof}

\section{Remarks on trivialization of Hodge--Tate gerbe}\label{sec:remarks on gerbes}
  We remark that Bhatt--Lurie proved that when $\rF_{\frakX_0}:\frakX_0\to\frakX_0$ admits a lifting $\rF_{\frakX}:\frakX\to\frakX$ over $\phi_A:A\to A$, the gerbe $\pi_{\frakX_0}^{\rm HT}:\frakX_0^{\rm HT}\to\frakX_0$ splits \cite[Rem. 5.13(1)]{BL22b}. For furthur discussion, let us recall the proof in the case where $\frakX_0 = \Spf(R_0)$ is affine with a fixed chart (cf. \cite[The second paragraph of the proof of Prop. 5.12]{BL22b}):
  Let $R$ be a smooth lifting of $R_0$ over $A$ with a $\delta$-structure $\delta$. For any $R_0$-algebra $S$, 
  let $\overline{\rW(S)}$ be the derived quotient of $\rW(S)$ modulo $p$, i.e. the cofiber of the generalized invertible ideal $\rW(S)\xrightarrow{\times p}\rW(S)$. When $S$ is reduced, $\overline{\rW(S)}$ coincides with the classical quotient  
  \[\overline{\rW(S)} = \rW(S)/p\rW(S) = \rW(S)/\{(0,a_1^p,a_2^p,\dots)\in\rW(S)\mid a_1,a_2,\dots\in S\}\]
  as $\rW(S)$ is $p$-torison free. Then 
  we have
  \[\begin{split}
       &\frakX_0(S) = \Hom_{A_0\text{-algebra}}(R_0,S) = \Hom_{A\text{-algebra}}(R,S)
      \xrightarrow[\cong]{\iota_{\delta}} \Hom_{\delta\text{-}A\text{-algebra}}(R,\rW(S))\\
      \to & \Hom_{A\text{-algebra}}(R,\overline{\rW(S)})= \Hom_{A_0\text{-algebra}}(R_0,\overline{\rW(S)}) =  \frakX_0^{\rm HT}(S),
  \end{split}\]
  where $\iota_{\delta}$ is the only isomorphism which is non-canonical and depends on the choice of the $\delta$-structure on $R$.
  This induces a section
  \[s_{\delta}:\frakX_0\to \frakX_0^{\rm HT}\]
  of the Hodge--Tate structure morphism $\pi^{\rm HT}_{\frakX_0}:\frakX_0^{\rm HT}\to \frakX_0$ (when $\frakX_0 = \Spf(R_0)$) which depends on the choice of $\delta$. 
   We remark that for any $f\in \Hom_{A\text{-algebra}}(R,S)$ and any $x\in R$, the morphism of $\delta$-$A$-algebras $\iota_{\delta}(f)$ uniquely factors as
   \[R\xrightarrow{x\mapsto(x,\delta(x),\delta_2(x),\dots)}\rW(R)\xrightarrow{\rW(f)}\rW(S),\]
   where $\delta_n$ denotes the $n$-th Joyal's coordinate of Witt vectors (cf. \cite[Rem. 2.13]{BS22}).
  In general, suppose that $\rF_{\frakX_0}:\frakX_0\to\frakX_0$ admits a lifting $\rF_{\frakX}:\frakX\to\frakX$ over $A$. For any \'etale localization $\Spf(R_0)\to\frakX_0$, it admits a lifting $\Spf(R)\to \frakX$ over $A$ such that the Frobenius $\rF_{\frakX}$ provides a canonical choice of $\delta$-structure on $R$. So the sections $s_{\delta}$'s, for all affine \'etale localizations of $\frakX_0$ constructed above, glue and give a section $s$ of $\pi^{\rm HT}_{\frakX_0}$.

  Inspired by Theorem \ref{thm:main}, one may ask if $R$ admits two $\delta$-structures $\delta_1$ and $\delta_2$, does $s_{\delta_1}$ coincide with $s_{\delta_2}$ as long as Assumption \ref{assumpsion:key} is true? If this were true, then we would get a section of the gerbe $\pi_{\frakX_0}^{\rm HT}$ as long as $(\frakX_0,\rF_{\frakX_0})$ has a lifting over $A_1$.
  Although we are not able to answer this question here, we can give a refined version of the above result of Bhatt and Lurie.
  We will prove that if $\delta_1\equiv \delta_2\mod p^m$ for some $m\geq 1$ (equivalently, the corresponding Frobenius endomorphisms $\phi_i$ with respect to $\delta_i$ satisfy that $\phi_1\equiv \phi_2\mod p^{m+1}$), then $s_{\delta_1}$ and $s_{\delta_2}$ coincide in $\Hom_{A\text{-algebra}}(R,\overline{\rW_{m+1}(S)})$ for any $R_0$-algebra $S$ when $A_0$ is reduced (cf. Corollary \ref{cor:truncated split}).
  \begin{dfn}\label{dfn:good}
      Let $A$ be a ring and $\phi:A/p\to A/p$ be the absolute Frobenius map. An element $a$ in $A$ is called \emph{good}, if its reduction modulo $p$ belongs to $\Ima(\phi_A:A/p\to A/p)$. For any $a_1,a_2\in A$, we write $a_1\sim a_2$ if $a_1-a_2$ is good. This is an equivalence relation on $A$.
  \end{dfn}
  
  Note that for any two $\delta$-structures $\delta_1$ and $\delta_2$ on $R$, if $\delta_{1,n}(x)\sim \delta_{2,n}(x)$ for all $x\in R$ and each $0\leq n\leq m+1$, then we have $s_{\delta_1} = s_{\delta_2}$ in $\Hom_{A\text{-algebra}}(R,\overline{\rW_{m+1}(S)})$ for any $R_0$-algebra $S$.
  
  \begin{lem}\label{lem:good}
      Let $A$ be a ring and $a\in A$. The following are equivalent:
      \begin{enumerate}
          \item[(1)] The $a$ is good.

          \item[(2)] There are some $b,c\in A$ such that $a = b^p+pc$.

          \item[(3)] If furthermore $A$ is a $\delta$-ring with the induced Frobenius endomorphism $\phi$, then there are some $b,c\in A$ such that $a = \phi(b)+pc$.
      \end{enumerate}
  \end{lem}
  \begin{proof}
      Easy!
  \end{proof}
  \begin{lem}\label{lem:congruence}
      Let $A$ be a $\delta$-ring with the induced Frobenius endomorphism $\phi$. Then we have that
      \begin{enumerate}
          \item[(1)] if $a\in A$ is good, then so is $\delta(a)$,

          \item[(2)] for any $a,b,c\in A$, $\delta(a+pb+c^p)\sim \delta(a)-\sum_{i=1}^{p-1}\frac{\binom{p}{i}}{p}a^ic^{p(p-i)}$; in particular, if $a_1,a_2\in A$ satisfying $a_1\equiv a_2\mod p$, then $\delta(a_1)\sim\delta(a_2)$.
      \end{enumerate}
  \end{lem}
  \begin{proof}
      For Item (1): As $a$ is good, we can write $a = \phi(b)+pc$ for some $b,c\in A$. Then we have
      \[\begin{split}
      \delta(a) &= \delta(\phi(b)+pc) \\
      &= \phi(\delta(b))+\delta(pc) - \sum_{i=1}^{p-1}\binom{p}{i}p^{i-1}c^i\phi(b)^{p-i}\\
      &= \phi(\delta(b))+ \delta(p)\phi(c) + p^p\delta(c) - \sum_{i=1}^{p-1}\binom{p}{i}p^{i-1}c^i\phi(b)^{p-i}\\
      &=\phi(\delta(b)+c)-p^{p-1}\phi(c)+ p^p\delta(c) - \sum_{i=1}^{p-1}\binom{p}{i}p^{i-1}c^i\phi(b)^{p-i}.
      \end{split}\]
      Thus, $\delta(a)$ is good.

      For Item (2):  We have
      \[\begin{split}
          \delta(a+pb+c^p) = & \delta(a) + \delta(pb+c^p)-\sum_{i=1}^{p-1}\frac{\binom{p}{i}}{p}a^i(pb+c^p)^{p-i}\\
          \sim & \delta(a) + \delta(pb+c^p)-\sum_{i=1}^{p-1}\frac{\binom{p}{i}}{p}a^ic^{p(p-i)}\\
          \sim & \delta(a)-\sum_{i=1}^{p-1}\frac{\binom{p}{i}}{p}a^ic^{p(p-i)},
      \end{split}\]
      where the last equivalence follows from Item (1).
  \end{proof}

  Now, we are able to prove the following result:
  
  \begin{prop}\label{prop:truncated split}
      Fix an $m\geq 1$.
      Let $\delta_1,\delta_2$ be two $\delta$-structures on $R$ so that $\delta_1\equiv \delta_2\mod p^m$. Then for any reduced $R_0$-algebra $S$, the $s_{\delta_1}$ coincides with $s_{\delta_2}$ as elements in 
      \[\Hom_{A\text{-algebra}}(R,\overline{\rW_{m+1}(S)}).\]
  \end{prop}
  \begin{proof}
      For any $\delta$-structure $\delta$ on $R$, denote by $\iota_{\delta}$ the associated non-canonical isomorphism between the following two functors
      \[S\mapsto \Hom_{A\text{-algebra}}(R,S)\]
      and 
      \[S\mapsto \Hom_{A\text{-}\delta\text{-algebra}}(R,\rW(S))\]
      from the category of $A$-algebras to the category of sets. In particular, we get $\iota_{\delta_1}$ and $\iota_{\delta_2}$ for the given $\delta$-structures $\delta_1$ and $\delta_2$ on $R$, respectively. We also denote by 
      \[\iota_i:R\to \rW(R)\]
      the morphism induced by $\id_R$ via the isomorphism
      \[\iota_{\delta_i}(R):\Hom_{A\text{-algebra}}(R,R) \xrightarrow{\cong} \Hom_{A\text{-}\delta\text{-algebra}}(R,\rW(R)).\]

      Let $\pr:R\to R_0$ be the natural projection. Then for any morphism of $A_0$-algebras $f:R_0\to S$, the morphism
      \[\iota_{\delta_i}(f)\in\Hom_{A\text{-algebra}}(R,\rW(S))\]
      factors as
      \[R\xrightarrow{\iota_i}\rW(R)\xrightarrow{\rW(\pr)}\rW(R_0)\xrightarrow{\rW(f)}\rW(S).\]
      Now, suppose that $S$ is reduced. Consider the following commutative diagram
      \[\begin{tikzcd}
         R \arrow[r, "\iota_i"] & \rW(R) \arrow[r,"\rW(\pr)"] \arrow[d,"{\rm Res}_{m+1}"]   & \rW(R_0) \arrow[r, "\rW(f)"] \arrow[d,"{\rm Res}_{m+1}"] & \rW(S) \arrow[d,"{\rm Res}_{m+1}"]\\
          & \rW_{m+1}(R) \arrow[r, "\rW_{m+1}(\pr)"] & \rW_{m+1}(R_0) \arrow[r, "\rW_{m+1}(f)"] \arrow[d,"\pi_p"]  & \rW_{m+1}(S) \arrow[d,"\pi_p"]  \\
          & & \overline{\rW_{m+1}(R_0)} \arrow[r, "\overline{\rW_{m+1}(f)}"] & \overline{\rW_{m+1}(S)},
\end{tikzcd}\]
 where ${\rm Res}_{m+1}:\rW(-)\to \rW_{m+1}(-)$ is the restriction of Witt functors and $\pi_p:\rW_{m+1}(-)\to \overline{\rW_{m+1}(-)}$ is the reduction modulo $p$. Then we have 
 \[s_{\delta_i} = \pi_p\circ{\rm Res}_{m+1}\circ\rW(f)\circ\rW(\pr)\circ\iota_i = \overline{\rW_{m+1}(f)}\circ\pi_p\circ{\rm Res}_{m+1}\circ\rW(\pr)\circ\iota_i.\]
 Thus, in order to show $s_{\delta_1} = s_{\delta_2}$, it suffices to show that
 \[\pi_p\circ\rW_{m+1}(\pr)\circ{\rm Res}_{m+1}\circ\iota_1 = \pi_p\circ\rW_{m+1}(\pr)\circ{\rm Res}_{m+1}\circ\iota_2.\] 
 Note that for any $i\in\{1,2\}$ and any $x\in R$,
      \[{\rm Res}_{m+1}(\iota_i(x)) = (x,\delta_i(x),\delta_i^2(x),\dots,\delta_i^{m+1}(x))\in \rW_{m+1}(R).\]
 By Lemma \ref{lem:congrunce of Joyal} below, we see that for any $0\leq l\leq m+1$ and any $x\in R$, we have $\delta_1^l(x)\sim \delta_2^l(x)$, yielding that
 \[(\pi_p\circ\rW_{m+1}(\pr)\circ{\rm Res}_{m+1}\circ\iota_1)(x) = (\pi_p\circ\rW_{m+1}(\pr)\circ{\rm Res}_{m+1}\circ\iota_2)(x)\]
      as desired. This completes the proof.
  \end{proof}

  \begin{lem}\label{lem:high congruence}
      Fix an $m\geq 1$.
      Let $\delta_1$ and $\delta_2$ be two $\delta$-structures on $R$ satisfying $\delta_1\equiv\delta_2\mod p^m$. For any $n\geq 1$, let $\delta_i^n$ be the self-composite of $n$ copies of $\delta_i$. Then for any $r\in R$, if $\delta_1(r) =  \delta_2(r)+p^mx$ for some $x\in R$, then we have that for any $0\leq l\leq m$,
          \[\delta^{1+l}_1(r)\equiv \delta_2^{1+l}(r)+p^{m-l}x^{p^l}\mod p^{m+1-l}.\]
  \end{lem}
  \begin{proof}
      We do induction on $l$. Assume we have shown the result for some $0\leq l<m$. Write
      \[\delta_1^{l}(r) = \delta_2^l(r)+p^{m+1-l}x^{p^{l-1}}+p^{m+2-l}y\]
      for some $y\in R$. Then we have
      \[\begin{split}
          \delta_1^{1+l}(r)-\delta_2^{1+l}(r) = &\delta_1(\delta_2^l(r)+p^{m+1-l}x^{p^{l-1}}+p^{m+2-l}y) -\delta_2^{l+1}(r)\\
           \equiv &\delta_2(\delta_2^l(r)+p^{m+1-l}x^{p^{l-1}}+p^{m+2-l}y)-\delta_2^{l+1}(r)\mod p^m \qquad \because \delta_1\equiv \delta_2\mod p^{m+1-l}\\
           =&\delta_2(p^{m+1-l}(x^{p^{l-1}}+py))-\sum_{i=1}^{p-1}\frac{\binom{p}{i}}{p}\delta_2^l(r)^{i}(p^{m+1-l}(x^{p^{l-1}}+py))^{p-i}\\
          \equiv &\frac{p^{m+1-l}-p^{p(m+1-l)}}{p}(x^{p^{l-1}}+py)^p\mod p^{m+1-l}\\
          \equiv & p^{m-l}x^{p^l}\mod p^{m+1-l}.
      \end{split}\]
      So we can conclude by induction.
  \end{proof}
  
    \begin{lem}\label{lem:congrunce of Joyal}
      Fix an $m\geq 1$.
      Let $\delta_1$ and $\delta_2$ be two $\delta$-structures on $R$ satisfying $\delta_1\equiv\delta_2\mod p^m$. For any $r\in R$, if $\delta_1(r) =  \delta_2(r)+p^mx$ for some $x\in R$, then we have that for any $0\leq l\leq m-1$,
          \[\delta_{1,1+l}(r)\equiv \delta_{2,1+l}(r)\mod p^{m-l},\]
          and 
          \[\delta_{1,m}(r)\equiv \delta_{2,m}(r)+x^{p^m}\mod p.\]
  \end{lem}
  \begin{proof}
      Recall that for any $\delta$-structure on $R$, if we denote by $\delta_n$ the $n$-th Joyal's coordinate of Witt vectors, then by \cite[Rem. 2.14]{BS22}, there exists a monic polynomial $f_n(X) = X^n+c_{n-1}X^{n-1}+\cdots+c_0\in\bZ[X]$ such that
      \begin{equation*}\label{equ:Joyal vs delta}
          \delta_n = f_n(\delta) =\delta^n+c_{n-1}\delta^{n-1}+\cdots+c_0.
      \end{equation*}
      
      Now, for any $0\leq l\leq m$, we have
      \[\delta_{i,1+l}(r) = \delta_{i}^{l+1}(r)+\sum_{j=0}^{l}c_j\delta_i^j(r).\]
      Then the result follows from Lemma \ref{lem:high congruence} immediately.
  \end{proof}
  
  \begin{cor}\label{cor:truncated split}
      Assume $A_0$ is reduced. 
      Let $\delta_1,\delta_2$ be two $\delta$-structures on $R$ so that $\delta_1\equiv \delta_2\mod p^m$ for some $m\geq 1$. Then for any $R_0$-algebra $S$, the section $s_{\delta_1}$ coincides with $s_{\delta_2}$ as elements in 
      \[\Hom_{A\text{-algebra}}(R,\overline{\rW_{m+1}(S)}).\]
  \end{cor}
  \begin{proof}
      As in the proof of Proposition \ref{prop:truncated split}, we are reduced to the case for $S = R_0$. As $A_0$ is reduced, so is $R_0$, because it is a smooth algebra over $A_0$. Thus, we can conclude by Proposition \ref{prop:truncated split}.
  \end{proof}

\section*{Data Availability Statement}
 The paper contains no extra data from other places.


\begin{thebibliography}{99}

 \bibitem[BL22]{BL22b} Bhargav Bhatt and Jacob Lurie. {\it The prismatisation of $p$-adic formal schemes}. arXiv:2201.06124, 2022.


 \bibitem[BO78]{Ber-Ogus} Pierre Berthelot and Arthur Ogus. {\it Notes on crystalline cohomology}. Princeton University Press, Princeton, NJ; University of Tokyo Press, Tokyo, 1978.
 
\bibitem[BS22]{BS22} Bhargav Bhatt and Peter Scholze. {\it Prisms and prismatic cohomology}. Ann. of Math. (2),
196(3):1135–1275, 2022.

\bibitem[BS23]{BS23} Bhargav Bhatt and Peter Scholze. {\it Prismatic F-crystals and crystalline Galois representations}. Camb. J. Math., 11(2):507–562, 2023. 36

\bibitem[DI87]{DI} Pierre Deligne and Luc Illusie. {\it Rel\'evements modulo p et d\'ecomposition du complexe de
de Rham}. Inventiones mathematicae, 89:247–270, 1987.

\bibitem[GLSQ10]{GLSQ10} Michel Gros, Bernard Le Stum, and Adolfo Quir\'os. {\it A Simpson correspondence in positive
characteristic}. Publ. Res. Inst. Math. Sci., 46(1):1–35, 2010.

\bibitem[LSZ19]{LSZ19} Guitang Lan, Mao Sheng, and Kang Zuo. {\it Semistable Higgs bundles, periodic Higgs
bundles and representations of algebraic fundamental groups}. J. Eur. Math. Soc. (JEMS),21(10):3053–3112, 2019.

\bibitem[Ogu22]{Ogu22} Arthur Ogus. {\it Crystalline prisms: Reflections on the present and the past}.
arxiv:2204.06621v3, 2022.

\bibitem[OV07]{OV} Arthur Ogus and Vadim Vologodsky. {\it Nonabelian Hodge theory in characteristic p}. Publications
math\'ematiques IHES, 106(1):1–138, 2007.

\bibitem[Pet23]{Pet23} Alexander Petrov. {\it Non-decomposability of the de rham complex and non-semisimplicity of the sen operator}. arXiv:2302.11389, 2023.

\bibitem[Shi15]{Shiho} Atsushi Shiho. {\it Notes on generalizations of local Ogus-Vologodsky correspondence}. J. Math. Sci. Univ. Tokyo, 22(3):793–875, 2015.

\bibitem[Tia23]{Tia23} Yichao Tian. {\it Finiteness and duality for the cohomology of prismatic crystals}. J. Reine
Angew. Math., 800:217–257, 2023.

\bibitem[Xu19]{Xu19} Daxin Xu. {\it Lifting the Cartier transform of Ogus-Vologodsky modulo $p^n$}. M´em. Soc.
Math. Fr. (N.S.), (163):133, 2019.

\bibitem[Yu24a]{Yu} Jiahong Yu. {\it Prismatic crystals for schemes in characteristic $p$}. arXiv:2407.15381, 2024.


\bibitem[Yu24b]{Yu2} Jiahong Yu. {\it Notes on Hodge--Tate stacks}. arXiv:2410.06630, 2024.

\end{thebibliography}
\end{document}